\documentclass[11pt]{article}
\usepackage{amsfonts}
\usepackage{mathrsfs}
\usepackage{bbm}
\usepackage{amsfonts}
\usepackage{amssymb,amsmath,amsthm,graphicx}
\usepackage{float}
\usepackage[colorlinks, linkcolor=red]{hyperref}
\usepackage{color, xcolor}
\usepackage{cite}
\usepackage[paper=a4paper, left=1.6cm, right=1.6cm, top=1.8cm, bottom=1.6cm, headheight=5.5pt, footskip=0.8cm, footnotesep=0.8cm, centering, includefoot]{geometry}
\hypersetup{urlcolor=red, citecolor=red}

\DeclareMathOperator{\divv}{div}
\DeclareMathOperator{\curl}{curl}

\allowdisplaybreaks

\begin{document}
\title{Global classical solution for three-dimensional compressible isentropic magneto-micropolar fluid equations with Coulomb force
and slip boundary conditions in bounded domains
\thanks{
This research was partially supported by National Natural Science Foundation of China (No. 12371227), Scientific Research Foundation of Jilin Provincial Education Department (No. JJKH20210873KJ), Postdoctoral Science Foundation of China (No. 2021M691219), and Exceptional Young Talents Project of Chongqing Talent (No. cstc2021ycjh-bgzxm0153).}
}
\author{Yang Liu$\,^{\rm 1}\,$,\ Xin Zhong$\,^{\rm 2}\,${\thanks{Corresponding author. E-mail addresses: liuyang0405@ccsfu.edu.cn (Y. Liu), xzhong1014@amss.ac.cn (X. Zhong).}}
\date{}\\
\footnotesize $^{\rm 1}\,$
College of Mathematics, Changchun Normal
University, Changchun 130032, P. R. China\\
\footnotesize $^{\rm 2}\,$ School of Mathematics and Statistics, Southwest University, Chongqing 400715, P. R. China} \maketitle
\newtheorem{theorem}{Theorem}[section]
\newtheorem{definition}{Definition}[section]
\newtheorem{lemma}{Lemma}[section]
\newtheorem{proposition}{Proposition}[section]
\newtheorem{corollary}{Corollary}[section]
\newtheorem{remark}{Remark}[section]
\renewcommand{\theequation}{\thesection.\arabic{equation}}
\catcode`@=11 \@addtoreset{equation}{section} \catcode`@=12
\maketitle{}

\begin{abstract}
We study an initial-boundary value problem of three-dimensional (3D) compressible isentropic magneto-micropolar fluid equations with Coulomb force and slip boundary conditions in a bounded simply connected domain, whose boundary has a finite number of two-dimensional connected components. We derive the global existence and uniqueness of classical solutions provided that the initial total energy is suitably small. Our result generalizes the Cauchy problems of compressible Navier-Stokes equations with Coulomb force (J. Differential Equations 269: 8468--8508, 2020) and compressible MHD equations (SIAM J. Math. Anal. 45: 1356--1387, 2013) to the case of bounded domains although tackling many surface integrals caused by the slip boundary condition are complex. The main ingredient of this paper is to overcome the strong nonlinearity caused by Coulomb force, magnetic field, and rotation effect of micro-particles by applying piecewise-estimate method and delicate analysis based on the effective viscous fluxes involving velocity and micro-rotational velocity.
\end{abstract}



\tableofcontents

%%%%%%%%%%%%%%%%%%%%%%%%%%%%%%%%%%%%%%%%%%%%%%%%%%%%%%%%%%%%%%%%%%%%%%%%%%%%%%%%%%%%%%%%%%%%%%%%%%

\section{Introduction}
Let $\Omega\subset \mathbb{R}^3$ be a bounded domain, we consider the compressible isentropic magneto-micropolar fluid equations with Coulomb force (see \cite{JTZ19}) in $\Omega\times(0,T)$:
\begin{align}\label{a1}
\begin{cases}
 \rho_t+\divv(\rho u)=0,\\
\rho u_t+\rho u\cdot\nabla u+\nabla P
=(\lambda+\mu-\mu_r)\nabla \divv u+(\mu+\mu_r)\Delta u+2\mu_r\curl w
+(\nabla \times b)\times b+\rho\nabla\Phi,\\
\rho w_t+\rho u\cdot\nabla w+4\mu_rw
=(c_0+c_d-c_a)\nabla\divv w+(c_a+c_d)\Delta w+2\mu_r\curl u,\\
b_t-\nabla\times(u\times b)=-\nabla\times(\nabla\times b),\\
\Delta \Phi=\rho-\tilde{\rho},\\
\divv b=0.
\end{cases}
\end{align}
The unknown functions $\rho$, $u=(u^1, u^2, u^3)$, $w=(w^1, w^2, w^3)$, $b=(b^1, b^2, b^3)$, and $\Phi$ are the density, velocity, micro-rotational velocity, magnetic field, and the electrostatic potential, respectively. $P(\rho)=a\rho^\gamma\ (a>0, \gamma>1)$ is the pressure. $\mu$ and $\lambda$ are the coefficients of viscosity, while $\mu_r$, $c_0$, $c_d$, and $c_a$ are the coefficients of micro-viscosity, and they satisfy
\begin{align*}
\mu>0,\ 3\lambda+2\mu\ge 0, \ \mu_r>0, \
c_d>0,\ 3c_0+2c_d\ge 0,\ c_a>0.
\end{align*}
The given smooth function $\tilde{\rho}(x)$ is the non-flat doping profile and satisfies
\begin{align*}
0<\underline{\rho}\le \tilde{\rho}(x)\le \bar{\rho}
\end{align*}
for some positive constants $\underline{\rho}$ and $\bar{\rho}$.

The system \eqref{a1} is supplemented with
the initial condition
\begin{align}
(\rho, \rho u, \rho w, b_0)(x, 0)=(\rho_0, \rho_0u_0, \rho_0w_0, b_0)(x), \quad x\in\Omega,
\end{align}
and slip boundary conditions
\begin{align}\label{a6}
u\cdot n=0,~ \curl u\times n=0, ~ w\cdot n=0,~\curl w\times n=0, ~
b\cdot n=0, ~\curl b\times n=0, ~\nabla\Phi\cdot n=0, &~{\rm on}~ \partial\Omega,
\end{align}
where $n=(n^1, n^2, n^3)$ is the unit outward normal vector to $\partial\Omega$.

It should be noted that \eqref{a1} reduces to the compressible micropolar fluid equations with Coulomb force in the absence of the magnetic field. The theory of micropolar fluids introduced by Eringen in the 1960s (see \cite{E1964,E1966}) is a significant step towards the generalization of the classical Navier-Stokes model. The micropolar fluids describe a class of microstructure related fluids, for example, liquid crystals,
blood, or polymeric suspensions. The model has many potential applications in both mathematics and engineering (see \cite{E1999,L1999}). Due to the profound physical background and important mathematical significance, the compressible micropolar fluid equations have been extensively studied. Based on the spectral analysis of the linearized system, Liu and Zhang \cite{LZ2016} studied the optimal time decay rate problem for solutions to the constant equilibrium state. Later on, Tong-Pan-Tan \cite{TPT2021} improved the result of \cite{LZ2016} in the sense that the solution to the 3D compressible micropolar fluid equations has the exactly same decay rates as the linearized problem. By constructing global weak solutions as limits of smooth solutions, Chen-Xu-Zhang \cite{CXZ2015} proved global existence of weak solutions to the three-dimensional compressible micropolar fluid system with initial data which may be discontinuous and may contain vacuum states. Wu and Wang \cite{WW2018} studied pointwise estimates of the compressible micropolar fluid system in three dimensions. However, one has to include electromagnetic effects when considering some flows such as flows of electroconducting fluids in a magnetic field or polarized fluids in an electric field. Thus, coupling the conservation laws for micropolar fluids with the equations of electrodynamics, one comes to magnetohydrodynamics of micropolar fluids (see \cite{AS1974}). As a couple system, the compressible magneto-micropolar fluid equations contain much richer structures than the classical compressible fluid equations. It is not merely a combination of the conservation laws for micropolar fluids with the equations of electrodynamics but an interactive system. Their distinctive features make analytical studies a great challenge but offer new opportunities. Next, we only briefly recall some related results about the well-posedness theory. In the Lions-Feireisl framework for compressible Navier-Stokes equations (see \cite{L1998,F2004}), Amirat and Hamdache \cite{AH09} proved the global existence of weak solutions in 3D bounded domains with initial vacuum. Wei-Guo-Li \cite{WGL2017} obtained global existence and optimal convergence rates of smooth solutions in $\mathbb{R}^3$ if the initial perturbation is small in $H^3$-norm. Based on the spectral analysis and energy estimates, Tang and Tan \cite{TT2019} considered the 3D Cauchy problem with initial data close to some constant steady state, and they proved the stabilization of the solution in the sense that the solution converges to its constant equilibrium state, as well as that the convergence rate is optimal.

When the rotation effect of micro-particles $w$ is ignored, the system \eqref{a1} reduces to the compressible isentropic magnetohydrodynamic (MHD) equations with Coulomb force. Tan and Wang \cite{TW09} studied the initial boundary value problem and established the existence and large-time behavior of global weak solutions via the weak convergence method.
Very recently, with the help of the techniques of anti-symmetric matrix and an induction argument on the order of the space derivatives of solutions in energy estimates, Feng-Li-Wang \cite{FLW21} proved global smooth solutions near the non-constant equilibrium solutions, and they also showed the asymptotic behavior of solutions when the time goes to infinity.

Recently, Jia-Tan-Zhou \cite{JTZ19} derived the global existence and large time behaviour of solutions near a constant state for the system \eqref{a1} in $\mathbb{R}^3$ under the assumption that the $H^3$-norm of the initial data is small. It should be noted that all the results mentioned above were only concerned with the whole space or with non-slip boundary conditions. It is rather complicated to investigate global well-posedness and dynamical behaviors of the compressible fluid equations with slip boundary condition due to the compatibility issues of the nonlinear terms with the slip boundary condition. In particular, whether the classical solution in bounded domains with density containing vacuum initially to the system \eqref{a1} exists globally in time remains open. Very recently, by introducing some new techniques in dealing with the boundary estimates, Cai and Li \cite{CJ21} proved the global well-posedness of classical solutions for the compressible isentropic Navier-Stokes equations with slip boundary condition and small initial energy, which extended the famous work on the Cauchy problem \cite{HLX12} to the case of bounded domains. Later on, Chen-Huang-Shi \cite{CHS21} generalized the result \cite{CJ21} to the compressible isentropic MHD equations (that is, \eqref{a1} with no rotation effect of micro-particles and Coulomb force). The other interesting studies on the global well-posedness for multi-dimensional compressible isentropic Navier-Stokes equations and MHD equations with vacuum can be referred to
\cite{LXZ13,HW10,SH12,HHPZ,LS19,FL20,LSX16,LX19,JZ01,FNP01,D95} and references therein.
Motivated by \cite{CHS21,CJ21}, the main purpose of the present paper is to investigate the global existence of classical solutions with large oscillations and vacuum to the problem \eqref{a1}--\eqref{a6} subject to large and non-flat doping profiles, when the initial energy is suitably small. In comparison with \cite{CJ21,CHS21}, there are two main difficulties on the analysis of this paper. One is that the large and non-flat doping profile will significantly affect the dynamic motion of compressible flows and cause some serious difficulties in the energy estimates, thus the crucial techniques of proofs in \cite{CJ21,CHS21} cannot be adapted directly. To overcome this obstacle, we need to establish the uniform \textit{a priori} estimates of the local smooth solutions with respect to the time via the piecewise-estimate method based on \cite{YZ17}. The other difficulty is that the system \eqref{a1} has the strong nonlinearity caused by the rotation effect of micro-particles which makes the \textit{a priori} estimates more complicated and tedious. We will obtain the desired estimates based on delicate analysis of the effective viscous flux (see \eqref{2.7} for the definition).

We consider the steady-state solutions of \eqref{a1} with zero velocity, micro-rotational velocity, and zero magnetic field. Let $(\rho_s, u_s, w_s, b_s, \Phi_s)$ be such a solution of variable $x$ with $u_s=0$, $w_s=0$, and $b_s=0$.
Then, one has
\begin{align}\label{a7}
\begin{cases}
\nabla P(\rho_s)=\rho_s\nabla\Phi_s\quad &{\rm in}~\Omega,\\
\Delta\Phi_s=\rho_s-\tilde{\rho}\quad &{\rm in}~\Omega,\\
\nabla\Phi_s\cdot n=0\quad &{\rm on}~\partial\Omega.
\end{cases}
\end{align}
We denote the initial total energy of \eqref{a1} by
\begin{align}
E_0\triangleq\int\Big(\frac12\rho_0|u_0|^2+\frac12\rho_0|w_0|^2+\frac12|b_0|^2+\frac12|\nabla(\Phi_0-\Phi_s)|^2+G(\rho_0)\Big)dx,
\end{align}
where
\begin{align}
G(\rho)\triangleq\int_{\rho_s}^\rho\int_{\rho_s}^\xi\frac{P'(r)}{r}dr d\xi=\rho\int_{\rho_s}^\rho\frac{P(\xi)-P(\rho_s)}{\xi^2}d\xi.
\end{align}
Moreover, we write
\begin{align*}
H^2_{\omega} \triangleq\left\{v\in H^{2}(\Omega): v\cdot n=0\ \text{and}\ \curl v\times n=0\ \text{on}\ \partial\Omega\right\}.
\end{align*}

Now we state our main result concerning the global classical solutions to the problem \eqref{a1}--\eqref{a6}.
\begin{theorem}\label{thm1}
Let $\Omega$ be a bounded simply connected smooth domain in $\mathbb{R}^3$ and its boundary $\partial\Omega$ has a finite number
of two-dimensional connected components. Let $(\rho_s, \Phi_s)$ be the stationary solution of \eqref{a7}.
For some given positive constants $M_1$, $M_2$, $M_3$ (not necessarily small) and $\hat{\rho}\ge \bar{\rho}+1$,
suppose that the initial data $(\rho_0, u_0, w_0, b_0)$ satisfying, for $q\in (3, 6)$,
\begin{align}\label{a10}
\begin{cases}
\underline{\rho}\le\rho_0\le \hat{\rho}, \ (\rho_0, P(\rho_0))\in W^{2, q},\ (u_0, w_0, b_0)\in H^2_{\omega}, \\
\|\nabla u_0\|_{L^2}^2\le M_1, \ \|\nabla w_0\|_{L^2}^2\le M_2,\ \|\nabla b_0\|_{L^2}^2\le M_3,\\
\end{cases}
\end{align}
and the compatibility conditions
\begin{align}\label{a11}
\int_{\Omega}(\rho_0-\tilde{\rho})dx=0,
\end{align}
and
\begin{align}\label{a12}
\begin{cases}
-(\mu+\mu_r)\Delta u_0-(\lambda+\mu-\mu_r)\nabla\divv u_0+\nabla P(\rho_0)-(\nabla \times b_0)\times b_0-2\mu_r\curl w_0=\sqrt{\rho_0}g_1,\\
-(c_a+c_d)\Delta w_0-(c_0+c_d-c_a)\nabla\divv w_0+4\mu_rw_0
-2\mu_r\curl u_0=\sqrt{\rho_0}g_2,
\end{cases}
\end{align}
with $g_1, g_2\in L^2$.
There exists a positive constant $\varepsilon$ depending only on $\mu$, $\lambda$, $u_r$, $c_0$, $c_a$, $c_d$,
$\gamma$, $a$, $\hat{\rho}$, $\Omega$, $M_1$, $M_2$, $M_3$, and $\tilde{\rho}$ such that if
\begin{align}
E_0\le \varepsilon,
\end{align}
the problem \eqref{a1}--\eqref{a6} has a unique global classical solution $(\rho, u, w, b, \Phi)$ in $\Omega\times(0, \infty)$
satisfying, for any $0<\tau<T<\infty$,
\begin{align}
0< \rho(x, t)\le 2\hat{\rho}, \quad (x, t)\in \Omega\times(0, \infty),
\end{align}
and
\begin{align}
\begin{cases}
(\rho, P)\in C([0, T]; W^{2, q}),\\
(\nabla u,\nabla w)\in C([0, T]; H^1)\cap L^\infty(\tau, T; W^{2, q}),\\
b\in C([0, T]; H^2)\cap L^\infty(\tau, T; H^4),\\
(u_t,w_t)\in L^\infty(\tau, T; H^2)\cap H^1(\tau, T; H^1),\\
b_t\in C([0, T];L^2)\cap H^1(\tau, T; H^1)\cap L^\infty(\tau, T; H^2),\\
\nabla\Phi\in C([0, T]; W^{2, q})\cap L^\infty(0, T; W^{3, q}),\\
\nabla\Phi_t\in C([0, T]; L^2)\cap L^\infty(0, T; H^2).
\end{cases}
\end{align}
\end{theorem}

Several remarks are in order.

\begin{remark}
The compatibility condition \eqref{a11} is necessary for solvability of the Poisson equation with Neumann boundary. The solution obtained in Theorem \ref{thm1} becomes a classical one for positive time. Although it has small energy, yet its oscillations could be arbitrarily large.
%In particular, initial vacuum state is allowed.
\end{remark}

\begin{remark}
Our Theorem \ref{thm1} generalizes the Cauchy problem \cite{JTZ19} to the case of bounded domains. However, this is a non-trivial generalization
since we need to deal with many surface integrals caused by the slip boundary condition \eqref{a6}.
\end{remark}

\begin{remark}
We emphasize that the positive lower bound of the initial density is a technical assumption. This condition is only used to estimate the boundary integral term $\int_{\partial\Omega}\eta(t)\dot{u}\cdot(\dot{w}\times n)dS$ (see \eqref{t3.36}), which is important to derive the estimates \eqref{3.75} and \eqref{3.76} (see Lemma \ref{l36}), and in turn plays a crucial role in the treatment of the upper bound of the density. In other words, if the rotation effect of micro-particles is ignored, the initial density can allow to vanish. Thus, the micro-particles acts some significant roles on the global well-posedness of solutions.
\end{remark}

\begin{remark}
When there is no rotation effect of micro-particles, the system \eqref{a1} turns to be compressible MHD equations with Coulomb force, and Theorem \ref{thm1} also holds in this case. Thus we extend the global classical solutions of \cite{FLW21} to the initial boundary value problem.
\end{remark}

%\begin{remark}
%It should be noted that \eqref{a1} becomes to be compressible MHD equations if we omit rotation effect of micro-particles and Coulomb force, and Theorem \ref{thm1} is the same as the global existence result of classical solutions in \cite{CHS21}. Hence we extend partial result of \cite{CHS21} to compressible magneto-micropolar fluid equations with Coulomb force.
%\end{remark}

\begin{remark}
It should be pointed out that the exponential decay rates of the solution can be established when we consider \eqref{a1} with no Coulomb force. This can be proved by similar arguments as those in \cite{CJ21}. This means that the Coulomb force acts as some significant roles on the large time behavior of the solutions.
\end{remark}

%\begin{remark}
%Combining the techniques used in this paper with \cite{CHS212}, we may show global strong solution to the Vaigant-Kazhikhov model of \eqref{a1} with slip boundary condition and large initial data. We leave this problem to the interested readers.
%\end{remark}

Let us sketch the strategy for the proof of Theorem \ref{thm1}. As mentioned in many related works (see, for example, \cite{HLX12,LXZ13,CHS21}), the key issue is to derive the uniform-in-time lower-order estimates and uniform upper bound of the density as well as the time-dependent higher-order estimates of $(\rho, u, w, b, \Phi)$.
To deal with many surface integrals (see \eqref{z3.21}, \eqref{3.24}, \eqref{z3.19}, \eqref{z3.22}, and \eqref{z3.24} for example) caused by the slip boundary condition \eqref{a6}, the observation $(v\cdot\nabla v\cdot n)|_{\partial\Omega}=-(v\cdot\nabla n\cdot v)|_{\partial\Omega}$ for any smooth vector field $v$ satisfying $(v\cdot n)|_{\partial\Omega}=0$ is very important. Moreover, the trace theorem (see Lemma \ref{l23}) and Poincar{\'e}'s inequality (see Lemma \ref{l21}) as well as $L^p$-estimates based on  the effective viscous flux (see Lemma \ref{l28}) play crucial roles. Next, we borrow some ideas developed in \cite{LXZ20,YZ17,XZZ22} to derive time-spatial estimates of material derivatives and gradients of velocity and micro-rotational velocity (see Lemmas \ref{l35} and \ref{l36}). Meanwhile, the strong nonlinearity and coupled terms due to the rotation effect of micro-particles can be settled by some careful analysis of the solutions based on a full use of the structure of \eqref{a1}. Then it allows us to obtain the uniform upper bound of the density provided that the initial energy is properly small (see Lemma \ref{l38}).
Next, we need to bound time-dependent higher-order estimates of the solutions. The key step is to derive the $L^q$-estimate of the gradient of density, which strongly relies on the bound for the $L^1(0,T;L^\infty(\Omega))$-norm of
the gradient of velocity and will be achieved by solving a logarithm Gronwall's inequality based on a Beale-Kato-Majda type inequality (see Lemma \ref{l211}) and the \textit{a priori} estimates we have just derived. It should be noted here that we do not require smallness of the gradient of the initial density, which prevents the appearance of vacuum \cite{JTZ19}.

The rest of the paper is arranged as follows. In Section \ref{sec2}, we collect some known facts and give crucial $L^p$-estimates involving the effective viscous flux and the vorticity. In Subsection \ref{sec3.1}, we make some \textit{a priori} assumptions and show the uniformly \textit{a priori} estimates of local smooth solutions independent of the time, while the energy estimates for the higher order derivatives are obtained in Subsection \ref{sec3.2}. Finally, we give the proof of Theorem \ref{thm1} in Section \ref{sec4}.

\section{Preliminaries}\label{sec2}
In this section, we recall some known facts and inequalities which will be used later.

First, similar to the proof of \cite{GS06}, we have the following existence and uniqueness of solutions to \eqref{a7}.
\begin{lemma}
Assume that the smooth function $\tilde{\rho}(x)$ satisfies $0<\underline{\rho}\le \tilde{\rho}(x)\le \bar{\rho}$.
Then the problem \eqref{a7} has a unique classical solution $(\rho_s, \Phi_s)$ satisfying
\begin{gather}
\underline{\rho}\le \rho_s\le \bar{\rho},\\
\|\nabla \rho_s\|_{H^3}+\|\nabla\Phi_s\|_{H^4}\le C, \label{2.2}
\end{gather}
where $C$ depends only on $a$, $\gamma$, $\tilde{\rho}(x)$, and $\Omega$.
\end{lemma}

Next, by arguments as those in \cite{TG16, TZ10} and contraction mapping principle, we can obtain the following local existence theorem of classical solutions of \eqref{a1}--\eqref{a6}. Here we omit the details for simplicity.
\begin{lemma}\label{l22}
Assume that the initial data $(\rho_0, u_0, w_0, b_0)$ satisfies the conditions \eqref{a10} and \eqref{a12}. Then there exists
a positive time $T_0>0$ and a unique classical solution $(\rho, u, w, b, \Phi)$ of the problem \eqref{a1}--\eqref{a6} in $\Omega\times(0, T_0]$.
\end{lemma}

Next, the well-known Gagliardo-Nirenberg inequality (see \cite{N59}) will be used frequently.
\begin{lemma}
Assume that $\Omega$ is a bounded Lipschitz domain in $\mathbb{R}^3$. For $p\in [2, 6]$, $q\in (1, \infty)$, and $r\in (3, \infty)$,
there exist two generic constants $C_1, C_2>0$, which may depend on $p$, $q$, $r$, and $\Omega$ such that,
for any $f\in H^1(\Omega)$ and $g\in L^q(\Omega)\cap W^{1, r}(\Omega)$,
\begin{align}
&\|f\|_{L^p}\le C_1\|f\|_{L^2}^\frac{6-p}{2p}\|\nabla f\|_{L^2}^\frac{3p-6}{2p}+C_2\|f\|_{L^2},\label{2.3}\\[3pt]
&\|g\|_{L^\infty}\le C_1\|g\|_{L^q}^\frac{q(r-3)}{3r+q(r-3)}\|\nabla g\|_{L^r}^\frac{3r}{3r+q(r-3)}+C_2\|g\|_{L^2}.
\end{align}
Moreover, if $f\cdot n|_{\partial\Omega}=0$ and $g\cdot n|_{\partial\Omega}=0$, then $C_2=0$.
\end{lemma}

Next, the following trace theorem (see \cite[p. 272]{E10}) plays a crucial role in tackling the boundary integrals in the next section.
\begin{lemma}\label{l23}
Assume that $\Omega$ is a bounded domain and $\partial\Omega$ is $C^1$. Then there exists a bounded linear operator
\begin{align*}
T:W^{1,p}(\Omega)\rightarrow L^p(\partial\Omega),\ 1\leq p<\infty
\end{align*}
such that
\begin{align*}
Tu=u|_{\partial\Omega}\ \ \text{for}\ \ u\in W^{1,p}(\Omega)\cap C(\overline{\Omega}),
\end{align*}
and
\begin{align*}
\|Tu\|_{L^p(\partial\Omega)}\leq C\|u\|_{W^{1,p}(\Omega)}\ \ \text{for}\ \ u\in W^{1,p}(\Omega),
\end{align*}
with the constant $C$ depending only on $p$ and $\Omega$.
\end{lemma}

Since our solution $(u,w,b)$ does not vanish on the boundary,
we need to use the following generalized Poincar{\'e} inequality (see \cite[Lemma 8]{BS2012}).
\begin{lemma}\label{l21}
Let $\Omega\subset\mathbb{R}^3$ be bounded with Lipschitz boundary. Then, for $1<p<\infty$, there exists a positive constant $C$ depending only on $p$ and $\Omega$ such that
\begin{align}\label{z2.5}
\|f\|_{L^p} \leq  C\|\nabla f\|_{L^p},
\end{align}
for each vector field $f\in W^{1,p}(\Omega)$ satisfying either $(f\cdot n)|_{\partial\Omega}=0$ or $(f\times n)|_{\partial\Omega}=0$.
\end{lemma}

The following two lemmas are given in \cite{JA14, W92}.
\begin{lemma}\label{l24}
Let $k\ge 0$ be an integer and $1<q<\infty$. Assume that $\Omega$ is a simply connected bounded domain in $\mathbb{R}^3$ with
$C^{k+1, 1}$ boundary $\partial\Omega$. Then, for $v\in W^{k+1, q}(\Omega)$ with $(v\cdot n)|_{\partial\Omega}=0$, it holds that
\begin{align*}
\|v\|_{W^{k+1, q}}\le C\big(\|{\rm div}\,v\|_{W^{k, q}}+\|{\rm curl}\,v\|_{W^{k, q}}\big).
\end{align*}
In particular, for $k=0$, we have
\begin{align*}
\|\nabla v\|_{L^q}\le C\big(\|{\rm div}\,v\|_{L^q}+\|{\rm curl}\,v\|_{L^q}\big).
\end{align*}
\end{lemma}

\begin{lemma}\label{l25}
Let $k\ge 0$ be an integer and $1<q<\infty$. Suppose that $\Omega$ is bounded domain in $\mathbb{R}^3$ and its $C^{k+1, 1}$
boundary $\Omega$ has a finite number of two-dimensional connected components. Then, for $v\in W^{k+1, q}(\Omega)$ with
$(v\times n)|_{\partial\Omega}=0$, we have
\begin{align*}
\|v\|_{W^{k+1, q}}\le C\big(\|{\rm div}\,v\|_{W^{1, q}}+\|{\rm curl}\,v\|_{W^{k, q}}+\|v\|_{L^q}\big).
\end{align*}
In particular, if $\Omega$ is a simply connected bounded domain, then it holds that
\begin{align*}
\|v\|_{W^{k+1, q}}\le C\big(\|{\rm div}\,v\|_{W^{k, q}}+\|\curl v\|_{W^{k, q}}\big).
\end{align*}
\end{lemma}

When $v$ satisfies $v\cdot n=0$ on $\partial\Omega$, we will also use the identity
\begin{equation}\label{2.5}
(v\cdot\nabla)v\cdot n=-(v\cdot\nabla)n\cdot v\ \ \mbox{on}\ \ \partial\Omega
\end{equation}
for any smooth vector field $v$.

The following estimates (see \cite[Lemma 2.10]{CJ21}) on the material derivative of $u$ will be useful.
\begin{lemma}
If $(\rho, u, w, b, \Phi)$ is a smooth solution of \eqref{a1}--\eqref{a6}. Then there exists a positive constant $C$ depending only on $\Omega$ such that
\begin{align}
\|\dot{u}\|_{L^6}& \le C\big(\|\nabla\dot{u}\|_{L^2}+\|\nabla u\|_{L^2}^2\big),\label{3.2}\\
\|\nabla\dot{u}\|_{L^2}& \le C\big(\|{\rm div}\,\dot{u}\|_{L^2}+\|{\rm curl}\,\dot{u}\|_{L^2}+\|\nabla u\|_{L^4}^2\big),
\end{align}
where $\dot{f}\triangleq f_t+u\cdot\nabla f$.
\end{lemma}

Similarly, we have
\begin{align}
\|\dot{w}\|_{L^6}& \le C\big(\|\nabla\dot{w}\|_{L^2}+\|\nabla u\|_{L^2}^2+\|\nabla w\|_{L^2}^2\big),\label{z2.9}\\
\|\nabla\dot{w}\|_{L^2}& \le C\big(\|\divv \dot{u}\|_{L^2}+\|\curl \dot{u}\|_{L^2}+\|\nabla u\|_{L^4}^2
+\|\nabla w\|_{L^4}^2\big).
\end{align}

Next, we introduce the effective viscous
flux of the system \eqref{a1} as the following
\begin{align}\label{2.7}
F_1\triangleq(2\mu+\lambda)\divv u-(P-P_s)-\frac12|b|^2, \
F_2\triangleq(2c_d+c_0)\divv w.
\end{align}
\begin{lemma}\label{l28}
Let $(\rho, u, w, b, \Phi)$ be a smooth solution of \eqref{a1}--\eqref{a6} in $\Omega\times(0, T]$. Then, for any $p\in [2, 6]$ and $1<q<\infty$,
there exists a positive constant $C$ depending only on $p$, $q$, $\mu$, $\lambda$, $u_r$, $c_0$, $c_a$, $c_d$, and $\Omega$ such that
\begin{align}
\|\nabla u\|_{L^q}&\le C\big(\|{\rm div}\,u\|_{L^q}+\|{\rm curl}\,u\|_{L^q}\big),\label{2.8}\\
\|\nabla w\|_{L^q}&\le C\big(\|{\rm div}\,w\|_{L^q}+\|{\rm curl}\,w\|_{L^q}\big),\label{2.9}\\
\|\nabla b\|_{L^q}&\le C\|{\rm curl}\,b\|_{L^q},\label{2.10}\\
\|\nabla F_1\|_{L^q}&\le C(\|\rho\dot{u}\|_{L^q}+\|{\rm curl}\, w\|_{L^q}+\|b\cdot\nabla b\|_{L^q}
+\|\rho\nabla(\Phi-\Phi_s)\|_{L^q}+\|(\rho-\rho_s)\nabla\Phi_s\|_{L^q}),\label{2.11}\\[3pt]
\|\nabla F_2\|_{L^q}&\le C\big(\|\rho\dot{w}\|_{L^q}+\|{\rm curl}\, u\|_{L^q}+\|w\|_{L^q}\big),\label{2.12}\\
\|F_1\|_{L^p}&\le C\big(\|\rho\dot{u}\|_{L^2}+\|\nabla w\|_{L^2}+\|b\cdot\nabla b\|_{L^2}
+\|\rho\nabla(\Phi-\Phi_s)\|_{L^2}+\|(\rho-\rho_s)\nabla\Phi_s\|_{L^2}\big)^\frac{3p-6}{2p}\nonumber\\
&\quad\times\big(\|\nabla u\|_{L^2}+\|P-P_s\|_{L^2}+\|b\|_{L^4}^2\big)^\frac{6-p}{2p}+C\big(\|\nabla u\|_{L^2}+\|P-P_s\|_{L^2}+\|b\|_{L^4}^2\big),\\
\|F_2\|_{L^p}&\le C\big(\|\rho\dot{w}\|_{L^2}+\|\nabla u\|_{L^2}+\|w\|_{L^2}\big)^\frac{3p-6}{2p}\|\nabla w\|_{L^2}^\frac{6-p}{2p}
+C\|\nabla w\|_{L^2},\\
\|{\rm curl}\,u\|_{L^p}&\le C\big(\|\rho\dot{u}\|_{L^2}+\|\rho\dot{w}\|_{L^2}
+\|\rho\nabla(\Phi-\Phi_s)\|_{L^2}+\|(\rho-\rho_s)\nabla\Phi_s\|_{L^2}
+\|b\cdot\nabla b\|_{L^2}\nonumber\\
&\quad+\|\nabla w\|_{L^2}+\|w\|_{L^2}\big)^\frac{3p-6}{2p}\|\nabla u\|_{L^2}^\frac{6-p}{2p}
+C\|\nabla u\|_{L^2},\label{2.16}\\
\|{\rm curl}\,w\|_{L^p}&\le C\big(\|\rho\dot{u}\|_{L^2}+\|\rho\dot{w}\|_{L^2}
+\|\rho\nabla(\Phi-\Phi_s)\|_{L^2}+\|(\rho-\rho_s)\nabla\Phi_s\|_{L^2}
+\|\nabla u\|_{L^2}\nonumber\\
&\quad+\|b\cdot\nabla b\|_{L^2}+\|w\|_{L^2}\big)^\frac{3p-6}{2p}\|\nabla w\|_{L^2}^\frac{6-p}{2p}+C\|\nabla w\|_{L^2},\label{2.17}
\end{align}
and
\begin{align}\label{2.13}
\|\nabla{\rm curl}\,u\|_{L^p}+\|\nabla{\rm curl}\,w\|_{L^p}
& \le
C\big(\|\rho\dot{u}\|_{L^p}+\|\rho\dot{w}\|_{L^p}
+\|\rho\nabla(\Phi-\Phi_s)\|_{L^p}+\|(\rho-\rho_s)\nabla\Phi_s\|_{L^p}\big)
\nonumber\\
&\quad+C\big(\|b\cdot\nabla b\|_{L^p}+\|\nabla u\|_{L^2}+\|\nabla w\|_{L^2}+\|w\|_{L^2}\big),
\end{align}
Moreover, one also has
\begin{align}
\|F_1\|_{L^p}
&\le C(\|\rho\dot{u}\|_{L^2}+\|\nabla w\|_{L^2}+\|\nabla u\|_{L^2}+\|b\cdot\nabla b\|_{L^2}+\|P-P_s\|_{L^2})\nonumber\\
&\quad+C(\|\rho\nabla(\Phi-\Phi_s)\|_{L^2}+\|(\rho-\rho_s)\nabla\Phi_s\|_{L^2}
+\|b\|_{L^4}^2),\label{2.18}\\
\|F_2\|_{L^p}&\le C(\|\rho\dot{w}\|_{L^2}+\|\nabla u\|_{L^2}+\|\nabla w\|_{L^2}+\|w\|_{L^2}),\label{2.19}\\
\|{\rm curl}\,u\|_{L^p}+\|{\rm curl}\,w\|_{L^p}&\le C(\|\rho\dot{u}\|_{L^2}+\|\rho\dot{w}\|_{L^2}
+\|\rho\nabla(\Phi-\Phi_s)\|_{L^2}+\|(\rho-\rho_s)\nabla\Phi_s\|_{L^2})\nonumber\\[3pt]
&\quad+C(\|b\cdot\nabla b\|_{L^2}+\|\nabla u\|_{L^2}+\|\nabla w\|_{L^2}+\|w\|_{L^2}),\label{2.20}\\
\|\nabla w\|_{L^p}&\le C(\|\rho\dot{u}\|_{L^2}+\|\rho\dot{w}\|_{L^2}
+\|\rho\nabla(\Phi-\Phi_s)\|_{L^2}+\|(\rho-\rho_s)\nabla\Phi_s\|_{L^2}\nonumber\\
&\quad+\|\nabla u\|_{L^2}+\|b\cdot\nabla b\|_{L^2}+\|w\|_{L^2})^\frac{3p-6}{2p}\|\nabla w\|_{L^2}^\frac{6-p}{2p}+C\|\nabla w\|_{L^2},\label{2.22}
\end{align}
and
\begin{align}
\|\nabla u\|_{L^p}&\le  C\big(\|\rho\dot{u}\|_{L^2}+\|\rho\dot{w}\|_{L^2}+\|\nabla w\|_{L^2}+\|b\cdot\nabla b\|_{L^2}
+\|\rho\nabla(\Phi-\Phi_s)\|_{L^2}
+\|(\rho-\rho_s)\nabla\Phi_s\|_{L^2}\big)^\frac{3p-6}{2p}\nonumber\\
&\quad\times\big(\|\nabla u\|_{L^2}+\|P-P_s\|_{L^2}+\|b\|_{L^4}^2\big)^\frac{6-p}{2p}+C\big(\|\nabla u\|_{L^2}+\|P-P_s\|_{L^p}+\||b|^2\|_{L^p}\big).
\end{align}

\end{lemma}
\begin{proof}[Proof]
1. Due to $\eqref{a6}$ and $\divv b=0$, we derive \eqref{2.8}--\eqref{2.10} from Lemma \ref{l24}. Moreover, by $\eqref{a1}_4$, $\eqref{a6}$, and $\divv(\curl b)=0$,
we obtain from Lemmas \ref{l24} and \ref{l25} that, for any integer $k\ge 1$,
\begin{align}\label{2.23}
\|b\|_{W^{k+1, q}}\le C\|{\rm curl}\,b\|_{W^{k, q}}
\le C\big(\|{\rm curl}^2b\|_{W^{k-1, q}}+\|{\rm curl}\,b\|_{L^q}\big).
\end{align}
By $\eqref{a1}_2$, $\eqref{a1}_3$, \eqref{a6}, and \eqref{2.7},
one finds that $F_1$ and $F_2$ satisfy
\begin{align}\label{2.26}
\begin{cases}
\Delta F_1={\rm div}\,(\rho\dot{u}-2\mu_r{\rm curl}\,w-b\cdot\nabla b-\rho\nabla(\Phi-\Phi_s)
-(\rho-\rho_s)\nabla\Phi_s)&\text{in}\ \ \Omega,\\
\frac{\partial F_1}{\partial n}=(\rho\dot{u}-2\mu_r{\rm curl}\,w-b\cdot\nabla b)\cdot n &\text{on}\ \ \partial\Omega,
\end{cases}
\end{align}
and
\begin{align}\label{2.27}
\begin{cases}
\Delta F_2={\rm div}\,(\rho\dot{w}+4\mu_rw-2\mu_r{\rm curl}\,u)&\text{in}\ \ \Omega,\\
\frac{\partial F_2}{\partial n}=(\rho\dot{w}-2\mu_r{\rm curl}\,u)\cdot n &\text{on}\ \ \partial\Omega.
\end{cases}
\end{align}
Then we obtain \eqref{2.11} and \eqref{2.12} from \eqref{2.26}, \eqref{2.27}, and \cite[Lemma 4.27]{NS04}. Moreover,
for any integer $k\ge 0$,
\begin{align}\label{z2.28}
\|\nabla F_1\|_{W^{k+1, q}}&\le
C(\|\rho\dot{u}\|_{L^q}+\|{\rm curl} w\|_{L^q}+\|b\cdot\nabla b\|_{L^q}
+\|\rho\nabla(\Phi-\Phi_s)\|_{L^q}+\|(\rho-\rho_s)\nabla\Phi_s\|_{L^q})\nonumber\\
&\quad+C(\|\nabla(\rho\dot{u})\|_{W^{k, q}}+\|\nabla\curl w\|_{W^{k, q}}+\|\nabla(b\cdot\nabla b)\|_{W^{k, q}}
+\|\nabla(\rho\nabla(\Phi-\Phi_s))\|_{W^{k, q}})\nonumber\\
&\quad+C\|\nabla((\rho-\rho_s)\nabla\Phi_s)\|_{W^{k, q}},
\end{align}
and
\begin{align}\label{z2.29}
\|\nabla F_2\|_{W^{k+1, q}}\le C(\|\rho\dot{w}\|_{L^q}+\|{\rm curl}\, u\|_{L^q}+\|w\|_{L^q}
+\|\nabla(\rho\dot{w})\|_{W^{k, q}}+\|\nabla\curl u\|_{W^{k, q}}+\|\nabla w\|_{W^{k, q}}).
\end{align}

2. By \eqref{2.7} and \eqref{a7}, we rewrite $\eqref{a1}_2$ and $\eqref{a1}_3$ as follows
\begin{align}
&(\mu+\mu_r)\curl{\rm curl}\,u=\nabla F_1-\rho\dot{u}+b\cdot\nabla b+\rho\nabla(\Phi-\Phi_s)+(\rho-\rho_s)\nabla\Phi_s+2\mu_r{\rm curl}\,w,\\
&(c_a+c_d)\curl{\rm curl}\,w=\nabla F_2-\rho\dot{w}-4\mu_rw+2\mu_r{\rm curl}\,u.
\end{align}
Noticing that ${\rm div}\,(\curl{\rm curl}\,u)=0$ and
${\rm div}\,(\curl{\rm curl}\,w)=0$, we get from \eqref{a6} and Lemma \ref{l25} that
\begin{align}
\|\nabla{\rm curl}\,u\|_{L^q}&\le C(\|\curl{\rm curl}\,u\|_{L^q}+\|{\rm curl}\,u\|_{L^q})\nonumber\\
&\le C(\|\rho\dot{u}\|_{L^q}+\|b\cdot\nabla b\|_{L^q}+\|\rho\nabla(\Phi-\Phi_s)\|_{L^q}+\|(\rho-\rho_s)\nabla\Phi_s\|_{L^q})\nonumber\\
&\quad+(\|{\rm curl}\,w\|_{L^q}+\|{\rm curl}\,u\|_{L^q}),\label{2.30}\\[3pt]
\|\nabla{\rm curl}\,w\|_{L^q}&\le C(\|\curl{\rm curl}\,w\|_{L^q}+\|{\rm curl}\,w\|_{L^q})\nonumber\\
&\le C(\|\rho\dot{w}\|_{L^q}+\|{\rm curl}\,u\|_{L^q}+\|{\rm curl}\,w\|_{L^q}+\|w\|_{L^q}).\label{2.31}
\end{align}
By Sobolev's inequality, H\"older's inequality, \eqref{2.30}, and \eqref{2.31}, we derive that, for $p\in [2, 6]$,
\begin{align}\label{2.32}
\|\nabla{\rm curl}\,u\|_{L^p}&\le C(\|\rho\dot{u}\|_{L^p}+\|b\cdot\nabla b\|_{L^p}+\|\rho\nabla(\Phi-\Phi_s)\|_{L^p}+\|(\rho-\rho_s)\nabla\Phi_s\|_{L^p})\nonumber\\
&\quad+(\|{\rm curl}\,w\|_{L^p}+\|{\rm curl}\,u\|_{L^p})\nonumber\\
&\le C(\|\rho\dot{u}\|_{L^p}+\|{\rm curl}\,u\|_{L^2}
+\|{\rm curl}\,w\|_{L^2}+\|\nabla{\rm curl}\,u\|_{L^2}+\|\nabla{\rm curl}\,w\|_{L^2})\nonumber\\
&\quad+C(\|b\cdot\nabla b\|_{L^p}+\|\rho\nabla(\Phi-\Phi_s)\|_{L^p}+\|(\rho-\rho_s)\nabla\Phi_s\|_{L^p})\nonumber\\
&\le C(\|\rho\dot{u}\|_{L^p}+\|\rho\dot{w}\|_{L^p}+\|b\cdot\nabla b\|_{L^p}
+\|\rho\nabla(\Phi-\Phi_s)\|_{L^p}+\|(\rho-\rho_s)\nabla\Phi_s\|_{L^p})\nonumber\\
&\quad+C(\|\nabla u\|_{L^2}+\|\nabla w\|_{L^2}+\|w\|_{L^2}),
\end{align}
and
\begin{align}\label{2.33}
\|\nabla{\rm curl}\,w\|_{L^p}&\le C(\|\rho\dot{w}\|_{L^p}+\|{\rm curl}\,u\|_{L^p}+\|{\rm curl}\,w\|_{L^p}+\|w\|_{L^p})\nonumber\\
&\le C(\|\rho\dot{w}\|_{L^p}+\|{\rm curl}\,u\|_{L^2}
+\|{\rm curl}\,w\|_{L^2}+\|\nabla w\|_{L^2}+\|w\|_{L^2})\nonumber\\
&\quad+C(\|\nabla{\rm curl}\,u\|_{L^2}+\|\nabla{\rm curl}\,w\|_{L^2})\nonumber\\
&\le C(\|\rho\dot{u}\|_{L^p}+\|\rho\dot{w}\|_{L^p}+\|b\cdot\nabla b\|_{L^p}
+\|\rho\nabla(\Phi-\Phi_s)\|_{L^p}+\|(\rho-\rho_s)\nabla\Phi_s\|_{L^p})\nonumber\\
&\quad+C(\|\nabla u\|_{L^2}+\|\nabla w\|_{L^2}+\|w\|_{L^2}).
\end{align}
Thus, \eqref{2.13} follows from \eqref{2.32} and \eqref{2.33}. Furthermore, for any integer $k\ge 0$,
\begin{align}\label{z2.36}
&\|\nabla\curl u\|_{W^{k+1, q}}+\|\nabla\curl w\|_{W^{k+1, q}}\nonumber\\
&\le C(\|\curl\curl u\|_{W^{k+1, q}}+\|\curl\curl w\|_{W^{k+1, q}}
+\|\curl u\|_{L^q}+\|\curl w\|_{L^q})\nonumber\\
&\le C(\|\rho\dot{u}\|_{L^q}+\|\rho\dot{w}\|_{L^q}+\|b\cdot\nabla b\|_{L^q}+\|\rho\nabla(\Phi-\Phi_s)\|_{L^q}+\|(\rho-\rho_s)\nabla\Phi_s\|_{L^q}
+\|\nabla(\rho\dot{u})\|_{W^{k, q}})\nonumber\\
&\quad+C(\|\nabla u\|_{L^2}+\|\nabla w\|_{L^2}+\|w\|_{L^2}+\|\nabla\curl w\|_{W^{k, q}}
+\|\nabla \curl u\|_{W^{k, q}}+\|\nabla(b\cdot\nabla b)\|_{W^{k, q}})\nonumber\\
&\quad+C\|\nabla(\rho\nabla(\Phi-\Phi_s))\|_{W^{k, q}})
+C\|\nabla((\rho-\rho_s)\nabla\Phi_s)\|_{W^{k, q}}.
\end{align}

3. One deduces from \eqref{2.3}, \eqref{2.11}, and \eqref{2.12} that, for $p\in [2, 6]$,
\begin{align*}
\|F_1\|_{L^p}&\le C\|F_1\|_{L^2}^\frac{6-p}{2p}\|\nabla F_1\|_{L^2}^\frac{3p-6}{2p}+C\|F_1\|_{L^2}\nonumber\\
&\le C(\|\rho\dot{u}\|_{L^2}+\|\nabla w\|_{L^2}+\|b\cdot\nabla b\|_{L^2}
+\|\rho\nabla(\Phi-\Phi_s)\|_{L^2}+\|(\rho-\rho_s)\nabla\Phi_s\|_{L^2})^\frac{3p-6}{2p}\nonumber\\
&\quad\times(\|\nabla u\|_{L^2}+\|P-P_s\|_{L^2}+\|b\|_{L^4}^2)^\frac{6-p}{2p}+C(\|\nabla u\|_{L^2}+\|P-P_s\|_{L^2}+\|b\|_{L^4}^2),\\
\|F_2\|_{L^p}&\le C\|F_2\|_{L^2}^\frac{6-p}{2p}\|\nabla F_2\|_{L^2}^\frac{3p-6}{2p}+C\|F_2\|_{L^2}\nonumber\\
&\le C(\|\rho\dot{w}\|_{L^2}+\|\nabla u\|_{L^2}+\|w\|_{L^2})^\frac{3p-6}{2p}\|\nabla w\|_{L^2}^\frac{6-p}{2p}
+C\|\nabla w\|_{L^2},
\end{align*}
which implies that
\begin{align}
\|F_1\|_{L^p}&\le C(\|\rho\dot{u}\|_{L^2}+\|\nabla w\|_{L^2}+\|\nabla u\|_{L^2}+\|b\cdot\nabla b\|_{L^2}+\|P-P_s\|_{L^2})\nonumber\\
&\quad+C(\|\rho\nabla(\Phi-\Phi_s)\|_{L^2}+\|(\rho-\rho_s)\nabla\Phi_s\|_{L^2}+\|b\|_{L^4}^2),\\[3pt]
\|F_2\|_{L^p}&\le C(\|\rho\dot{w}\|_{L^2}+\|\nabla u\|_{L^2}+\|\nabla w\|_{L^2}+\|w\|_{L^2}).
\end{align}

4. By \eqref{2.3} and \eqref{2.13}, we arrive at
\begin{align*}
\|{\rm curl}\,u\|_{L^p}&\le C\|{\rm curl}\,u\|_{L^2}^\frac{6-p}{2p}\|\nabla{\rm curl}\,u\|_{L^2}^\frac{3p-6}{2p}
+C\|{\rm curl}\,u\|_{L^2}\nonumber\\
&\le C(\|\rho\dot{u}\|_{L^2}+\|\rho\dot{w}\|_{L^2}+\|b\cdot\nabla b\|_{L^2}
+\|\rho\nabla(\Phi-\Phi_s)\|_{L^2}+\|(\rho-\rho_s)\nabla\Phi_s\|_{L^2}\nonumber\\
&\quad+\|\nabla u\|_{L^2}+\|\nabla w\|_{L^2}+\|w\|_{L^2})^\frac{3p-6}{2p}\|\nabla u\|_{L^2}^\frac{6-p}{2p}
+C\|\nabla u\|_{L^2}\nonumber\\
&\le C(\|\rho\dot{u}\|_{L^2}+\|\rho\dot{w}\|_{L^2}
+\|\rho\nabla(\Phi-\Phi_s)\|_{L^2}+\|(\rho-\rho_s)\nabla\Phi_s\|_{L^2}\nonumber\\
&\quad+\|b\cdot\nabla b\|_{L^2}+\|\nabla w\|_{L^2}+\|w\|_{L^2})^\frac{3p-6}{2p}\|\nabla u\|_{L^2}^\frac{6-p}{2p}
+C\|\nabla u\|_{L^2},\\
\|{\rm curl}\,w\|_{L^p}&\le C\|{\rm curl}\,w\|_{L^2}^\frac{6-p}{2p}\|\nabla{\rm curl}\,w\|_{L^2}^\frac{3p-6}{2p}
+C\|{\rm curl}\,w\|_{L^2}\nonumber\\
&\le C(\|\rho\dot{u}\|_{L^2}+\|\rho\dot{w}\|_{L^2}+\|b\cdot\nabla b\|_{L^2}
+\|\rho\nabla(\Phi-\Phi_s)\|_{L^2}+\|(\rho-\rho_s)\nabla\Phi_s\|_{L^2}\nonumber\\
&\quad+\|\nabla u\|_{L^2}+\|\nabla w\|_{L^2}+\|w\|_{L^2})^\frac{3p-6}{2p}\|\nabla w\|_{L^2}^\frac{6-p}{2p}
+C\|\nabla w\|_{L^2}\nonumber\\
&\le C(\|\rho\dot{u}\|_{L^2}+\|\rho\dot{w}\|_{L^2}
+\|\rho\nabla(\Phi-\Phi_s)\|_{L^2}+\|(\rho-\rho_s)\nabla\Phi_s\|_{L^2}\nonumber\\
&\quad+\|\nabla u\|_{L^2}+\|b\cdot\nabla b\|_{L^2}+\|w\|_{L^2})^\frac{3p-6}{2p}\|\nabla w\|_{L^2}^\frac{6-p}{2p}+C\|\nabla w\|_{L^2},
\end{align*}
which leads to
\begin{align}
\|{\rm curl}\,w\|_{L^p}+\|{\rm curl}\,u\|_{L^p}&\le C(\|\rho\dot{u}\|_{L^2}+\|\rho\dot{w}\|_{L^2}
+\|\rho\nabla(\Phi-\Phi_s)\|_{L^2}+\|(\rho-\rho_s)\nabla\Phi_s\|_{L^2})\nonumber\\
&\quad+C(\|b\cdot\nabla b\|_{L^2}\|\nabla u\|_{L^2}+\|\nabla w\|_{L^2}+\|w\|_{L^2}).
\end{align}

5. By virtue of \eqref{2.7}, \eqref{2.8}, \eqref{2.9}, \eqref{2.11}, \eqref{2.12}, \eqref{2.18}, and \eqref{2.19}, it indicates that
\begin{align*}
\|\nabla u\|_{L^p}&\le C\big(\|{\rm div}\,u\|_{L^p}+\|{\rm curl}\,u\|_{L^p}\big)\nonumber\\
&\le C\big(\|F_1\|_{L^p}+\|P-P_s\|_{L^p}+\||b|^2\|_{L^p}+\|{\rm curl}\,u\|_{L^p}\big)\nonumber\\
&\le C\big(\|\rho\dot{u}\|_{L^2}+\|\rho\dot{w}\|_{L^2}+\|\nabla w\|_{L^2}+\|b\cdot\nabla b\|_{L^2}
+\|\rho\nabla(\Phi-\Phi_s)\|_{L^2}+\|(\rho-\rho_s)\nabla\Phi_s\|_{L^2}\big)^\frac{3p-6}{2p}\nonumber\\
&\quad\times\big(\|\nabla u\|_{L^2}+\|P-P_s\|_{L^2}+\|b\|_{L^4}^2\big)^\frac{6-p}{2p}+C\big(\|\nabla u\|_{L^2}+\|P-P_s\|_{L^p}+\||b|^2\|_{L^p}\big),
\end{align*}
and
\begin{align*}
\|\nabla w\|_{L^p}&\le C\big(\|{\rm div}\,w\|_{L^p}+\|{\rm curl}\,w\|_{L^p}\big)\nonumber\\
&\le C\big(\|F_2\|_{L^p}+\|{\rm curl}\,w\|_{L^p}\big)\nonumber\\
&\le  C\big(\|\rho\dot{u}\|_{L^2}+\|\rho\dot{w}\|_{L^2}
+\|\rho\nabla(\Phi-\Phi_s)\|_{L^2}+\|(\rho-\rho_s)\nabla\Phi_s\|_{L^2}\nonumber\\
&\quad+\|\nabla u\|_{L^2}+\|b\cdot\nabla b\|_{L^2}+\|w\|_{L^2}\big)^\frac{3p-6}{2p}\|\nabla w\|_{L^2}^\frac{6-p}{2p}+C\|\nabla w\|_{L^2}.
\end{align*}
This completes the proof.
\end{proof}

\begin{remark}
From \eqref{2.23} and $\eqref{a1}_6$, for $p\in [2, 6]$,
\begin{align}
\|\nabla^2b\|_{L^p}&\le C\|\curl b\|_{W^{1, p}}\le C(\|\curl^2b\|_{L^p}+\|\curl b\|_{L^p}),\\
\|\nabla^3b\|_{L^p}&\le C\|\curl b\|_{W^{2, p}}\le C(\|\curl^2b\|_{W^{1, p}}+\|\curl b\|_{L^p}).\label{z2.41}
\end{align}
Moreover, by Lemma \ref{l24}, \eqref{2.7}, \eqref{2.11}, \eqref{2.12}, \eqref{2.13},
\eqref{z2.28}, \eqref{z2.29}, and \eqref{z2.36}, for $p\in [2, 6]$,
\begin{align}\label{z2.42}
&\|\nabla^2u\|_{L^p}+\|\nabla^2w\|_{L^p}\nonumber\\
&\le C(\|\divv u\|_{W^{1, p}}+\|\divv w\|_{W^{1, p}}+\|\curl u\|_{W^{1, p}}+\|\curl w\|_{W^{1, p}})\nonumber\\
&\le C(\|F_1+P-P_s+|b|^2\|_{W^{1, p}}+\|F_2\|_{W^{1, p}}+\|\curl u\|_{W^{1, p}}+\|\curl w\|_{W^{1, p}})\nonumber\\
&\le C(\|\rho\dot{u}\|_{L^p}+\|\rho\dot{w}\|_{L^p}+\|b\cdot\nabla b\|_{L^p}
+\|\rho\nabla(\Phi-\Phi_s)\|_{L^p}+\|(\rho-\rho_s)\nabla\Phi_s\|_{L^p})\nonumber\\
&\quad+C(\||b|^2\|_{L^p}+\|P-P_s\|_{L^p}+\|\nabla(P-P_s)\|_{L^p}+\|\nabla u\|_{L^2}+\|\nabla w\|_{L^2}+\|w\|_{L^2}),
\end{align}
and
\begin{align}\label{z2.43}
&\|\nabla^3u\|_{L^p}+\|\nabla^3w\|_{L^p}\nonumber\\
&\le C(\|\divv u\|_{W^{2, p}}+\|\divv w\|_{W^{2, p}}+\|\curl u\|_{W^{2, p}}+\|\curl w\|_{W^{2, p}})\nonumber\\
&\le C(\|F_1+P-P_s+|b|^2\|_{W^{2, p}}+\|F_2\|_{W^{2, p}}+\|\curl u\|_{W^{2, p}}+\|\curl w\|_{W^{2, p}})\nonumber\\
&\le C(\|\rho\dot{u}\|_{L^p}+\|\rho\dot{w}\|_{L^p}+\|b\cdot\nabla b\|_{L^p}+\|\rho\nabla(\Phi-\Phi_s)\|_{L^p}+\|(\rho-\rho_s)\nabla\Phi_s\|_{L^p}
+\|\nabla w\|_{L^2})\nonumber\\
&\quad+C(\|\nabla u\|_{L^2}+\|w\|_{L^2}+\|\nabla(\rho\dot{u})\|_{L^p}+\|\nabla\curl w\|_{L^p}
+\|\nabla \curl u\|_{L^p}+\|\nabla(b\cdot\nabla b)\|_{L^p})\nonumber\\
&\quad+C\|\nabla(\rho\nabla(\Phi-\Phi_s))\|_{L^p})
+C\|\nabla((\rho-\rho_s)\nabla\Phi_s)\|_{L^p}+C\|\nabla^2(P-P_s)\|_{L^p}+\||b|^2\|_{L^p}.
\end{align}

\end{remark}

Next, we derive the estimates of $\nabla(\Phi-\Phi_s)$ and $\nabla(\Phi-\Phi_s)_t$.
\begin{lemma}\label{l29}
Let $(\rho, u, w, b, \Phi)$ be a smooth solution of \eqref{a1}--\eqref{a6}. Then, for any integer $k\ge 0$ and $1<q<\infty$, there exists
a positive constant $C$ depending only on $k$, $q$, and $\Omega$ such that
\begin{align}
&\|\nabla(\Phi-\Phi_s)\|_{W^{k+1, q}}\le C\|\rho-\rho_s\|_{W^{k ,q}},\label{2.42}\\
&\|\nabla(\Phi-\Phi_s)_t\|_{W^{k+1, q}}\le C(\|\rho u\|_{L^q}+\|\rho u\|_{W^{k, q}}).\label{2.43}
\end{align}
\end{lemma}
\begin{proof}[Proof]
It follows from $\eqref{a1}_5$, \eqref{a6}, and \eqref{a7} that
\begin{align}\label{2.44}
\begin{cases}
\Delta(\Phi-\Phi_s)=\rho-\rho_s &in ~\Omega,\\
\frac{\partial(\Phi-\Phi_s)}{\partial n}=0 &on~\partial\Omega,
\end{cases}
\end{align}
which together with the classical regularity theory for the Neumann problem of elliptic equation \cite{BB74} implies \eqref{2.42}.
By $\eqref{a1}_1$ and \eqref{2.44}, one has
\begin{align}
\begin{cases}
\Delta(\Phi-\Phi_s)=-{\rm div}\,(\rho u) &in ~\Omega,\\
\frac{\partial(\Phi-\Phi_s)_t}{\partial n}=0 &on~\partial\Omega.
\end{cases}
\end{align}
Similar to the proof of Lemma \ref{l28}, we obtain that, for $1<q<\infty$,
\begin{align}
\|\nabla(\Phi-\Phi_s)_t\|_{L^q}\le C\|\rho u\|_{L^q}.
\end{align}
Moreover, for $k\ge 0$,
\begin{align}
\|\nabla(\Phi-\Phi_s)_t\|_{W^{k+1, q}}\le C\big(\|\rho u\|_{L^q}+\|\rho u\|_{W^{k, q}}\big).
\end{align}
The conclusion follows.
\end{proof}

The following Beale-Kato-Majda type inequality (see \cite[Lemma 2.7]{CJ21}) with respect to the slip boundary condition \eqref{a6} will be used to estimate $\|\nabla u\|_{L^\infty}$.
\begin{lemma}\label{l211}
Let $\Omega$ be a bounded simply connected domain in $\mathbb{R}^3$ with smooth boundary. Assume that $v\in W^{2, q}(\Omega)\ (3<q<\infty)$ satisfies that $v\cdot n=0$
and $\curl v\times n=0$ on $\partial\Omega$, then there exists a constant $C=C(q, \Omega)$ such that
\begin{align*}
\|\nabla v\|_{L^\infty}
\le C\big(\|\divv v\|_{L^\infty}+\|\curl v\|_{L^\infty}\big)\ln\big(e+\|\nabla^2 v\|_{L^q}\big)+C\|\nabla v\|_{L^2}+C.
\end{align*}
\end{lemma}

Finally, the following Zlotnik inequality (see \cite[Lemma 1.3]{Z00}) will be used to get the uniform-in-time upper bound of the density.
\begin{lemma}\label{l210}
Suppose the function $y$ satisfy
\begin{align*}
y'(t)=g(y)+b'(t)~on~[0, T], \quad y(0)=y^0,
\end{align*}
with $g\in C(R)$ and $y, b\in W^{1, 1}(0, T)$. If $g(\infty)=-\infty$ and
\begin{align*}
b(t_2)-b(t_1)\le N_0+N_1(t_2-t_1)
\end{align*}
for all $0\le t<t_2\le T$ with some $N_0\ge 0$ and $N_1\ge 0$, then
\begin{align*}
y(t)\le \max\{y^0, \xi_0\}+N_0<\infty~on~[0, T],
\end{align*}
where $\xi_0$ is a constant such that
\begin{align*}
g(\xi)\le -N_1, \quad for\quad \xi\ge \xi_0.
\end{align*}
\end{lemma}

\section{\textit{A priori} estimates}\label{sec3}
In this section, we will establish some necessary \textit{a priori} bounds for smooth solutions to the problem \eqref{a1}--\eqref{a6} in order to extend the local classical solutions guaranteed by Lemma \ref{l22}. Let $T>0$ be a fixed time and $(\rho, u, w, b, \Phi)$ be a smooth solution to \eqref{a1}--\eqref{a6} in $\Omega\times (0, T]$ with smooth initial data $(\rho_0, u_0, w_0, b_0)$ satisfying \eqref{a10} and \eqref{a12}.

\subsection{Lower-order estimates}\label{sec3.1}
Throughout this subsection, we will use $C$ or $C_i\ (i=1, 2, \cdots)$
to denote the generic positive constants, which may depend on $\mu$, $\lambda$, $u_r$, $c_0$, $c_a$, $c_d$, $\gamma$, $a$, $\hat{\rho}$, $\Omega$, $M_1$, $M_2$, $M_3$, and $\tilde{\rho}$. In particular, they are independent of $T$. Sometimes we use $C(\alpha)$ to emphasize the dependence of $C$ on $\alpha$.

Set $\sigma=\sigma(t)\triangleq\min\{1, t\}$, we define
\begin{align}\label{3.5}
&A_1(T)\triangleq\sup_{0\le t\le T}\big[\sigma(t)\big(\|\nabla u\|_{L^2}^2+\|\nabla w\|_{L^2}^2+\|\nabla b\|_{L^2}^2\big)\big],\\ \label{z3.2}
&A_2(T)\triangleq\sup_{0\le t\le T}\big(\|\nabla u\|_{L^2}^2+\|\nabla w\|_{L^2}^2+\|\nabla b\|_{L^2}^2\big).
\end{align}
The main aim of this subsection is to obtain the following key \textit{a priori} estimates, which gives the uniform upper bound of density.
\begin{proposition}\label{p31}
Under the conditions of Theorem \ref{thm1}, there exist positive constants
$\varepsilon$ and $K$ both depending on $\mu$, $\lambda$, $u_r$, $c_0$, $c_a$, $c_d$,
$\gamma$, $a$, $\bar{\rho}$, $\hat{\rho}$, $\Omega$, $M_1$, $M_2$, and $M_3$ such that if $(\rho, u, w, b, \Phi)$ is a smooth solution of \eqref{a1}--\eqref{a6} in
$\Omega\times (0, T]$ satisfying
\begin{align}\label{3.6}
\sup_{\Omega\times[0, T]}\rho\le 2\hat{\rho}, \quad A_1(T)\le 2E_0^\frac13, \quad A_2(\sigma(T))\le 4K,
\end{align}
then the following estimates hold
\begin{align}\label{3.7}
\sup_{\Omega\times[0, T]}\rho\le \frac74\hat{\rho}, \quad A_1(T)\le E_0^\frac13, \quad A_2(\sigma(T))\le 3K,
\end{align}
provided that $E_0\le\varepsilon$.
\end{proposition}

Before proving Proposition \ref{p31}, we show some necessary \textit{a priori} estimates, see Lemmas \ref{l32}--\ref{l38} below.

\begin{lemma}\label{l32}
Let $(\rho, u, w, b, \Phi)$ be a smooth solution of \eqref{a1}--\eqref{a6} in $\Omega\times [0, T]$, then it holds that
\begin{align}\label{3.8}
&\sup_{0\le t\le T}\Big(\frac12\|\sqrt{\rho}u\|_{L^2}^2+\frac12\|\sqrt{\rho}w\|_{L^2}^2
+\frac12\|b\|_{L^2}^2+\frac12\|\nabla(\Phi-\Phi_s)\|_{L^2}^2+\|G(\rho)\|_{L^1}\Big)\nonumber\\
&\quad+\int_0^T\int\Big(\mu|\nabla u|^2+(\mu+\lambda)({\rm div}\,u)^2+c_d|\nabla w|^2+(c_0+c_d)({\rm div}\,w)^2\nonumber\\
&\quad+c_a|{\rm curl}\,w|^2
+\mu_r|\curl u-2w|^2+|\nabla b|^2\Big)dxdt\le E_0,
\end{align}
and
\begin{align}\label{3.9}
\sup_{0\le t\le T}\|\rho-\rho_s\|_{L^2}^2+\int_0^T\big(\|w\|_{H^1}^2+\|\nabla u\|_{L^2}^2+\|\nabla b\|_{L^2}^2\big)dt\le CE_0.
\end{align}
\end{lemma}
\begin{proof}[Proof]
1. Noting that
\begin{align*}
&-\Delta u=-\nabla{\rm div}\,u+\curl{\rm curl}\,u, \quad -\Delta w=-\nabla{\rm div}\,w+\curl\curl w,\\
&\nabla\times(u\times b)=b\cdot\nabla u-u\cdot\nabla b-b{\rm div}\,u,\\
&(\nabla\times b)\times b=b\cdot\nabla b-\frac12\nabla|b|^2,\quad -\nabla\times(\nabla\times b)=\Delta b,
\end{align*}
and \eqref{a7}, we rewrite $\eqref{a1}_2$ and $\eqref{a1}_3$ as
\begin{align}
&\rho u_t+\rho u\cdot\nabla u-(2\mu+\lambda)\nabla{\rm div}\,u+(\mu+\mu_r)\curl\curl u
+\nabla(P-P_s)\nonumber\\
&=\rho\nabla(\Phi-\Phi_s)+(\rho-\rho_s)\nabla\Phi_s+(\nabla\times b)\times b+2\mu_r{\rm curl}\,w,\label{3.10}\\[3pt]
&\rho w_t+\rho u\cdot\nabla w-(2c_d+c_0)\nabla{\rm div}\,w+(c_d+c_a)\curl{\rm curl}\,w+4\mu_rw=2\mu_r{\rm curl}\,u.\label{3.11}
\end{align}
Multiplying \eqref{3.10}, \eqref{3.11}, $\eqref{a1}_4$, $\eqref{a1}_1$ by $u$, $w$, $b$, $G'(\rho)$, respectively, summing up, and integrating the resulting equality over $\Omega$, one obtains from \eqref{a7} that
\begin{align}\label{3.12}
&\frac{d}{dt}\int\Big(\frac12\rho|u|^2+\frac12\rho|w|^2+\frac12|b|^2
+G(\rho)\Big)dx+(2\mu+\lambda)\int({\rm div}\,u)^2dx+\mu\int|{\rm curl}\,u|^2dx\nonumber\\
&\quad+(c_a+c_d)\int|{\rm curl}\,w|^2dx+(c_0+2c_d)\int({\rm div}\,w)^2dx
+\mu_r\int|\curl u-2w|^2dx+\int|\nabla b|^2dx\nonumber\\
&=-\int{\rm div}\,(\rho u)G'(\rho)dx-\int u\cdot\nabla(P-P_s)dx+\int(\rho-\rho_s)u\cdot\nabla\Phi_sdx+\int\rho u\cdot\nabla(\Phi-\Phi_s)dx\nonumber\\
&=\int \rho u(\nabla Q(\rho)-\nabla Q(\rho_s))-\int u\cdot\nabla Pdx+\int \rho u\cdot\nabla\Phi_sdx
-\int{\rm div}\,(\rho u)(\Phi-\Phi_s)dx\nonumber\\
&=-\int\rho u\cdot\rho_s^{-1}\nabla P_sdx+\int \rho u\cdot\nabla\Phi_sdx+\int\Delta(\Phi-\Phi_s)_t(\Phi-\Phi_s)dx\nonumber\\
&=-\frac12\frac{d}{dt}\int|\nabla(\Phi-\Phi_s)|^2dx,
\end{align}
where we have used
\begin{align*}
\int w\cdot{\rm curl}\,udx=\int u\cdot{\rm curl}\,wdx,\
G'(\rho)=Q(\rho)-Q(\rho_s),\ Q'(\rho)=P'(\rho)/\rho.
\end{align*}
Thus, integrating \eqref{3.12} over $[0, T]$ yields \eqref{3.8}.

2. It is easy to check that there exists a positive constant $C$ which may depend only on $a$, $\gamma$, and $\hat{\rho}$ such that
\begin{align*}
&|P-P_s|\le C|\rho-\rho_s|, \quad C^{-1}(\rho-\rho_s)^2\le G(\rho)\le C(\rho-\rho_s)^2,\\
&\|w\|_{L^2}^2=\frac14\|{\rm curl}\,u-2w-{\rm curl}\,u\|_{L^2}^2
\leq C\|\curl u-2w\|_{L^2}^2+C\|\curl u\|_{L^2}^2,
\end{align*}
which together with \eqref{2.8}, \eqref{2.9}, and \eqref{3.8} gives \eqref{3.9}.
\end{proof}

\begin{lemma}\label{l33}
Let $(\rho, u, w, b, \Phi)$ be a smooth solution
of \eqref{a1}--\eqref{a6} satisfying \eqref{3.6}. Assume that
$\eta(t)\ge 0$ is a piecewise differentiable function, then it holds that
\begin{align}\label{3.15}
&\frac{d}{dt}\Big(\frac{2\mu+\lambda}{2}\eta(t)\|{\rm div}\,u\|_{L^2}^2
+\frac{\mu}{2}\eta(t)\|{\rm curl}\,u\|_{L^2}^2+2\mu_r\eta(t)\Big\|\frac12{\rm curl}\,u-w\Big\|_{L^2}^2
+\frac{2c_d+c_0}{2}\eta(t)\|{\rm div}\,w\|_{L^2}^2\Big)\nonumber\\
&\quad+\frac12\eta(t)\big(\|\sqrt{\rho}\dot{u}\|_{L^2}^2
+\|\sqrt{\rho}\dot{w}\|_{L^2}^2\big)
+\eta(t)\|{\rm curl}^2b\|_{L^2}^2+\eta(t)\|b_t\|_{L^2}^2\nonumber\\
&\le \frac{d}{dt}\int\eta(t)(P-P_s)\divv udx
+\frac{d}{dt}\int\eta(t)(\rho-\rho_s)u\cdot\nabla\Phi_sdx\nonumber\\
&\quad+\frac{d}{dt}\int\eta(t)\Big(b\cdot\nabla b-\frac12\nabla|b|^2\Big)\cdot udx+\Big(CE_0^\frac12+\delta\Big)\eta(t)\|{\rm curl}^2b\|_{L^2}^2
+\delta\eta(t)\|b_t\|_{L^2}^2\nonumber\\
&\quad+C\big(\eta(t)+|\eta'(t)|\big)\|\nabla u\|_{L^2}^2+C|\eta'(t)|\big(\|w\|_{L^2}^2+\|\nabla w\|_{L^2}^2+\|\nabla b\|_{L^2}^2+\|\nabla u\|_{L^2}^4\big)\nonumber\\
&\quad+C\eta(t)\big(\|\nabla u\|_{L^3}^3+\|\nabla w\|_{L^3}^3+\|\nabla u\|_{L^2}^4+\|\nabla b\|_{L^2}^4+\|w\|_{H^1}^4+\|\nabla u\|_{L^2}^6\big)\nonumber\\
&\quad+C\eta(t)\big(\|\nabla b\|_{L^2}^6+\|\nabla w\|_{L^2}^6+\|\nabla w\|_{L^2}^2+\|\nabla b\|_{L^2}^2\big)+C\big(\eta(t)+|\eta'(t)|\big)E_0.
\end{align}
Moreover, there exists a positive constant $\varepsilon_1$ such that
\begin{align}\label{3.16}
\sup_{0\le t\le T}\|\nabla b\|_{L^2}^2+\int_0^T\big(\|{\rm curl}^2b\|_{L^2}^2+\|b_t\|_{L^2}^2\big)dt\le C\|\nabla b_0\|_{L^2}^2
\end{align}
provided that $E_0\leq\varepsilon_1$.
\end{lemma}
\begin{proof}[Proof]
1. Multiplying $\eqref{3.10}$ by $\eta(t)\dot{u}$ and $\eqref{3.11}$ by $\eta(t)\dot{w}$, respectively, summing up, and integrating the resulting equality over $\Omega$ leads to
\begin{align}\label{3.17}
\int\eta(t)\big(\rho|\dot{u}|^2+\rho|\dot{w}|^2\big)dx
&=-\int\eta(t)\dot{u}\cdot\nabla(P-P_s)dx+(2\mu+\lambda)\int\eta(t)\nabla{\rm div}\,u\cdot\dot{u}dx\nonumber\\
&\quad-(\mu+\mu_r)\int\eta(t)\curl{\rm curl}\,u\cdot\dot{u}dx+2\mu_r\int\eta(t)\curl w\cdot\dot{u}dx\nonumber\\
&\quad+(2c_d+c_0)\int\eta(t)\nabla{\rm div}\,w\cdot\dot{w}dx-(c_d+c_a)\int\eta(t)\curl\curl w\cdot\dot{w}dx\nonumber\\
&\quad+2\mu_r\int\eta(t){\rm curl}\,u\cdot\dot{w}dx-4\mu_r\int\eta(t)\dot{w}\cdot wdx\nonumber\\
&\quad+\int\eta(t)\rho\dot{u}\cdot\nabla(\Phi-\Phi_s)dx
+\int\eta(t)(\rho-\rho_s)\dot{u}\cdot\nabla\Phi_sdx\nonumber\\
&\quad+\int\eta(t)\Big(b\cdot\nabla b-\frac12\nabla|b|^2\Big)\cdot\dot{u}dx\triangleq\sum_{i=1}^{11}I_i.
\end{align}
By $\eqref{a1}_1$ and $P=a\rho^\gamma$, we have
%\begin{align}
%(P-P_s)_t+u\cdot\nabla(P-P_s)+\gamma P{\rm div}\,u+u\cdot\nabla P_s=0,
%\end{align}
%or
\begin{align}\label{z3.13}
(P-P_s)_t+\divv(Pu)+(\gamma-1)P\divv u=0,
\end{align}
which together with integration by parts and \eqref{3.6} shows that
\begin{align}\label{3.20}
I_1&=-\int\eta(t)u_t\cdot\nabla(P-P_s)dx-\int\eta(t)u\cdot\nabla u\cdot\nabla (P-P_s)dx\nonumber\\
&=\frac{d}{dt}\int\eta(t)(P-P_s){\rm div}\,udx-\eta'(t)\int(P-P_s){\rm div}\,udx-\int\eta(t){\rm div}\,u(P-P_s)_tdx\nonumber\\
&\quad-\int\eta(t)u\cdot\nabla u\cdot\nabla (P-P_s)dx\nonumber\\
&=\frac{d}{dt}\int\eta(t)(P-P_s){\rm div}\,udx-\eta'(t)\int(P-P_s){\rm div}\,udx+\int\eta(t){\rm div}\,u{\rm div}\,(Pu)dx\nonumber\\
&\quad+(\gamma-1)\int\eta(t)P({\rm div}\,u)^2dx-\int\eta(t)u\cdot\nabla u\cdot\nabla (P-P_s)dx\nonumber\\
&=\frac{d}{dt}\int\eta(t)(P-P_s){\rm div}\,udx-\eta'(t)\int(P-P_s){\rm div}\,udx
+\int\eta(t)P\nabla u:\nabla udx\nonumber\\
&\quad+(\gamma-1)\int\eta(t)P({\rm div}\,u)^2dx-\int_{\partial\Omega}\eta(t)Pu\cdot\nabla u\cdot ndS
+\int\eta(t)u\cdot\nabla u\cdot\nabla P_sdx\nonumber\\
&\le \frac{d}{dt}\int\eta(t)(P-P_s){\rm div}\,udx+C\eta(t)\|\nabla u\|_{L^2}^2
+|\eta'(t)|\|P-P_s\|_{L^2}\|\nabla u\|_{L^2}\nonumber\\
&\le \frac{d}{dt}\int\eta(t)(P-P_s){\rm div}\,udx+C(\eta(t)+|\eta'(t)|)\|\nabla u\|_{L^2}^2+C|\eta'(t)|\|P-P_s\|_{L^2}^2\nonumber\\
&\le \frac{d}{dt}\int\eta(t)(P-P_s){\rm div}\,udx+C(\eta(t)+|\eta'(t)|)\|\nabla u\|_{L^2}^2+C|\eta'(t)|E_0,
\end{align}
where we have used
\begin{align*}
&\int\eta(t){\rm div}\,(Pu){\rm div}\,udx \\
&=-\int\eta(t) Pu^j\partial_{ji}u^idx
=\int\eta(t)\partial_i(Pu^j)\partial_ju^idx\nonumber\\
&=-\int_{\partial\Omega}\eta(t)Pu\cdot\nabla u\cdot ndS
+\int\eta(t)\partial_iPu^j\partial_ju^idx+\int\eta(t)P\partial_iu^j\partial_ju^idx\nonumber\\
&=-\int_{\partial\Omega}\eta(t)Pu\cdot\nabla u\cdot ndS+\int\eta(t) u\cdot\nabla u\cdot\nabla Pdx+\int\eta(t) P\nabla u:\nabla udx,
\end{align*}
and
\begin{align}\label{z3.21}
-\int_{\partial\Omega}\eta(t)Pu\cdot\nabla u\cdot ndS&=\int_{\partial\Omega}\eta(t)Pu\cdot\nabla n\cdot udS
\le C\eta(t)\int_{\partial\Omega}|u|^2dS\le C\eta(t)\|\nabla u\|_{L^2}^2,
\end{align}
due to \eqref{2.5}, \eqref{3.6}, Lemma \ref{l23}, and \eqref{z2.5}. Here and in what follows, we use the Einstein convention that the repeated indices denote the summation.

2. By \eqref{a6} and \eqref{2.5}, we derive from integration by parts that
\begin{align*}
I_2&=(2\mu+\lambda)\int_{\partial\Omega}\eta(t){\rm div}\,u(\dot{u}\cdot n)dS-(2\mu+\lambda)\int\eta(t){\rm div}\,u{\rm div}\,\dot{u}dx\nonumber\\
&=(2\mu+\lambda)\int_{\partial\Omega}\eta(t){\rm div}\,u(u\cdot\nabla u\cdot n)dS-\frac{2\mu+\lambda}{2}\frac{d}{dt}\int\eta(t)({\rm div}\,u)^2dx\nonumber\\
&\quad-(2\mu+\lambda)\int\eta(t){\rm div}\,u{\rm div}\,(u\cdot\nabla u)dx+\frac{2\mu+\lambda}{2}\eta'(t)\int({\rm div}u)^2dx\nonumber\\
&=-\frac{2\mu+\lambda}{2}\frac{d}{dt}\int\eta(t)(\divv u)^2dx
-(2\mu+\lambda)\int_{\partial\Omega}\eta(t)\divv u(u\cdot\nabla n\cdot u)dS\nonumber\\
&\quad-(2\mu+\lambda)\int\eta(t){\rm div}\,u\partial_i(u^j\partial_ju^i)dx+\frac{2\mu+\lambda}{2}\eta'(t)\int({\rm div}\,u)^2dx\nonumber\\
&=-\frac{2\mu+\lambda}{2}\frac{d}{dt}\int\eta(t)({\rm div}\,u)^2dx
-(2\mu+\lambda)\int_{\partial\Omega}\eta(t){\rm div}\,u(u\cdot\nabla n\cdot u)dS\nonumber\\
&\quad-(2\mu+\lambda)\int\eta(t){\rm div}\,u\nabla u:\nabla udx-(2\mu+\lambda)\int\eta(t){\rm div}\,uu^j\partial_{ji}u^idx+\frac{2\mu+\lambda}{2}\eta'(t)\int({\rm div}\,u)^2dx\nonumber\\
&=-\frac{2\mu+\lambda}{2}\frac{d}{dt}\int\eta(t)({\rm div}\,u)^2dx
-(2\mu+\lambda)\int_{\partial\Omega}\eta(t){\rm div}\,u(u\cdot\nabla n\cdot u)dS\nonumber\\
&\quad-(2\mu+\lambda)\int\eta(t){\rm div}\,u\nabla u:\nabla udx+\frac{2\mu+\lambda}{2}\int\eta(t)({\rm div}\,u)^3dx
+\frac{2\mu+\lambda}{2}\eta'(t)\int({\rm div}\,u)^2dx\nonumber\\
&\le -\frac{2\mu+\lambda}{2}\frac{d}{dt}\int\eta(t)({\rm div}\,u)^2dx
+\frac14\eta(t)\|\sqrt{\rho}\dot{u}\|_{L^2}^2+CE_0^\frac12\eta(t)\|{\rm curl}^2b\|_{L^2}^2\nonumber\\
&\quad+C|\eta'(t)|\|\nabla u\|_{L^2}^2+C\eta(t)\big(\|\nabla u\|_{L^3}^3+\|\nabla u\|_{L^2}^4+\|\nabla b\|_{L^2}^4+\|\nabla u\|_{L^2}^2\big)\nonumber\\
&\quad+C\eta(t)\big(\|\nabla w\|_{L^2}^2+\|\nabla b\|_{L^2}^2\big),
\end{align*}
where we have used
\begin{align*}
\int\eta(t){\rm div}\,uu^j\partial_{ji}u^idx
&=-\int\eta(t)\partial_j(\partial_ku^ku^j)\partial_iu^idx\nonumber\\
&=-\int\eta(t)\partial_{jk}u^ku^j\partial_iu^idx
-\int\eta(t)\divv u\partial_ju^j\partial_iu^idx\nonumber\\
&=-\int\eta(t)\partial_{ji}u^iu^j\divv udx
-\int\eta(t)({\rm div}\,u)^3dx,
\end{align*}
and
\begin{align}\label{3.24}
&\Big|-(2\mu+\lambda)\int_{\partial\Omega}{\rm div}\,u(u\cdot\nabla n\cdot u)dS\Big|\nonumber\\
&=\Big|-\int_{\partial\Omega}\Big(F_1+(P-P_s)+\frac12|b|^2\Big)(u\cdot\nabla n\cdot u)dS\Big|\nonumber\\
&\le \Big|\int_{\partial\Omega}F_1(u\cdot\nabla n\cdot u)dS\Big|+\Big|\int_{\partial\Omega}(P-P_s)(u\cdot\nabla n\cdot u)dS\Big|
+\frac12\Big|\int_{\partial\Omega}|b|^2(u\cdot\nabla n\cdot u)dS\Big|\nonumber\\
&\le C\int_{\partial\Omega}|F_1||u|^2dS+C\int_{\partial\Omega}|u|^2dS
+C\int_{\partial\Omega}|b|^2|u|^2dS
\nonumber\\
&\le C\big(\|\nabla F_1\|_{L^2}\|u\|_{L^4}^2+\|F_1\|_{L^6}\|u\|_{L^3}\|\nabla u\|_{L^2}
+\|F_1\|_{L^2}\|u\|_{L^4}^2\big)+C\|\nabla u\|_{L^2}^2\nonumber\\
&\quad+C\big(\|\nabla b\|_{L^2}\|b\|_{L^6}\|u\|_{L^6}^2+\|b\|_{L^6}^2\|u\|_{L^6}\|\nabla u\|_{L^2}+\|b\|_{L^4}^2\|u\|_{L^4}^2\big)\nonumber\\
&\le C\|F_1\|_{H^1}\|u\|_{H^1}^2+C\|\nabla u\|_{L^2}^2+C\|\nabla u\|_{L^2}^4+C\|\nabla b\|_{L^2}^4\nonumber\\
&\le \frac{1}{2}\|\sqrt{\rho}\dot{u}\|_{L^2}^2+\delta\|{\rm curl}^2b\|_{L^2}^2
+C\big(\|\nabla u\|_{L^2}^4+\|\nabla b\|_{L^2}^4+\|\nabla u\|_{L^2}^2+\|\nabla w\|_{L^2}^2+\|\nabla b\|_{L^2}^2\big),
\end{align}
due to Lemma \ref{l23}, \eqref{2.3}, Lemma \ref{l28}, and
\begin{align*}
\|b\cdot\nabla b\|_{L^2}
&\le C\|b\|_{L^6}\|\nabla b\|_{L^3} \\
&\le C\|b\|_{H^1}\|\nabla b\|_{L^2}^\frac12\|\nabla b\|_{L^6}^\frac12 \\
& \le C\|\nabla b\|_{L^2}^\frac32\|\nabla b\|_{L^6}^\frac12\ \ (\text{by Lemma } \ref{l21}) \\
&\le C\|\nabla b\|_{L^2}^\frac32\big(\|{\rm curl}^2b\|_{L^2}+\|\nabla b\|_{L^2}\big)^\frac12\ \ (\text{by Lemma } \ref{l25}) \\
&\le C\|\nabla b\|_{L^2}^\frac32\|{\rm curl}^2b\|_{L^2}^\frac12+\|\nabla b\|_{L^2}^2.
\end{align*}

3. Noting that
\begin{align*}
\int{\rm curl}\,u\cdot(u^i\partial_i{\rm curl}\,u)dx=-\int{\rm curl}\,u\cdot(u^i\partial_i{\rm curl}\,u)dx-\int|{\rm curl}\,u|^2{\rm div}\,udx.
\end{align*}
Thus, we have
\begin{align*}
\int{\rm curl}\,u\cdot(u^i\partial_i{\rm curl}\,u)dx
=-\frac12\int|{\rm curl}\,u|^2{\rm div}\,udx.
\end{align*}
This implies that
\begin{align*}
&(\mu+\mu_r)\int{\rm curl}\,u\cdot{\rm curl}\,(u\cdot\nabla u)dx\nonumber\\
&=(\mu+\mu_r)\int{\rm curl}\,u\cdot{\rm curl}\,(u^i\partial_i u)dx\nonumber\\
&=(\mu+\mu_r)\int{\rm curl}\,u\cdot\big(u^i{\rm curl}\,\partial_iu+\nabla u^i\times\partial_iu\big)dx\nonumber\\
&=-(\mu+\mu_r)\int\partial_i({\rm curl}\,u u^i){\rm curl}\,udx+(\mu+\mu_r)\int(\nabla u^i\times\partial_iu)\cdot{\rm curl}\,udx\nonumber\\
&=-\frac{\mu+\mu_r}{2}\int|{\rm curl}\,u|^2{\rm div}\,udx+(\mu+\mu_r)\int(\nabla u^i\times\partial_iu)\cdot{\rm curl}\,udx,
\end{align*}
which combined with \eqref{a6} and integration by parts leads to
\begin{align}\label{3.26}
I_3&=-(\mu+\mu_r)\int\eta(t){\rm curl}\,u\cdot{\rm curl}\,\dot{u}dx\nonumber\\
&=-\frac{\mu+\mu_r}{2}\frac{d}{dt}\int\eta(t)|{\rm curl}\,u|^2dx
+\frac{\mu+\mu_r}{2}\eta'(t)\int|{\rm curl}\,u|^2dx\nonumber\\
&\quad-(\mu+\mu_r)\int\eta(t){\rm curl}\,u\cdot{\rm curl}\,(u\cdot\nabla u)dx\nonumber\\
&=-\frac{\mu+\mu_r}{2}\frac{d}{dt}\int\eta(t)|{\rm curl}\,u|^2dx
+\frac{\mu+\mu_r}{2}\eta'(t)\int|{\rm curl}\,u|^2dx\nonumber\\
&\quad-(\mu+\mu_r)\int\eta(t)(\nabla u^i\times\partial_iu)\cdot{\rm curl}\,udx
+\frac{\mu+\mu_r}{2}\int\eta(t)|{\rm curl}\,u|^2{\rm div}\,udx\nonumber\\
&\le -\frac{\mu+\mu_r}{2}\frac{d}{dt}\int\eta(t)|{\rm curl}\,u|^2dx
+C|\eta'(t)|\|\nabla u\|_{L^2}^2+C\eta(t)\|\nabla u\|_{L^3}^3.
\end{align}

4. Integration by parts together with H{\"o}lder's inequality and \eqref{2.3} gives that
\begin{align*}
I_4+I_7&=2\mu_r\int\eta(t)\curl u_t\cdot wdx+2\mu_r\int\eta(t)u\cdot\nabla u\cdot {\rm curl}\,wdx\nonumber\\
&\quad+2\mu_r\int\eta(t)\curl u\cdot w_tdx+2\mu_r\int\eta(t)u\cdot\nabla w\cdot {\rm curl}\,udx\nonumber\\
&=2\mu_r\frac{d}{dt}\int\eta(t)\curl u\cdot wdx-2\mu_r\eta'(t)\int\curl u\cdot wdx
\nonumber\\
&\quad+2\mu_r\int\eta(t)u\cdot\nabla u\cdot{\rm curl}\,wdx+2\mu_r\int\eta(t)u\cdot\nabla w\cdot {\rm curl}\,udx\nonumber\\
&\leq 2\mu_r\frac{d}{dt}\int\eta(t)\curl u\cdot wdx+C|\eta'(t)|\|{\rm curl}\,u\|_{L^2}\|w\|_{L^2}\nonumber\\
&\quad+C\eta(t)\|u\|_{L^6}\|\nabla u\|_{L^3}\|{\rm curl}\,w\|_{L^2}+C\eta(t)\|u\|_{L^6}\|\nabla w\|_{L^3}\|{\rm curl}\,u\|_{L^2}\nonumber\\
&\le 2\mu_r\frac{d}{dt}\int\eta(t)\curl u\cdot wdx+C|\eta'(t)|\big(\|\nabla u\|_{L^2}^2+\|w\|_{L^2}^2\big)\nonumber\\
&\quad+C\eta(t)\big(\|\nabla w\|_{L^2}^2+\|\nabla u\|_{L^2}^2+\|\nabla u\|_{L^2}^6+\|\nabla w\|_{L^2}^6+\|\nabla u\|_{L^3}^3
+\|\nabla w\|_{L^3}^3\big).
\end{align*}

5. By \eqref{a6} and \eqref{2.5}, one has
\begin{align*}
I_5&=(2c_d+c_0)\int_{\partial\Omega}\eta(t){\rm div}\,w(\dot{w}\cdot n)dS-(2c_d+c_0)\int\eta(t)\divv w\divv\dot{w}dx\nonumber\\
&=-(2c_d+c_0)\int_{\partial\Omega}\eta(t){\rm div}\,w(u\cdot\nabla n\cdot w)dS
-\frac{2c_d+c_0}{2}\frac{d}{dt}\int\eta(t)({\rm div}\,w)^2dx\nonumber\\
&\quad+\frac{2c_d+c_0}{2}\eta'(t)\int({\rm div}\,w)^2dx-(2c_d+c_0)\int\eta(t){\rm div}\,w\partial_i(u^j\partial_jw^i)dx\nonumber\\
&=-(2c_d+c_0)\int_{\partial\Omega}\eta(t){\rm div}\,w(u\cdot\nabla n\cdot w)dS
-\frac{2c_d+c_0}{2}\frac{d}{dt}\int\eta(t)({\rm div}\,w)^2dx\nonumber\\
&\quad+\frac{2c_d+c_0}{2}\eta'(t)\int({\rm div}\,w)^2dx-(2c_d+c_0)\int\eta(t){\rm div}\,w\partial_iu^j\partial_jw^idx\nonumber\\
&\quad+\frac{2c_d+c_0}{2}\int\eta(t)({\rm div}\,w)^2{\rm div}\,udx\nonumber\\
&\le -\frac{2c_d+c_0}{2}\frac{d}{dt}\int\eta(t)({\rm div}\,w)^2dx+C|\eta'(t)|\|\nabla w\|_{L^2}^2
+C\eta(t)\|\nabla u\|_{L^3}^3\nonumber\\
&\quad+C\eta(t)\|\nabla w\|_{L^3}^3+\frac14\eta(t)\|\sqrt{\rho}\dot{w}\|_{L^2}^2+C\eta(t)
\big(\|\nabla u\|_{L^2}^2+\|w\|_{H^1}^2+\|w\|_{H^1}^4\big),
\end{align*}
where we have used
\begin{align}\label{z3.19}
&\Big|-(2c_d+c_0)\int_{\partial\Omega}{\rm div}\,w(u\cdot\nabla n\cdot w)dS\Big|\nonumber\\
&=\Big|-\int_{\partial\Omega}F_2(u\cdot\nabla n\cdot w)dS\Big| \\
& \le C\int_{\partial\Omega}|F_2||u||w|dS\nonumber\\
& \le C\||F_2||u||w|\|_{W^{1,1}}\\
&\le C\big(\|F_2\|_{L^2}\|u\|_{L^4}\|w\|_{L^4}+\|\nabla F_1\|_{L^2}\|u\|_{L^4}\|w\|_{L^4}+\|F_2\|_{L^6}\|\nabla u\|_{L^2}\|w\|_{L^3}+\|F_2\|_{L^6}\|u\|_{L^3}\|\nabla w\|_{L^2}\big)\nonumber\\
&\le C\|F_2\|_{H^1}\|u\|_{H^1}\|w\|_{H^1}\nonumber\\
&\le C\big(\|\rho\dot{w}\|_{L^2}+\|\nabla u\|_{L^2}+\|w\|_{H^1}\big)\|\nabla u\|_{L^2}\|w\|_{H^1}\nonumber\\
&\le \frac12\|\sqrt{\rho}\dot{w}\|_{L^2}^2+C\big(\|\nabla u\|_{L^2}^2+\|\nabla u\|_{L^2}^4+\|w\|_{H^1}^2+\|w\|_{H^1}^4\big),
\end{align}
due to Lemma \ref{l23}, H{\"o}lder's inequality, \eqref{z2.5}, Lemma \ref{l28}, and \eqref{3.6}.

6. Similarly to \eqref{3.26}, we deduce that
\begin{align*}
I_6&=-(c_d+c_a)\int\eta(t){\rm curl}\,w\cdot{\rm curl}\,\dot{w}dx\nonumber\\
&=-\frac{c_d+c_a}{2}\frac{d}{dt}\int\eta(t)|{\rm curl}\,w|^2dx
+\frac{c_d+c_a}{2}\eta'(t)\int|{\rm curl}\,w|^2dx\nonumber\\
&\quad-(c_d+c_a)\int\eta(t){\rm curl}\,w\cdot{\rm curl}\,(u\cdot\nabla w)dx\nonumber\\
&=-\frac{c_d+c_a}{2}\frac{d}{dt}\int\eta(t)|{\rm curl}\,w|^2dx
+\frac{c_d+c_a}{2}\eta'(t)\int|{\rm curl}\,u|^2dx\nonumber\\
&\quad-(c_d+c_a)\int\eta(t)(\nabla u^i\times\partial_iw)\cdot{\rm curl}\,wdx
+\frac{c_d+c_a}{2}\int\eta(t)|{\rm curl}\,w|^2{\rm div}\,udx\nonumber\\
&\le -\frac{c_d+c_a}{2}\frac{d}{dt}\int\eta(t)|{\rm curl}\,w|^2dx
+C|\eta'(t)|\|\nabla w\|_{L^2}^2
+C\eta(t)\big(\|\nabla u\|_{L^3}^3+\|\nabla w\|_{L^3}^3\big).
\end{align*}
By H{\"o}lder's inequality, Sobolev's inequality, and Lemma \ref{l25}, we find that
\begin{align*}
I_8&=-4\mu_r\int\eta(t)w_t\cdot wdx-4\mu_r\int\eta(t)u\cdot\nabla w\cdot wdx\nonumber\\
&=-2\mu_r\frac{d}{dt}\int\eta(t)|w|^2dx+2\mu_r\eta'(t)\int|w|^2dx-4\mu_r\int\eta(t)u\cdot\nabla w\cdot wdx\nonumber\\
&\le -2\mu_r\frac{d}{dt}\int\eta(t)|w|^2dx+2\mu_r|\eta'(t)|\|w\|_{L^2}^2+C\eta(t)\|u\|_{L^6}\|\nabla w\|_{L^3}\|w\|_{L^2}\nonumber\\
&\le -2\mu_r\frac{d}{dt}\int\eta(t)|w|^2dx+C|\eta'(t)|\|w\|_{L^2}^2+C\eta(t)\big(\|\nabla u\|_{L^2}^3+\|\nabla w\|_{L^3}^3\big).
\end{align*}
It follows from \eqref{3.6} and \eqref{3.8} that
\begin{align*}
I_9\le C(\hat{\rho})\eta(t)\|\sqrt{\rho}\dot{u}\|_{L^2}\|\nabla(\Phi-\Phi_s)\|_{L^2}
\le \frac14\eta(t)\|\sqrt{\rho}\dot{u}\|_{L^2}^2+C\eta(t)E_0.
\end{align*}
Similarly to the derivation of \eqref{3.20}, we get from \eqref{3.9} and Lemma \ref{l29} that
\begin{align*}
I_{10}&=\frac{d}{dt}\int\eta(t)(\rho-\rho_s)u\cdot\nabla\Phi_sdx-\eta'(t)\int(\rho-\rho_s)u\cdot\nabla\Phi_sdx\nonumber\\
&\quad-\eta(t)\int(\rho-\rho_s) u\cdot\nabla^2(\Phi-\Phi_s)\cdot udx+\eta(t)\int{\rm div}\,(\rho_su)u\cdot\nabla\Phi_sdx\nonumber\\
&\le \frac{d}{dt}\int\eta(t)(\rho-\rho_s)u\cdot\nabla\Phi_sdx+C|\eta'(t)|\|\nabla u\|_{L^2}\|\rho-\rho_s\|_{L^2}\|\nabla\Phi_s\|_{L^3}\nonumber\\
&\quad+C\eta(t)\big(\|\rho-\rho_s\|_{L^6}+\|\nabla\rho_s\|_{L^2}
+\|\rho_s\|_{L^\infty}\big)\|\nabla u\|_{L^2}^2\|\nabla\Phi_s\|_{H^1}\nonumber\\
&\le \frac{d}{dt}\int\eta(t)(\rho-\rho_s)u\cdot\nabla\Phi_sdx
+C(\hat{\rho})\big(\eta(t)+|\eta'(t)|\big)\|\nabla u\|_{L^2}^2+C(\hat{\rho})|\eta'(t)|E_0,
\end{align*}
where we have used
\begin{align*}
(\rho-\rho_s)_t+{\rm div}((\rho-\rho_s)u)+{\rm div}(\rho_su)=0.
\end{align*}

7. By \eqref{a6} and \eqref{2.23}, it indicates that
\begin{align*}
I_{11}&=\int\eta(t)\Big(b\cdot\nabla b-\frac12\nabla|b|^2\Big)\cdot u_tdx+\int\eta(t)\Big(b\cdot\nabla b-\frac12\nabla|b|^2\Big)\cdot u\cdot\nabla udx\nonumber\\
&=\frac{d}{dt}\int\eta(t)\Big(b\cdot\nabla b-\frac12\nabla|b|^2\Big)\cdot udx
-\eta'(t)\int\Big(b\otimes b:\nabla u-\frac12\nabla|b|^2{\rm div}\,u\Big)dx\nonumber\\
&\quad+\eta(t)\int\Big((b\otimes b)_t:\nabla u-\frac12(|b|^2)_t{\rm div}\,u\Big)dx
+\eta(t)\int\Big(b\cdot\nabla b-\frac12\nabla|b|^2\Big)\cdot u\cdot\nabla udx\nonumber\\
&\le \frac{d}{dt}\int\eta(t)\Big(b\cdot\nabla b-\frac12\nabla|b|^2\Big)\cdot udx
+C|\eta'(t)|\|b\|_{L^4}^2\|\nabla u\|_{L^2}\nonumber\\
&\quad+C\eta(t)\|\nabla u\|_{L^3}\|b_t\|_{L^2}\|b\|_{L^6}+C\eta(t)\|b\|_{L^6}\big(\|{\rm curl}^2b\|_{L^2}+\|\nabla b\|_{L^2}\big)\|\nabla u\|_{L^2}\|u\|_{L^6}\nonumber\\
&\le \frac{d}{dt}\int\eta(t)\Big(b\cdot\nabla b-\frac12\nabla|b|^2\Big)\cdot udx
+C|\eta'(t)|\big(\|\nabla b\|_{L^2}^2+\|\nabla u\|_{L^2}^4\big)\nonumber\\
&\quad+C\delta\eta(t)(\|b_t\|_{L^2}^2+\|{\rm curl}^2b\|_{L^2}^2)+C\eta(t)\big(\|\nabla b\|_{L^2}^2\|\nabla u\|_{L^2}^2+\|\nabla u\|_{L^3}^3\big)\nonumber\\
&\quad+C\eta(t)\big(\|\nabla b\|_{L^2}^6+\|\nabla b\|_{L^2}^2\|\nabla u\|_{L^2}^4\big).
\end{align*}
Putting the above estimates on $I_i\ (i=1, 2, \cdots, 11)$ into \eqref{3.17}, one obtains that
\begin{align}\label{2.35}
&\frac{d}{dt}\Big(\frac{2\mu+\lambda}{2}\eta(t)\|{\rm div}\,u\|_{L^2}^2
+\frac{\mu}{2}\eta(t)\|{\rm curl}\,u\|_{L^2}^2+2\mu_r\eta(t)\Big\|\frac12{\rm curl}\,u-w\Big\|_{L^2}^2
+\frac{2c_d+c_0}{2}\eta(t)\|{\rm div}\,w\|_{L^2}^2\Big)\nonumber\\
&\quad+\frac12\eta(t)\big(\|\sqrt{\rho}\dot{u}\|_{L^2}^2
+\|\sqrt{\rho}\dot{w}\|_{L^2}^2\big)\nonumber\\
&\le \frac{d}{dt}\int\eta(t)(P-P_s){\rm div}\,udx
+\frac{d}{dt}\int\eta(t)(\rho-\rho_s)u\cdot\nabla\Phi_sdx\nonumber\\
&\quad+\frac{d}{dt}\int\eta(t)\Big(b\cdot\nabla b-\frac12\nabla|b|^2\Big)\cdot udx+\Big(CE_0^\frac12+\delta\Big)\eta(t)\|{\rm curl}^2b\|_{L^2}^2
+\delta\eta(t)\|b_t\|_{L^2}^2\nonumber\\
&\quad+C\big(\eta(t)+|\eta'(t)|\big)\|\nabla u\|_{L^2}^2+C|\eta'(t)|\big(\|w\|_{L^2}^2+\|\nabla w\|_{L^2}^2+\|\nabla b\|_{L^2}^2+\|\nabla u\|_{L^2}^4\big)\nonumber\\
&\quad+C\eta(t)\big(\|\nabla u\|_{L^3}^3+\|\nabla w\|_{L^3}^3+\|\nabla u\|_{L^2}^4+\|\nabla b\|_{L^2}^4+\|w\|_{H^1}^4+\|\nabla u\|_{L^2}^6\big)\nonumber\\
&\quad+C\eta(t)\big(\|\nabla b\|_{L^2}^6+\|\nabla w\|_{L^2}^6+\|\nabla w\|_{L^2}^2+\|\nabla b\|_{L^2}^2\big)+C\big(\eta(t)+|\eta'(t)|\big)E_0.
\end{align}

8. In view of $\eqref{a1}_4$, one easily deduces from \eqref{2.3}, \eqref{a6}, and \eqref{2.10} that
\begin{align*}
&\frac{d}{dt}\|{\rm curl}\,b\|_{L^2}^2+\|{\rm curl}^2b\|_{L^2}^2+\|b_t\|_{L^2}^2\nonumber\\
&=\int|b_t-\curl\curl b|^2dx=\int|b\cdot\nabla u-u\cdot\nabla b-b\divv u|^2dx\nonumber\\
&\le C\|\nabla u\|_{L^2}^2\|b\|_{L^\infty}^2+C\|u\|_{L^6}^2\|\nabla b\|_{L^3}^2\nonumber\\
&\le C\|\nabla u\|_{L^2}^2\|\nabla b\|_{L^2}\|{\rm curl}^2b\|_{L^2}+C\|\nabla u\|_{L^2}^2\|\nabla b\|_{L^2}^2\nonumber\\
&\le \frac12\|{\rm curl}^2b\|_{L^2}^2+C\big(\|\nabla u\|_{L^2}^2+\|\nabla u\|_{L^2}^4\big)\|\nabla b\|_{L^2}^2,
\end{align*}
which yields that
\begin{align}\label{3.37}
&\frac{d}{dt}\big(\eta(t)\|{\rm curl}\,b\|_{L^2}^2\big)+\eta(t)\|{\rm curl}^2b\|_{L^2}^2+\eta(t)\|b_t\|_{L^2}^2\nonumber\\
&\le \eta'(t)\|{\rm curl}\,b\|_{L^2}^2+C\eta(t)\big(\|\nabla u\|_{L^2}^2+\|\nabla u\|_{L^2}^4\big)\|\nabla b\|_{L^2}^2.
\end{align}
This together with \eqref{2.35} gives
\eqref{3.15}.

9. Choosing $\eta(t)=1$, then we deduce from \eqref{3.37} that
\begin{align}\label{z3.10}
&\sup_{0\le t\le T}\|\nabla b\|_{L^2}^2+\int_0^T\big(\|b_t\|_{L^2}^2+\|{\rm curl}^2b\|_{L^2}^2\big)dt\nonumber\\
&\le C\|\nabla b_0\|_{L^2}^2+C\sup_{0\le t\le T}\|\nabla b\|_{L^2}^2\int_0^T\big(\|\nabla u\|_{L^2}^2+\|\nabla u\|_{L^2}^4\big)dt.
\end{align}
It follows from \eqref{3.6} and \eqref{3.8} that
\begin{align*}
\int_0^T\|\nabla u\|_{L^2}^4dt&=\int_0^{\sigma(T)}\|\nabla u\|_{L^2}^4dt+\int_{\sigma(T)}^T\|\nabla u\|_{L^2}^4dt\nonumber\\
&\le \sup_{0\le t\le \sigma(T)}\|\nabla u\|_{L^2}^2\int_0^{\sigma(T)}\|\nabla u\|_{L^2}^2dt
+\sup_{\sigma(T)\le t\le T}\big(\sigma\|\nabla u\|_{L^2}^2\big)\int_{\sigma(T)}^T\|\nabla u\|_{L^2}^2dt\nonumber\\
&\le C(K)E_0,
\end{align*}
which combined with \eqref{z3.10} and \eqref{3.8} yields that
\begin{align*}
\sup_{0\le t\le T}\|\nabla b\|_{L^2}^2+\int_0^T\big(\|b_t\|_{L^2}^2+\|{\rm curl}^2b\|_{L^2}^2\big)dt\le C\|\nabla b_0\|_{L^2}^2
+C_1E_0\sup_{0\le t\le T}\|\nabla b\|_{L^2}^2.
\end{align*}
This implies \eqref{3.16} provided that $E_0<\frac{1}{2C_1}$.
\end{proof}

\begin{lemma}
Let $(\rho, u, w, b, \Phi)$ be a smooth solution
of \eqref{a1}--\eqref{a6} satisfying \eqref{3.6} and $\eta(t)$ be as in Lemma \ref{l33}, then it holds that
\begin{align}\label{z3.18}
&\frac{d}{dt}\Big[\frac{\eta(t)}{2}\big(\|\sqrt{\rho}\dot{u}\|_{L^2}^2
+\|\sqrt{\rho}\dot{w}\|_{L^2}^2+\|b_t\|_{L^2}^2\big)\Big]
+(2\mu+\lambda)\eta(t)\|{\rm div}\,\dot{u}\|_{L^2}^2+\mu\eta(t)\|{\rm curl}\,\dot{u}\|_{L^2}^2\nonumber\\
&\quad+(2c_d+c_0)\eta(t)\|{\rm div}\,\dot{w}\|_{L^2}^2+(c_d+c_a)\eta(t)\|{\rm curl}\,\dot{w}\|_{L^2}
+4\mu_r\eta(t)\Big\|\frac12{\rm curl}\,\dot{u}-\dot{w}\Big\|_{L^2}^2+\eta(t)\|{\rm curl}\,b_t\|_{L^2}^2\nonumber\\
&\le -\frac{d}{dt}\int_{\partial\Omega}\eta(t)(u\cdot\nabla n\cdot u)F_1dS-\frac{d}{dt}\int_{\partial\Omega}\eta(t)(w\cdot\nabla n\cdot w)F_2dS
+ C\delta \eta(t)\big(\|\nabla\dot{u}\|_{L^2}^2+\|\nabla\dot{w}\|_{L^2}^2\big)\nonumber\\
&\quad+C|\eta'(t)|\big(\|\sqrt{\rho}\dot{u}\|_{L^2}^2+\|\sqrt{\rho}\dot{w}\|_{L^2}^2+\|\nabla u\|_{L^2}^2+\|\nabla w\|_{L^2}^2+\|\nabla b\|_{L^2}^2+\|\nabla u\|_{L^2}^4
\big)\nonumber\\
&\quad+C|\eta'(t)|\big(\|{\rm curl}^2b\|_{L^2}^2+\|b_t\|_{L^2}^2+\|\nabla u\|_{L^2}^4\|\nabla b\|_{L^2}^2+\|\nabla u\|_{L^2}^2\|\nabla b\|_{L^2}^2
+\|\nabla w\|_{L^2}^4\big)\nonumber\\
&\quad+C\eta(t)\big(\|\sqrt{\rho}\dot{u}\|_{L^2}^2\|\nabla u\|_{L^2}^2+\|\nabla u\|_{L^2}^6+\|\nabla u\|_{L^4}^4
+\|\nabla u\|_{L^2}^2+\|\nabla w\|_{L^2}^2
+\|\nabla b\|_{L^2}^2\big)\nonumber\\
&\quad+C\eta(t)\big(\|\nabla u\|_{L^2}^4+\|\nabla w\|_{L^2}^4+\|\nabla b\|_{L^2}^4+\|{\rm curl}^2b\|_{L^2}^2\|\nabla b\|_{L^2}\|\nabla u\|_{L^2}^2
+\|\nabla b\|_{L^2}^6\big)\nonumber\\
&\quad+C\eta(t)\big(\|\nabla u\|_{L^2}^4\|\sqrt{\rho}\dot{u}\|_{L^2}^2+\|\nabla w\|_{L^2}^6
+\|\nabla u\|_{L^2}^4\|{\rm curl}^2b\|_{L^2}^2+\|b_t\|_{L^2}^2\|\nabla b\|_{L^2}^4\big)\nonumber\\
&\quad+C\delta\eta(t)\|\nabla b\|_{L^2}\|{\rm curl}^2b\|_{L^2}^3
+C\eta(t)\big(\|\nabla w\|_{L^4}^4+\|\nabla b\|_{L^2}^2\|{\rm curl}^2b\|_{L^2}^2\big)\nonumber\\
&\quad+C\eta(t)\big(\|\nabla b\|_{L^2}^4+\|\nabla u\|_{L^2}^4\big)\|{\rm curl}^2b\|_{L^2}^2
+C\eta(t)E_0^\frac12\|\nabla b\|_{L^2}\|\nabla b_t\|_{L^2}^2\nonumber\\
&\quad+C\eta(t)(\|\dot{u}\|_{L^2}^2+\|\dot{w}\|_{L^2}^2),
\end{align}
provided that $E_0\leq\varepsilon_1$.
\end{lemma}
\begin{proof}[Proof]
1. By \eqref{2.7} and \eqref{a7}, we rewrite $\eqref{a1}_2$ and $\eqref{a1}_3$ as the followings
\begin{align}
&\rho\dot{u}=\nabla F_1-(\mu+\mu_r)\curl{\rm curl}\,u+\rho\nabla(\Phi-\Phi_s)+(\rho-\rho_s)\nabla\Phi_s+(\nabla\times b)\times b+2\mu_r{\rm curl}\,w,\label{3.41}\\
&\rho\dot{w}=\nabla F_2-4\mu_rw-(c_a+c_d)\curl{\rm curl}\,w+2\mu_r{\rm curl}\,u.\label{3.42}
\end{align}
Applying $\eta(t)\dot{u}^j[\partial/\partial t+{\rm div}(u\cdot)]$ to $\eqref{3.41}^j$
and $\eta(t)\dot{w}^j[\partial/\partial t+{\rm div}(u\cdot)]$ to $\eqref{3.42}^j$, respectively,
summing up, and then integrating the resulting equality over $\Omega$, we get that
\begin{align}\label{3.44}
&\frac12\frac{d}{dt}\int\eta(t)\big(\rho|\dot{u}|^2+\rho|\dot{w}|^2\big)dx
-\eta'(t)\int\big(\rho|\dot{u}|^2+\rho|\dot{w}|^2\big)dx\nonumber\\
&=\int\eta(t)\big(\dot{u}\cdot\nabla F_{1t}+\dot{u}^j{\rm div}\,(u\partial_jF_1)\big)dx
+\int\eta(t)\big(\dot{w}\cdot\nabla F_{2t}+\dot{w}^j{\rm div}\,(u\partial_jF_2)\big)dx\nonumber\\
&\quad-(\mu+\mu_r)\int\eta(t)\big(\dot{u}\cdot\curl{\rm curl}\,u_t+\dot{u}^j{\rm div}\,((\curl{\rm curl}\,u)^ju)\big)dx\nonumber\\
&\quad-(c_a+c_d)\int\eta(t)\big(\dot{w}\cdot\curl{\rm curl}\,w_t+\dot{w}^j{\rm div}\,((\curl{\rm curl}\,w)^ju)\big)dx\nonumber\\
&\quad+2\mu_r\int\eta(t)\big(\dot{u}\cdot\curl w_t+\dot{u}\cdot\partial_j(u^j{\rm curl}\,w)\big)dx\nonumber\\
&\quad+2\mu_r\int\eta(t)\big(\dot{w}\cdot\curl u_t+\dot{w}\partial_j(u^j{\rm curl}\,u)\big)dx-4\mu_r\int\eta(t)\dot{w}^j\big(\dot{w}^j+{\rm div}\,(w^ju)\big)dx\nonumber\\
&\quad+\int\eta(t)\big(\dot{u}\cdot{\rm div}\,(b\otimes b)_t+\dot{u}^j{\rm div}\,({\rm div}\,(b b^j)u)\big)dx\nonumber\\
&\quad+\int\eta(t)\big(\dot{u}\cdot\nabla\Phi_s(\rho-\rho_s)_t+\dot{u}^j{\rm div}\,((\rho-\rho_s)\partial_j\Phi_su)\big)dx\nonumber\\
&\quad+\int\eta(t)\big[\dot{u}\cdot(\rho\nabla(\Phi-\Phi_s))_t+\dot{u}^j{\rm div}\,(\rho\partial_j(\Phi-\Phi_s)u)\big]dx\triangleq\sum_{i=1}^{10}J_i.
\end{align}
We denote by $h\triangleq u\cdot(\nabla n+(\nabla n)^{tr})$ and $u^\bot\triangleq-u\times n$, then it deduces from Lemma \ref{l23} that
\begin{align}\label{z3.22}
&-\int_{\partial\Omega}\eta(t)F_{1t}(u\cdot\nabla n\cdot u)dS\nonumber\\
&=-\frac{d}{dt}\int_{\partial\Omega}\eta(t)(u\cdot\nabla n\cdot u)F_1dS
+\int_{\partial\Omega}\eta(t)F_1h\cdot\dot{u}dS-\int_{\partial\Omega}\eta(t)
F_1h\cdot(u\cdot\nabla u)dS\nonumber\\
&\quad+\eta'(t)\int_{\partial\Omega}(u\cdot\nabla n\cdot u)F_1dS\nonumber\\
&=-\frac{d}{dt}\int_{\partial\Omega}\eta(t)(u\cdot\nabla n\cdot u)F_1dS
+\int_{\partial\Omega}\eta(t)F_1h\cdot\dot{u}dS+\eta'(t)\int_{\partial\Omega}(u\cdot\nabla n\cdot u)F_1dS\nonumber\\
&\quad-\int_{\partial\Omega}\eta(t)F_1h^i(\nabla u^i\times u^\bot)\cdot ndS\nonumber\\
&=-\frac{d}{dt}\int_{\partial\Omega}\eta(t)(u\cdot\nabla n\cdot u)F_1dS
+\int_{\partial\Omega}\eta(t)F_1h\cdot\dot{u}dS
+\eta'(t)\int_{\partial\Omega}(u\cdot\nabla n\cdot u)F_1dS\nonumber\\
&\quad-\int\eta(t)\nabla u^i\times u^\bot\cdot\nabla(F_1h^i)dx+\int\eta(t)F_1h^i\nabla\times u^\bot\cdot\nabla u^idx\nonumber\\
&\le -\frac{d}{dt}\int_{\partial\Omega}\eta(t)(u\cdot\nabla n\cdot u)F_1dS
+C\eta(t)\|\nabla F_1\|_{L^2}\|u\|_{L^3}\|\dot{u}\|_{L^6}\nonumber\\
&\quad+C\eta(t)\big(\|F_1\|_{L^3}\|u\|_{L^6}\|\nabla\dot{u}\|_{L^2}
+\|F_1\|_{L^3}\|u\|_{L^6}\|\dot{u}\|_{L^2}
+\|F_1\|_{L^3}\|\nabla u\|_{L^2}\|\dot{u}\|_{L^6}\big)\nonumber\\
&\quad+C\eta(t)\big(\|\nabla u\|_{L^2}\|u\|_{L^6}^2\|\nabla F_1\|_{L^6}+\|\nabla u\|_{L^4}^2\|u\|_{L^6}\|F_1\|_{L^3}\big)\nonumber\\
&\quad+|\eta'(t)|\big(\|\nabla u\|_{L^2}\|u\|_{L^6}\|F_1\|_{L^3}+\|u\|_{L^4}^2\|F_1\|_{L^2}
+\|\nabla F_1\|_{L^2}\|u\|_{L^4}^2\big)\nonumber\\
&\le -\frac{d}{dt}\int_{\partial\Omega}\eta(t)(u\cdot\nabla n\cdot u)F_1dS+C|\eta'(t)|\|\nabla u\|_{L^2}^2\big(\|\rho\dot{u}\|_{L^2}+\|\nabla w\|_{L^2}+\|\nabla u\|_{L^2}+\|P-P_s\|_{L^2}\big)\nonumber\\
&\quad+C|\eta'(t)|\|\nabla u\|_{L^2}^2\big(\|b\|_{L^6}\|\nabla b\|_{L^3}+\|\rho\nabla(\Phi-\Phi_s)\|_{L^2}+\|(\rho-\rho_s)\nabla\Phi_s\|_{L^2}+\|b\|_{L^2}^\frac12\|\nabla b\|_{L^2}^\frac32\big)\nonumber\\
&\quad+C\eta(t)\big(\|\rho\dot{u}\|_{L^2}+\|P-P_s\|_{L^2}+\|\nabla u\|_{L^2}+\|\nabla w\|_{L^2}+\|b\|_{L^3}\|\nabla b\|_{L^6}+\|\nabla u\|_{L^2}^2
+\|b\|_{L^4}^2\nonumber\\
&\quad+\|\rho\nabla(\Phi-\Phi_s)\|_{L^2}+\|(\rho-\rho_s)\nabla\Phi_s\|_{L^2}\big)\|\nabla u\|_{L^2}\big(\|\nabla\dot{u}\|_{L^2}+\|\nabla u\|_{L^2}^2+\|\nabla u\|_{L^4}^2\big)\nonumber\\
&\quad+C\eta(t)\|\nabla u\|_{L^2}^3\|\nabla F_1\|_{L^6}\nonumber\\
&\le -\frac{d}{dt}\int_{\partial\Omega}\eta(t)(u\cdot\nabla n\cdot u)F_1dS
+C|\eta'(t)|\big(\|\sqrt{\rho}\dot{u}\|_{L^2}^2+\|\nabla u\|_{L^2}^2+\|\nabla w\|_{L^2}^2+\|\nabla b\|_{L^2}^2+\|\nabla u\|_{L^2}^4\big)\nonumber\\
&\quad+C|\eta'(t)|\big(\|{\rm curl}^2b\|_{L^2}^2+\|\nabla u\|_{L^2}^4\|\nabla b\|_{L^2}^2+\|\nabla u\|_{L^2}^2\|\nabla b\|_{L^2}^2\big)
+C\eta(t)\|\nabla u\|_{L^2}^3\|\nabla F_1\|_{L^6}\nonumber\\
&\quad+\delta\eta(t)\|\nabla\dot{u}\|_{L^2}^2+C\eta(t)\big(\|\sqrt{\rho}\dot{u}\|_{L^2}^2\|\nabla u\|_{L^2}^2+\|\nabla u\|_{L^2}^6+\|\nabla u\|_{L^4}^4
+\|\nabla u\|_{L^2}^2+\|\nabla w\|_{L^2}^2
\big)\nonumber\\
&\quad+C\eta(t)\big(\|\nabla u\|_{L^2}^4+\|\nabla w\|_{L^2}^4+\|\nabla b\|_{L^2}^4+\|{\rm curl}^2b\|_{L^2}^2\|\nabla b\|_{L^2}\|\nabla u\|_{L^2}^2
+\|\nabla b\|_{L^2}^6\big),
\end{align}
due to
\begin{align*}
&{\rm div}(\nabla u^i\times u^\bot)=u^\bot\cdot\curl\nabla u^i-\nabla u^i\cdot\curl u^\bot=-\nabla u^i\cdot\curl u^\bot,\\
&\|\sqrt{\rho}\dot{u}\|_{L^2}+\|\sqrt{\rho}\dot{w}\|_{L^2}\le C\big(\|\dot{u}\|_{L^2}+\|\dot{w}\|_{L^2}\big)
\le C\big(\|\nabla\dot{u}\|_{L^2}+\|\nabla\dot{w}\|_{L^2}\big),\\
&\|\dot{u}\|_{L^6}\le C\big(\|\nabla\dot{u}\|_{L^2}+\|\nabla u\|_{L^2}^2\big)\ \ (\text{see}\ \eqref{3.2}),
\end{align*}
and
\begin{align*}
\|\nabla F_1\|_{L^6}&\le C\big(\|\rho\dot{u}\|_{L^6}+\|{\rm curl}\, w\|_{L^6}+\|b\cdot\nabla b\|_{L^6}
+\|\rho\nabla(\Phi-\Phi_s)\|_{L^6}+\|(\rho-\rho_s)\nabla\Phi_s\|_{L^6}\big)\nonumber\\
&\le C(\hat{\rho})\big(\|\dot{u}\|_{L^6}+\|b\|_{L^\infty}\|\nabla b\|_{L^6}+\|\rho\dot{u}\|_{L^2}+\|\rho\dot{w}\|_{L^2}
+\|b\|_{L^6}\|\nabla b\|_{L^3}\big)\nonumber\\
&\quad+C(\hat{\rho})\big(\|\nabla u\|_{L^2}+\|\nabla w\|_{L^2}+\|w\|_{L^2}+1\big)\nonumber\\
&\le C(\hat{\rho})\Big(\|\nabla\dot{u}\|_{L^2}+\|\nabla\dot{w}\|_{L^2}
+\|\nabla b\|_{L^2}^\frac12\|{\rm curl}^2b\|_{L^2}^\frac32
+\|\nabla b\|_{L^2}^\frac32\|{\rm curl}^2b\|_{L^2}^\frac12\Big)\nonumber\\
&\quad+C(\hat{\rho})\big(\|\nabla b\|_{L^2}^2+\|\nabla u\|_{L^2}^2+\|\nabla w\|_{L^2}+\|w\|_{L^2}+1\big).
\end{align*}
Thus, it follows from integration by parts, \eqref{a6}, \eqref{2.5}, \eqref{3.6}, H{\"o}lder's inequality, \eqref{2.3}, Lemma \ref{l28}, and \eqref{z3.22} that
\begin{align}\label{3.45}
J_1&=\int_{\partial\Omega}\eta(t)F_{1t}\dot{u}\cdot ndS-\int\eta(t)F_{1t}{\rm div}\,\dot{u}dx
-\int\eta(t)u\cdot\nabla\dot{u}\cdot\nabla F_1dx\nonumber\\
&=-\int_{\partial\Omega}\eta(t)F_{1t}(u\cdot\nabla n\cdot u)dS-(2\mu+\lambda)\int\eta(t)({\rm div}\,\dot{u})^2dx
+(2\mu+\lambda)\int\eta(t){\rm div}\,\dot{u}\nabla u:\nabla udx\nonumber\\
&\quad-\gamma\int\eta(t) P{\rm div}\,\dot{u}{\rm div}\,udx
+\int\eta(t){\rm div}\,\dot{u}u\cdot\nabla F_1dx-\int\eta(t)u\cdot\nabla\dot{u}\cdot\nabla F_1dx\nonumber\\
&\quad-\int\eta(t){\rm div}\,\dot{u}u\cdot\nabla P_sdx+\int\eta(t){\rm div}\,\dot{u} b\cdot b_tdx\nonumber\\
&\le -\int_{\partial\Omega}\eta(t)F_{1t}(u\cdot\nabla n\cdot u)dS
-(2\mu+\lambda)\int\eta(t)({\rm div}\,\dot{u})^2dx+\delta\eta(t)\|\nabla\dot{u}\|_{L^2}^2\nonumber\\
&\quad+C(\delta)\eta(t)\big(\|\nabla u\|_{L^2}^2\|\nabla F_1\|_{L^3}^2+\|\nabla u\|_{L^4}^4+\|\nabla u\|_{L^2}^2
+\|b_t\|_{L^3}^2\|b\|_{L^6}^2\big)\nonumber\\
&\le  -\int_{\partial\Omega}\eta(t)F_{1t}(u\cdot\nabla n\cdot u)dS
-(2\mu+\lambda)\int\eta(t)({\rm div}\,\dot{u})^2dx+\delta\eta(t)\|\nabla\dot{u}\|_{L^2}^2\nonumber\\
&\quad+\delta\eta(t)\|\nabla b_t\|_{L^2}^2+C\eta(t)\big(\|\nabla u\|_{L^2}^2\|\nabla F_1\|_{L^2}\|\nabla F_1\|_{L^6}+\|\nabla u\|_{L^4}^4+\|\nabla u\|_{L^2}^2
+\|b_t\|_{L^2}^2\|\nabla b\|_{L^2}^4\big)\nonumber\\
&\le  -\frac{d}{dt}\int_{\partial\Omega}\eta(t)(u\cdot\nabla n\cdot u)F_1dS-(2\mu+\lambda)\int\eta(t)({\rm div}\,\dot{u})^2dx+\delta\eta(t)\|\nabla\dot{u}\|_{L^2}^2\nonumber\\
&\quad+C|\eta'(t)|\big(\|\sqrt{\rho}\dot{u}\|_{L^2}^2+\|\nabla u\|_{L^2}^2+\|\nabla w\|_{L^2}^2+\|\nabla b\|_{L^2}^2+\|\nabla u\|_{L^2}^4
+\|\nabla u\|_{L^2}^2\|\nabla b\|_{L^2}^2\big)\nonumber\\
&\quad+C|\eta'(t)|\big(\|{\rm curl}^2b\|_{L^2}^2+\|\nabla u\|_{L^2}^4\|\nabla b\|_{L^2}^2\big)+C\eta(t)\big(\|\nabla u\|_{L^2}^3+\|\nabla u\|_{L^2}^2\|\nabla F_1\|_{L^2}\big)\|\nabla F_1\|_{L^6}
\nonumber\\
&\quad+C\eta(t)\big(\|\sqrt{\rho}\dot{u}\|_{L^2}^2\|\nabla u\|_{L^2}^2+\|\nabla u\|_{L^2}^6+\|\nabla u\|_{L^4}^4
+\|\nabla u\|_{L^2}^2+\|\nabla w\|_{L^2}^2
\big)\nonumber\\
&\quad+C\eta(t)\big(\|\nabla u\|_{L^2}^4+\|\nabla w\|_{L^2}^4+\|\nabla b\|_{L^2}^4+\|{\rm curl}^2b\|_{L^2}^2\|\nabla b\|_{L^2}\|\nabla u\|_{L^2}^2
+\|\nabla b\|_{L^2}^6\big)\nonumber\\
&\quad+C\eta(t)\big(\|\nabla u\|_{L^2}^4\|\rho\dot{u}\|_{L^2}^2+\|\nabla w\|_{L^2}^6
+\|\nabla u\|_{L^2}^4\|{\rm curl}^2b\|_{L^2}^2+\|b_t\|_{L^2}^2\|\nabla b\|_{L^2}^4\big)\nonumber\\
&\le  -\frac{d}{dt}\int_{\partial\Omega}\eta(t)(u\cdot\nabla n\cdot u)F_1dS-(2\mu+\lambda)\int\eta(t)({\rm div}\,\dot{u})^2dx+\delta\eta(t)\|\nabla\dot{u}\|_{L^2}^2\nonumber\\
&\quad+C\eta(t)\big(\|\sqrt{\rho}\dot{u}\|_{L^2}^2\|\nabla u\|_{L^2}^2+\|\nabla u\|_{L^2}^6+\|\nabla u\|_{L^4}^4
+\|\nabla u\|_{L^2}^2+\|\nabla w\|_{L^2}^2
\big)\nonumber\\
&\quad+C\eta(t)\big(\|\nabla u\|_{L^2}^4+\|\nabla w\|_{L^2}^4+\|\nabla b\|_{L^2}^4+\|{\rm curl}^2b\|_{L^2}^2\|\nabla b\|_{L^2}\|\nabla u\|_{L^2}^2
+\|\nabla b\|_{L^2}^6\big)\nonumber\\
&\quad+C\eta(t)\big(\|\nabla u\|_{L^2}^4\|\rho\dot{u}\|_{L^2}^2+\|\nabla w\|_{L^2}^6
+\|\nabla u\|_{L^2}^4\|{\rm curl}^2b\|_{L^2}^2+\|b_t\|_{L^2}^2\|\nabla b\|_{L^2}^4\big)\nonumber\\
&\quad+C|\eta'(t)|\big(\|\sqrt{\rho}\dot{u}\|_{L^2}^2+\|\nabla u\|_{L^2}^2+\|\nabla w\|_{L^2}^2+\|\nabla b\|_{L^2}^2+\|\nabla u\|_{L^2}^4
+\|\nabla u\|_{L^2}^2\|\nabla b\|_{L^2}^2\big)\nonumber\\
&\quad+C|\eta'(t)|\big(\|{\rm curl}^2b\|_{L^2}^2+\|\nabla u\|_{L^2}^4\|\nabla b\|_{L^2}^2\big)+C\eta(t)\big(\|\nabla u\|_{L^2}^4\|{\rm curl}^2b\|_{L^2}^2
+\|\nabla b\|_{L^2}^2\big)\nonumber\\
&\quad+C\delta\eta(t)\|\nabla b\|_{L^2}\|{\rm curl}^2b\|_{L^2}^3,
\end{align}
where we have used
\begin{align*}
F_{1t}&=(2\mu+\lambda){\rm div}\,u_t-(P-P_s)_t-\frac12(|b|^2)_t\nonumber\\
&=(2\mu+\lambda){\rm div}\,\dot{u}-(2\mu+\lambda){\rm div}\,(u\cdot\nabla u)+u\cdot\nabla(P-P_s)+\gamma P{\rm div}\,u+u\cdot\nabla P_s
-b\cdot b_t\nonumber\\
&=(2\mu+\lambda){\rm div}\,\dot{u}-(2\mu+\lambda)u\cdot\nabla{\rm div}\,u
-(2\mu+\lambda)\nabla u:\nabla u+u\cdot\nabla(P-P_s)\nonumber\\
&\quad+\gamma P{\rm div}\,u+u\cdot\nabla P_s-b\cdot b_t\nonumber\\
&=(2\mu+\lambda){\rm div}\,\dot{u}
-(2\mu+\lambda)\nabla u:\nabla u+\gamma P{\rm div}\,u-u\cdot\nabla F_1+u\cdot\nabla P_s-b\cdot b_t.
\end{align*}

2. Denoting $\tilde{h}\triangleq w\cdot(\nabla n+(\nabla n)^{tr})$, we get that
\begin{align}\label{z3.24}
&-\int_{\partial\Omega}\eta(t)F_{2t}(w\cdot\nabla n\cdot w)dS\nonumber\\
&=-\frac{d}{dt}\int_{\partial\Omega}\eta(t)(w\cdot\nabla n\cdot w)F_2dS+\int_{\partial\Omega}\eta(t)F_2\tilde{h}\cdot\dot{w}dS
-\int_{\partial\Omega}\eta(t)F_2\tilde{h}\cdot(u\cdot\nabla w)dS\nonumber\\
&\quad+\eta'(t)\int_{\partial\Omega}(w\cdot\nabla n\cdot w)F_2dS\nonumber\\
&=-\frac{d}{dt}\int_{\partial\Omega}\eta(t)(w\cdot\nabla n\cdot w)F_2dS+\int_{\partial\Omega}F_2\tilde{h}\cdot\dot{w}dS
+\eta'(t)\int_{\partial\Omega}(w\cdot\nabla n\cdot w)F_2dS\nonumber\\
&\quad-\int_{\partial\Omega}F_2\tilde{h}\nabla w^i\times u^\bot\cdot ndS\nonumber\\
&=-\frac{d}{dt}\int_{\partial\Omega}\eta(t)(w\cdot\nabla n\cdot w)F_2dS+\int_{\partial\Omega}F_2\tilde{h}\cdot\dot{w}dS
+\eta'(t)\int_{\partial\Omega}(w\cdot\nabla n\cdot w)F_2dS\nonumber\\
&\quad-\int\eta(t)\nabla w^i\times u^\bot\cdot\nabla(F_2\tilde{h}^i)dx
+\int\eta(t)F_2\tilde{h}^i\nabla\times u^\bot\cdot\nabla w^idx\nonumber\\
&\le -\frac{d}{dt}\int_{\partial\Omega}\eta(t)(w\cdot\nabla n\cdot w)F_2dS
+C\eta(t)\|\nabla F_2\|_{L^2}\|w\|_{L^6}\|\dot{w}\|_{L^3}\nonumber\\
&\quad+C\eta(t)\big(\|F_2\|_{L^3}\|w\|_{L^6}\|\nabla\dot{w}\|_{L^2}
+\|F_2\|_{L^3}\|w\|_{L^6}\|\dot{w}\|_{L^2}
+\|F_2\|_{L^3}\|\nabla w\|_{L^2}\|\dot{w}\|_{L^6}\big)\nonumber\\
&\quad+C\eta(t)\big(\|\nabla w\|_{L^2}\|w\|_{L^6}\|u\|_{L^6}\|\nabla F_2\|_{L^6}+\|\nabla u\|_{L^4}\|\nabla w\|_{L^4}\|w\|_{L^6}\|F_2\|_{L^3}\big)\nonumber\\
&\quad+|\eta'(t)|\big(\|\nabla w\|_{L^2}\|w\|_{L^6}\|F_2\|_{L^3}+\|w\|_{L^4}^2\|F_2\|_{L^2}
+\|\nabla F_2\|_{L^2}\|w\|_{L^4}^2\big)\nonumber\\
&\le-\frac{d}{dt}\int_{\partial\Omega}\eta(t)(w\cdot\nabla n\cdot w)F_2dS
+C|\eta'(t)|\big(\|\rho\dot{w}\|_{L^2}^2+\|\nabla u\|_{L^2}^2+\|w\|_{H^1}^2+\|\nabla w\|_{L^2}^4\big)\nonumber\\
&\quad+\delta\eta(t)\|\nabla\dot{w}\|_{L^2}^2
+C\eta(t)\big(\|\rho\dot{w}\|_{L^2}^2\|\nabla w\|_{L^2}^2+\|\nabla u\|_{L^2}^6+\|\nabla w\|_{L^2}^6+\|\nabla u\|_{L^2}^4+\|w\|_{H^1}^4\big)\nonumber\\
&\quad+C\eta(t)\big(\|\nabla u\|_{L^4}^4+\|\nabla w\|_{L^4}^4\big)
+C\eta(t)\big(\|\nabla w\|_{L^2}^3+\|\nabla u\|_{L^2}^3\big)\|\|\nabla F_2\|_{L^6},
\end{align}
due to
\begin{align}
\|\nabla F_2\|_{L^6}&\le C\big(\|\rho\dot{w}\|_{L^6}+\|{\rm curl}\,u\|_{L^6}+\|w\|_{L^6}\big)\nonumber\\
&\le C\big(\|\rho\dot{w}\|_{L^6}+\|\rho\dot{u}\|_{L^2}+\|\rho\dot{w}\|_{L^2}
+\|\rho\nabla(\Phi-\Phi_s)\|_{L^2}+\|(\rho-\rho_s)\nabla\Phi_s\|_{L^2}\big)\nonumber\\[3pt]
&\quad+C\big(\|b\cdot\nabla b\|_{L^2}+\|\nabla u\|_{L^2}+\|\nabla w\|_{L^2}+\|w\|_{L^2}\big)\nonumber\\
&\le C\big(\|\nabla\dot{w}\|_{L^2}+\|\nabla\dot{w}\|_{L^2}+\|\nabla b\|_{L^2}^\frac32\|{\rm curl}^2b\|_{L^2}^\frac12\big)\nonumber\\
&\quad+C\big(\|\nabla b\|_{L^2}^2+\|\nabla u\|_{L^2}^2+\|\nabla w\|_{L^2}^2+\|w\|_{L^2}+1\big).
\end{align}
We infer from \eqref{z3.24} that
\begin{align}
J_2&=\int_{\partial\Omega}\eta(t)F_{2t}\dot{w}\cdot ndS-\int\eta(t)F_{2t}{\rm div}\,\dot{w}dx
-\int\eta(t)u\cdot\nabla\dot{w}\cdot\nabla F_2dx\nonumber\\
&=-\int_{\partial\Omega}\eta(t)F_{2t}(w\cdot\nabla n\cdot w)dS
-(2c_d+c_0)\int\eta(t)({\rm div}\,\dot{w})^2dx+(2c_d+c_0)\int\eta(t){\rm div}\,\dot{w}\nabla u:\nabla wdx\nonumber\\
&\quad+(2c_d+c_0)\int\eta(t){\rm div}\,\dot{w}u\cdot\nabla F_2dx
-\int\eta(t)u\cdot\nabla\dot{w}\cdot\nabla F_2dx\nonumber\\
&\le -\int_{\partial\Omega}\eta(t)F_{2t}(w\cdot\nabla n\cdot w)dS
-(2c_d+c_0)\int\eta(t)({\rm div}\,\dot{w})^2dx+\delta\eta(t)\|\nabla\dot{w}\|_{L^2}^2\nonumber\\
&\quad+C\eta(t)\big(\|\nabla u\|_{L^4}^4+\|\nabla w\|_{L^4}^4+\|\nabla u\|_{L^2}^2\|\nabla F_2\|_{L^3}^2\big)\nonumber\\
&\le -\frac{d}{dt}\int_{\partial\Omega}\eta(t)(w\cdot\nabla n\cdot w)F_2dS-(2c_d+c_0)\int\eta(t)({\rm div}\,\dot{w})^2dx
+\delta\eta(t)\|\nabla\dot{w}\|_{L^2}^2\nonumber\\
&\quad+C|\eta'(t)|\big(\|\rho\dot{w}\|_{L^2}^2+\|\nabla u\|_{L^2}^2+\|w\|_{H^1}^2+\|\nabla w\|_{L^2}^4\big)
+C\eta(t)\big(\|\nabla u\|_{L^4}^4+\|\nabla w\|_{L^4}^4+\|\rho\dot{w}\|_{L^2}^2\|\nabla w\|_{L^2}^2\big)\nonumber\\
&\quad+C\eta(t)\big(\|\nabla u\|_{L^2}^2\|\nabla F_2\|_{L^2}\|\nabla F_2\|_{L^6}
+\|\nabla u\|_{L^2}^6+\|\nabla w\|_{L^2}^6+\|\nabla u\|_{L^2}^4+\|w\|_{H^1}^4\big)\nonumber\\
&\le  -\frac{d}{dt}\int_{\partial\Omega}\eta(t)(w\cdot\nabla n\cdot w)F_2dS-(2c_d+c_0)\int\eta(t)({\rm div}\,\dot{w})^2dx
+\delta\eta(t)\big(\|\nabla\dot{w}\|_{L^2}^2+\|\nabla\dot{u}\|_{L^2}^2\big)\nonumber\\
&\quad+C|\eta'(t)|\big(\|\rho\dot{w}\|_{L^2}^2+\|\nabla u\|_{L^2}^2+\|w\|_{H^1}^2+\|\nabla w\|_{L^2}^4\big)
+C\eta(t)\big(\|\rho\dot{w}\|_{L^2}^2\|\nabla w\|_{L^2}^2+\|\nabla u\|_{L^2}^4\|\rho\dot{w}\|_{L^2}^2\big)\nonumber\\
&\quad+C\eta(t)\big(\|\nabla u\|_{L^2}^6+\|\nabla w\|_{L^2}^6+\|\nabla b\|_{L^2}^6
+\|\nabla u\|_{L^4}^4+\|\nabla w\|_{L^4}^4+\|\nabla b\|_{L^2}^4+\|\nabla u\|_{L^2}^4+\|w\|_{H^1}^4\big)\nonumber\\
&\quad+C\eta(t)\big(\|\nabla u\|_{L^2}^2+\|\nabla w\|_{L^2}^2+\|\nabla b\|_{L^2}^2\|{\rm curl}^2b\|_{L^2}^2\big),
\end{align}
owing to
\begin{align}
F_{2t}&=(2c_d+c_0){\rm div}\,w_t\nonumber\\
&=(2c_d+c_0){\rm div}\,\dot{w}-(2c_d+c_0){\rm div}\,(u\cdot\nabla w)\nonumber\\
&=(2c_d+c_0){\rm div}\,\dot{w}-(2c_d+c_0)u\cdot\nabla{\rm div}\,w-(2c_d+c_0)\nabla u:\nabla w\nonumber\\
&=(2c_d+c_0){\rm div}\,\dot{w}-(2c_d+c_0)\nabla u:\nabla w-\nabla F_2\cdot u.
\end{align}

3. By a direct calculation, one obtains that
\begin{align}
J_3&=-(\mu+\mu_r)\int\eta(t)\dot{u}\cdot(\curl{\rm curl}\,u_t)dx-(\mu+\mu_r)\int\eta(t)\dot{u}\cdot(\curl{\rm curl}\,u){\rm div}\,udx\nonumber\\
&\quad-(\mu+\mu_r)\int\eta(t)u^i\dot{u}\cdot\curl(\partial_i{\rm curl}\,u)dx\nonumber\\
&=-(\mu+\mu_r)\int\eta(t)|{\rm curl}\,\dot{u}|^2dx+(\mu+\mu_r)\int\eta(t){\rm curl}\,\dot{u}\cdot{\rm curl}\,(u\cdot\nabla u)dx\nonumber\\
&\quad+(\mu+\mu_r)\int\eta(t)({\rm curl}\,u\times\dot{u})\cdot\nabla{\rm div}\,udx
-(\mu+\mu_r)\int\eta(t){\rm div}\,u({\rm curl}\,u\cdot{\rm curl}\,\dot{u})dx\nonumber\\
&\quad-(\mu+\mu_r)\int\eta(t)u^i{\rm div}\,(\partial_i{\rm curl}\,u\times\dot{u})dx
-(\mu+\mu_r)\int\eta(t)u^i\partial_i{\rm curl}\,u\cdot{\rm curl}\,\dot{u}dx\nonumber\\
&=-(\mu+\mu_r)\int\eta(t)|{\rm curl}\,\dot{u}|^2dx+(\mu+\mu_r)\int\eta(t){\rm curl}\,\dot{u}\partial_iu\times\nabla u^idx\nonumber\\
&\quad+(\mu+\mu_r)\int\eta(t)({\rm curl}\,u\times\dot{u})\cdot\nabla{\rm div}\,udx
-(\mu+\mu_r)\int\eta(t){\rm div}\,u({\rm curl}\,u\cdot{\rm curl}\,\dot{u})dx\nonumber\\
&\quad-(\mu+\mu_r)\int\eta(t)u^i{\rm div}\,(\partial_i{\rm curl}\,u\times\dot{u})dx\nonumber\\
&=-(\mu+\mu_r)\int\eta(t)|{\rm curl}\,\dot{u}|^2dx+(\mu+\mu_r)\int\eta(t){\rm curl}\,\dot{u}\partial_iu\times\nabla u^idx\nonumber\\
&\quad+(\mu+\mu_r)\int\eta(t)({\rm curl}\,u\times\dot{u})\cdot\nabla{\rm div}\,udx
-(\mu+\mu_r)\int\eta(t){\rm div}\,u({\rm curl}\,u\cdot{\rm curl}\,\dot{u})dx\nonumber\\
&\quad-(\mu+\mu_r)\int\eta(t) u\cdot\nabla{\rm div}\,({\rm curl}\,u\times\dot{u})dx
+(\mu+\mu_r)\int\eta(t) u^i{\rm div}\,({\rm curl}\,u\times\partial_i\dot{u})dx\nonumber\\
&=-(\mu+\mu_r)\int\eta(t)|{\rm curl}\,\dot{u}|^2dx+(\mu+\mu_r)\int\eta(t){\rm curl}\,\dot{u}\nabla_iu\times\nabla u^idx\nonumber\\
&\quad-(\mu+\mu_r)\int\eta(t){\rm div}\,u({\rm curl}\,u\cdot{\rm curl}\,\dot{u})dx
-(\mu+\mu_r)\int\eta(t)\nabla u^i\cdot({\rm curl}\,u\times\partial_i\dot{u})dx\nonumber\\
&\le \delta\eta(t)\|\nabla\dot{u}\|_{L^2}^2+C\eta(t)\|\nabla u\|_{L^4}^4-(\mu+\mu_r)\eta(t)\|{\rm curl}\,\dot{u}\|_{L^2}^2,
\end{align}
due to
\begin{align*}
\curl(\dot{u}{\rm div}\,u)&={\rm div}\,u{\rm curl}\,\dot{u}+\nabla{\rm div}\,u\times\dot{u},\\
{\rm div}\,(\partial_i{\rm curl}\,u\times\dot{u})&=\dot{u}\cdot\curl(\partial_i{\rm curl}\,u)-\partial_i{\rm curl}\,u\cdot\curl\dot{u},\\
\int{\rm curl}\,u\cdot(\nabla{\rm div}\,u\times\dot{u})dx&=-\int({\rm curl}\,u\times\dot{u})\cdot\nabla{\rm div}\,udx,\\
\int{\rm curl}\,\dot{u}\cdot{\rm curl}\,(u\cdot\nabla u)dx&=\int{\rm curl}\,\dot{u}\cdot{\rm curl}\,(u^i\partial_iu)dx \\
& =\int{\rm curl}\,\dot{u}\big(u^i{\rm curl}\,\partial_iu+\partial_iu\cdot\nabla u^i\big)dx\nonumber\\
&=\int u^i\partial_i{\rm curl}\,u\cdot{\rm curl}\,\dot{u}dx+\int{\rm curl}\,\dot{u}\partial_iu\times\nabla u^idx,
\end{align*}
and
\begin{align*}
\int u\cdot\nabla{\rm div}\,({\rm curl}\,u\times\dot{u})dx
&=\int u^i\partial_i{\rm div}\,({\rm curl}\,u\times\dot{u})dx\nonumber\\
&=\int u^i\partial_i(\dot{u}\cdot\curl{\rm curl}\,u-{\rm curl}\,\dot{u}\cdot{\rm curl}\,u)dx\nonumber\\
&=\int u^i(\dot{u}\cdot\partial_i\curl{\rm curl}\,u-{\rm curl}\,\dot{u}\cdot\partial_i{\rm curl}\,u)dx\nonumber\\
&\quad+\int u^i(\partial_i\dot{u}\cdot\curl{\rm curl}\,udx-\partial_i{\rm curl}\,\dot{u}\cdot{\rm curl}\,u)dx\nonumber\\
&=\int u^i{\rm div}\,(\partial_i{\rm curl}\,u\times\dot{u})dx+\int u^i{\rm div}\,({\rm curl}\,u\times\partial_i\dot{u})dx.
\end{align*}

4. Similarly, we get that
\begin{align}
J_4&=-(c_a+c_d)\int\eta(t)\dot{w}\cdot(\curl{\rm curl}\,w_t)dx-(c_a+c_d)\int\eta(t)\dot{w}\cdot(\curl{\rm curl}\,w){\rm div}\,udx\nonumber\\
&\quad-(c_a+c_d)\int\eta(t)u^i\dot{w}\cdot\curl(\partial_i{\rm curl}\,w)dx\nonumber\\
&=-(c_a+c_d)\int\eta(t)|{\rm curl}\,\dot{w}|^2dx+(c_a+c_d)\int\eta(t){\rm curl}\,\dot{w}{\rm curl}\,(u\cdot\nabla w)dx\nonumber\\
&\quad+(c_a+c_d)\int\eta(t)({\rm curl}\,w\times\dot{w})\cdot\nabla{\rm div}\,udx
-(c_a+c_d)\int\eta(t){\rm div}\,u{\rm curl}\,\dot{w}\cdot{\rm curl}\,wdx\nonumber\\
&\quad-(c_a+c_d)\int\eta(t)u^i{\rm div}\,(\partial_i{\rm curl}\,w\times\dot{w})dx
-(c_a+c_d)\int\eta(t)u^i\partial_i{\rm curl}\,w\cdot\curl\dot{w}dx\nonumber\\
&=-(c_a+c_d)\int\eta(t)|{\rm curl}\,\dot{w}|^2dx+(c_a+c_d)\int\eta(t){\rm curl}\,\dot{w}\cdot(\nabla u^i\times\partial_iw)dx\nonumber\\
&\quad+(c_a+c_d)\int\eta(t){\rm curl}\,w\times\nabla u^i\cdot\partial_i\dot{w}dx
-(c_a+c_d)\int\eta(t){\rm div}\,u{\rm curl}\,\dot{w}\cdot{\rm curl}\,wdx\nonumber\\
&\le \delta\eta(t)\|\nabla\dot{w}\|_{L^2}^2+C\eta(t)\big(\|\nabla w\|_{L^4}^4+\|\nabla u\|_{L^4}^4\big)
-(c_a+c_d)\int\eta(t)|{\rm curl}\,\dot{w}|^2dx.
\end{align}

5. By Lemma \ref{l23}, H{\"o}lder's inequality, Sobolev's inequality, \eqref{3.2}, and \eqref{z2.5}, we have
\begin{align}\label{t3.36}
&\Big|-2\mu_r\int_{\partial\Omega}\eta(t)\dot{u}\cdot(\dot{w}\times n)dS\Big|\nonumber \\
&\le C\eta(t)\int_{\partial\Omega}|\dot{u}||\dot{w}|dS\nonumber \\
&\le C\eta(t)\||\dot{u}||\dot{w}|\|_{W^{1,1}}\nonumber \\
&\le C\eta(t)\big(\|\dot{u}\|_{L^2}\|\nabla\dot{w}\|_{L^2}
+\|\nabla\dot{u}\|_{L^2}\|\dot{w}\|_{L^2}+\|\dot{u}\|_{L^2}\|\dot{w}\|_{L^2}\big) \nonumber\\
&\le \delta\eta(t)(\|\nabla\dot{u}\|_{L^2}^2+\|\nabla\dot{w}\|_{L^2}^2)+C\eta(t)(\|\dot{u}\|_{L^2}^2+\|\dot{w}\|_{L^2}^2),
\end{align}
which together with \eqref{a6}, H\"older's inequality, and integration by parts over $\Omega$ implies that
\begin{align}
J_5+J_6&=2\mu_r\int\eta(t)\dot{u}\cdot({\rm curl}\,w_t+\partial_j(u^j{\rm curl}\,w))dx
+2\mu_r\int\eta(t)\dot{w}\cdot({\rm curl}\,w_t+\partial_j(u^j{\rm curl}\,u))dx\nonumber\\
&=2\mu_r\int\eta(t)(\dot{u}\cdot{\rm curl}\,\dot{w}+\dot{w}\cdot{\rm curl}\,\dot{u})dx-2\mu_r\int\eta(t)\dot{u}\cdot{\rm curl}\,(u\cdot\nabla w)dx\nonumber\\
&\quad
+2\mu_r\int\eta(t)\dot{u}\cdot\partial_j(u^j{\rm curl}\,w)dx-2\mu_r\int\eta(t)\dot{w}\cdot{\rm curl}(u\cdot\nabla u)dx\nonumber\\
&\quad
+2\mu_r\int\eta(t)\dot{w}\cdot\partial_j(u^j{\rm curl}\,u)dx\nonumber\\
&=2\mu_r\int\eta(t)(\dot{u}\cdot{\rm curl}\,\dot{w}+\dot{w}\cdot{\rm curl}\,\dot{u})dx-2\mu_r\int\eta(t)\dot{u}\cdot{\rm curl}\,(u^j\partial_j w)dx\nonumber\\
&\quad
+2\mu_r\int\eta(t)\dot{u}\cdot\partial_j(u^j{\rm curl}\,w)dx-2\mu_r\int\eta(t)\dot{w}\cdot{\rm curl}\,(u^j\partial_j u)dx\nonumber\\
&\quad
+2\mu_r\int\eta(t)\dot{w}\cdot\partial_j(u^j{\rm curl}\,u)dx\nonumber\\
&=2\mu_r\int\eta(t)(\dot{u}\cdot{\rm curl}\,\dot{w}+\dot{w}\cdot{\rm curl}\,\dot{u})dx-2\mu_r\int\eta(t)\dot{u}\cdot(\nabla u^j\times\partial_jw+u\cdot\nabla{\rm curl}\,w)dx\nonumber\\
&\quad+2\mu_r\int\eta(t)\dot{u}\cdot(\divv u{\rm curl}\,w+u\cdot\nabla{\rm curl}w)dx-2\mu_r\int\eta(t)\dot{w}\cdot(\nabla u^j\times\partial_ju+u\cdot\nabla{\rm curl}\,u)dx\nonumber\\
&\quad+2\mu_r\int\eta(t)\dot{w}\cdot(\divv u{\rm curl}\,u+u\cdot\nabla{\rm curl}\,u)dx\nonumber\\
&=4\mu_r\int\eta(t)\dot{w}\cdot{\rm curl}\,\dot{u}dx-2\mu_r\int_{\partial\Omega}\eta(t)\dot{u}\cdot(\dot{w}\times n)dS\nonumber\\
&\quad+2\mu_r\int\eta(t)\dot{u}\cdot(\divv u{\rm curl}\,w-\nabla u^j\times\partial_jw)dx+2\mu_r\int\eta(t)\dot{w}\cdot(\divv u{\rm curl}\,u-\nabla u^j\times\partial_ju)dx\nonumber\\
&\le 4\mu_r\int\eta(t)\dot{w}\cdot{\rm curl}\,\dot{u}dx+\delta\eta(t)(\|\nabla\dot{u}\|_{L^2}^2
+\|\nabla\dot{w}\|_{L^2}^2)+C\eta(t)\|\dot{w}\|_{L^2}^2+C\eta(t)\|\dot{u}\|_{L^2}^2\nonumber\\
&\quad+C\eta(t)\|\nabla u\|_{L^4}^4+C\eta(t)\|\nabla w\|_{L^4}^4.
\end{align}

6. Based on H\"older's inequality and integration by parts, we get that
\begin{align}
J_7&=-4\mu_r\int\eta(t)\dot{w}^j\big(\dot{w}^j-u\cdot\nabla w^j+{\rm div}\,(uw^j)\big)dx\nonumber\\
&=-4\mu_r\int\eta(t)|\dot{w}|^2dx-4\mu_r\int\eta(t)\dot{w}\cdot w{\rm div}\,udx\nonumber\\
&\le -4\mu_r\int\eta(t)|\dot{w}|^2dx+\delta\eta(t)\|\dot{w}\|_{L^2}^2
+C\|\nabla u\|_{L^4}^4+C\|\nabla w\|_{L^2}^4.
\end{align}
By H\"older's inequality, \eqref{3.2}, Sobolev's inequality, \eqref{z2.5}, and \eqref{3.8}, we derive from Lemmas \ref{l24} and \ref{l25} that
\begin{align}
J_8&\le C\eta(t)\|\dot{u}\|_{L^6}\|b\|_{L^3}\|\nabla b_t\|_{L^2}
+C\eta(t)\|\nabla\dot{u}\|_{L^2}\|b\|_{L^6}\|\nabla b\|_{L^6}\|u\|_{L^6}\nonumber\\
&\le C\eta(t)\big(\|\nabla\dot{u}\|_{L^2}+\|\nabla u\|_{L^2}^2\big)
\|b\|_{L^2}^{\frac12}\|b\|_{L^6}^{\frac12}\|\nabla b_t\|_{L^2} \notag \\
& \quad +C\eta(t)\|\nabla\dot{u}\|_{L^2}\big(\|\nabla b\|_{L^2}+\|{\rm curl}^2b\|_{L^2}\big)\|\nabla u\|_{L^2}\nonumber\\
&\le \delta\eta(t)\|\nabla\dot{u}\|_{L^2}^2+C\eta(t)E_0^\frac12\|\nabla b\|_{L^2}\|\nabla b_t\|_{L^2}^2+C\eta(t)\|\nabla u\|_{L^2}^4
\nonumber\\
&\quad+C\eta(t)\|\nabla b\|_{L^2}^4+C\eta(t)\|\nabla u\|_{L^2}^2\|{\rm curl}^2b\|_{L^2}^2+C\eta(t)\|\nabla b\|_{L^2}^4\|\nabla u\|_{L^2}^2.
\end{align}
Using \eqref{2.2} and \eqref{3.9}, we see that
\begin{align}
J_9&=\int\eta(t)(\rho-\rho_s)u\cdot\nabla^2\Phi_s\cdot\dot{u}dx
-\int\eta(t)\dot{u}\cdot\nabla\Phi_s{\rm div}\,(\rho_su)dx\nonumber\\
&\le C\eta(t)\big(\|\dot{u}\|_{L^6}\|\rho-\rho_s\|_{L^2}\|\nabla u\|_{L^2}+\|\nabla\dot{u}\|_{L^2}\|\nabla u\|_{L^2}\big)\|\nabla\Phi_s\|_{H^2}\nonumber\\
&\le \delta\eta(t)\|\nabla\dot{u}\|_{L^2}^2+C\eta(t)\|\rho-\rho_s\|_{L^2}\|\nabla u\|_{L^2}^3+C\eta(t)\|\nabla u\|_{L^2}^2\nonumber\\
&\le \delta\eta(t)\|\nabla\dot{u}\|_{L^2}^2+C\eta(t)\|\nabla u\|_{L^2}^2+C\eta(t)\|\nabla u\|_{L^2}^4,
\end{align}
and similarly, by \eqref{3.2}, \eqref{3.9}, and Lemma \ref{l29}, one deduces that
\begin{align}
J_{10}&=\int\eta(t)\rho\dot{u}\cdot\nabla(\Phi-\Phi_s)dx+\int\eta(t)\rho u\cdot\nabla^2(\Phi-\Phi_s)\cdot\dot{u}dx\nonumber\\
&\le C\eta(t)\big(\|\sqrt{\rho}\dot{u}\|_{L^2}\|\nabla(\Phi-\Phi_s)_t\|_{L^2}
+\|\nabla^2(\Phi-\Phi_s)\|_{L^2}\|\sqrt{\rho}u\|_{L^3}\|\dot{u}\|_{L^6}\big)\nonumber\\
&\le C\eta(t)\|\sqrt{\rho}\dot{u}\|_{L^2}\|\rho u\|_{L^2}
+C\eta(t)\|\rho-\rho_s\|_{L^2}\|\nabla u\|_{L^2}\big(\|\nabla\dot{u}\|_{L^2}+\|\nabla u\|_{L^2}^2\big)\nonumber\\
&\le \delta\eta(t)\|\nabla\dot{u}\|_{L^2}^2
+C\eta(t)\|\nabla u\|_{L^2}^2+C\eta(t)\|\nabla u\|_{L^2}^4.
\end{align}
Putting above estimates on $J_i\ (i=1, 2,\cdots, 10)$ into \eqref{3.44}, we obtain after choosing $\delta$ suitably small that
\begin{align}\label{3.59}
&\frac{d}{dt}\Big[\frac{\eta(t)}{2}\big(\|\sqrt{\rho}\dot{u}\|_{L^2}^2
+\|\sqrt{\rho}\dot{w}\|_{L^2}^2\big)\Big]
+(2\mu+\lambda)\eta(t)\|{\rm div}\,\dot{u}\|_{L^2}^2+\mu\eta(t)\|{\rm curl}\,\dot{u}\|_{L^2}^2\nonumber\\
&\quad+(2c_d+c_0)\eta(t)\|{\rm div}\,\dot{w}\|_{L^2}^2+(c_d+c_a)\eta(t)\|{\rm curl}\,\dot{w}\|_{L^2}
+4\mu_r\eta(t)\Big\|\frac12{\rm curl}\,\dot{u}-\dot{w}\Big\|_{L^2}^2\nonumber\\
&\le -\frac{d}{dt}\int_{\partial\Omega}\eta(t)(u\cdot\nabla n\cdot u)F_1dS-\frac{d}{dt}\int_{\partial\Omega}\eta(t)(w\cdot\nabla n\cdot w)F_2dS
+ C\delta\eta(t)\big(\|\nabla\dot{u}\|_{L^2}^2+\|\nabla\dot{w}\|_{L^2}^2\big)\nonumber\\
&\quad+C|\eta'(t)|\big(\|\sqrt{\rho}\dot{u}\|_{L^2}^2+\|\sqrt{\rho}\dot{w}\|_{L^2}^2+\|\nabla u\|_{L^2}^2+\|\nabla w\|_{L^2}^2+\|\nabla b\|_{L^2}^2+\|\nabla u\|_{L^2}^4
\big)\nonumber\\
&\quad+C|\eta'(t)|\big(\|{\rm curl}^2b\|_{L^2}^2+\|\nabla u\|_{L^2}^4\|\nabla b\|_{L^2}^2+\|\nabla u\|_{L^2}^2\|\nabla b\|_{L^2}^2
+\|\nabla w\|_{L^2}^4\big)\nonumber\\
&\quad+C\eta(t)\big(\|\sqrt{\rho}\dot{u}\|_{L^2}^2\|\nabla u\|_{L^2}^2+\|\nabla u\|_{L^2}^6+\|\nabla u\|_{L^4}^4
+\|\nabla u\|_{L^2}^2+\|\nabla w\|_{L^2}^2
+\|\nabla b\|_{L^2}^2\big)\nonumber\\
&\quad+C\eta(t)\big(\|\nabla u\|_{L^2}^4+\|\nabla w\|_{L^2}^4+\|\nabla b\|_{L^2}^4+\|{\rm curl}^2b\|_{L^2}^2\|\nabla b\|_{L^2}\|\nabla u\|_{L^2}^2
+\|\nabla b\|_{L^2}^6\big)\nonumber\\
&\quad+C\eta(t)\big(\|\nabla u\|_{L^2}^4\|\rho\dot{u}\|_{L^2}^2+\|\nabla w\|_{L^2}^6
+\|\nabla u\|_{L^2}^4\|{\rm curl}^2b\|_{L^2}^2+\|b_t\|_{L^2}^2\|\nabla b\|_{L^2}^4\big)\nonumber\\
&\quad+C\delta\eta(t)\|\nabla b\|_{L^2}\|{\rm curl}^2b\|_{L^2}^3
+C\eta(t)\big(\|\nabla w\|_{L^4}^4+\|\nabla b\|_{L^2}^2\|{\rm curl}^2b\|_{L^2}^2\big)\nonumber\\
&\quad+C\eta(t)\|\nabla b\|_{L^2}^4\|{\rm curl}^2b\|_{L^2}^2+C\eta(t)E_0^\frac12\|\nabla b\|_{L^2}\|\nabla b_t\|_{L^2}^2
+C\eta(t)(\|\dot{u}\|_{L^2}^2+\|\dot{w}\|_{L^2}^2),
\end{align}
due to
\begin{align*}
-\mu_r\int\eta(t)|{\rm curl}\,\dot{u}|^2dx+4\mu_r\int\eta(t)\dot{w}\cdot{\rm curl}\,\dot{u}dx-4\mu_r\int\eta(t)|\dot{w}|^2dx
=-4\mu_r\int\eta(t)\Big|\frac12{\rm curl}\,\dot{u}-\dot{w}\Big|^2dx.
\end{align*}

7. Noticing that
\begin{align}\label{z3.36}
\begin{cases}
b_{tt}-\curl{\rm curl}\,b_t=\big(b\cdot\nabla u-u\cdot\nabla b-b\cdot\nabla u\big)_t \quad&{\rm in}~\Omega,\\
b_t\cdot n=0, \quad {\rm curl}\,b_t\times n=0 \quad&{\rm on}~\partial\Omega.
\end{cases}
\end{align}
Multiplying \eqref{z3.36}$_1$ by $\eta(t)b_t$ and integration by parts, we find that
\begin{align}\label{3.61}
&\frac{d}{dt}\Big(\frac{\eta(t)}{2}\|b_t\|_{L^2}^2\Big)
+\eta(t)\|{\rm curl}\,b_t\|_{L^2}^2-\frac12\eta'(t)\|b_t\|_{L^2}^2\nonumber\\
&=\int\eta(t)(b_t\cdot\nabla u-u\cdot\nabla b_t-b_t{\rm div}\,u)\cdot b_tdx+\int\eta(t)(b\cdot\nabla\dot{u}-\dot{u}\cdot\nabla b-b{\rm div}\,\dot{u})\cdot b_tdx\nonumber\\
&\quad-\int\eta(t)(b\cdot\nabla(u\cdot\nabla u)-u\cdot\nabla u\cdot\nabla b-b{\rm div}\,(u\cdot\nabla u))\cdot b_tdx\triangleq\sum_{i=1}^3R_i.
\end{align}
It follows from H\"older's inequality, \eqref{2.3}, \eqref{a6}, and Young's inequality that
\begin{align*}
R_1&\le C\eta(t)\big(\|b_t\|_{L^3}\|b_t\|_{L^6}\|\nabla u\|_{L^2}+\|u\|_{L^6}\|b_t\|_{L^3}\|\nabla b_t\|_{L^2}\big)\nonumber\\
&\le C\eta(t)\|b_t\|_{L^2}^{\frac12}\|\nabla b_t\|_{L^2}^{\frac32}\|\nabla u\|_{L^2}\nonumber\\
&\le \delta\eta(t)\|\nabla b_t\|_{L^2}^2+C\eta(t)\|\nabla u\|_{L^2}^4\|b_t\|_{L^2}^2,\\
R_2&\le C\eta(t)\|b\|_{L^6}\|b_t\|_{L^3}\|\nabla\dot{u}\|_{L^2}+C\eta(t)\|\dot{u}\|_{L^6}\|\nabla b\|_{L^2}\|b_t\|_{L^3}\nonumber\\
&\le C\eta(t)\big(\|\nabla\dot{u}\|_{L^2}
+\|\nabla u\|_{L^2}^2\big)\|\nabla b\|_{L^2}\|b_t\|_{L^2}^\frac12\|\nabla b_t\|_{L^2}^\frac12\nonumber\\
&\le \delta\eta(t)\big(\|\nabla\dot{u}\|_{L^2}^2+\|\nabla b_t\|_{L^2}^2\big)+C\eta(t)\|\nabla u\|_{L^2}^4
+C\eta(t)\|\nabla b\|_{L^2}^4\|b_t\|_{L^2}^2,\\
R_3&=-\int_{\partial\Omega}\eta(t)b\cdot b_t(u\cdot\nabla u\cdot n)dS+\int\eta(t)b\cdot\nabla b_t\cdot(u\cdot\nabla u)dx-\int\eta(t)(u\cdot\nabla u)\cdot\nabla b\cdot b_tdx\nonumber\\
&\quad+\int\eta(t)u\cdot\nabla u\cdot\nabla b_t\cdot bdx+\int\eta(t)u\cdot\nabla u\cdot\nabla b\cdot b_tdx\nonumber\\
&\le \int_{\partial\Omega}\eta(t)b\cdot b_t(u\cdot\nabla n\cdot u)dS
+2\eta(t)\|b\|_{L^{12}}\|\nabla b_t\|_{L^2}\|u\|_{L^6}\|\nabla u\|_{L^4}\nonumber\\
&\le \int_{\partial\Omega}\eta(t)b\cdot b_t(u\cdot\nabla n\cdot u)dS
+\delta\eta(t)\|\nabla b_t\|_{L^2}^2+C\eta(t)\|\nabla u\|_{L^2}^4\|{\rm curl}^2b\|_{L^2}^2\nonumber\\
&\quad+C\eta(t)\|\nabla u\|_{L^4}^4+C\eta(t)\big(\|\nabla u\|_{L^2}^4+\|\nabla b\|_{L^2}^4\big),
\end{align*}
where we have used \eqref{2.5}, \eqref{3.16}, and
\begin{align*}
\|b\|_{L^{12}}& \leq C\||b|^2\|_{H^1}^{\frac12}\le C\|b\|_{L^4}+C\|\nabla |b|^2\|_{L^2}^{\frac12} \\
& \le C\|\nabla b\|_{L^2}+C\|b\|_{L^6}^{\frac12}\|\nabla b\|_{L^3}^{\frac12} \\
& \le C\|\nabla b\|_{L^2}+C\|\nabla b\|_{L^2}^\frac34\|\nabla b\|_{L^6}^\frac14 \\
& \le C\|\nabla b\|_{L^2}+C\|\nabla b\|_{L^2}^2+C\|\nabla b\|_{L^2}^\frac34\|{\rm curl}^2b\|_{L^2}^\frac14,
\end{align*}
due to Sobolev's inequality, Lemma \ref{l21}, and Lemma \ref{l25}.
Thus, substituting the above estimates on $R_i\ (i=1, 2, 3)$ into \eqref{3.61} shows that
\begin{align}\label{3.62}
&\frac{d}{dt}\Big(\frac{\eta(t)}{2}\|b_t\|_{L^2}^2\Big)
+\eta(t)\|{\rm curl}\, b_t\|_{L^2}^2-\frac12\eta'(t)\|b_t\|_{L^2}^2\nonumber\\
&\le \int_{\partial\Omega}\eta(t)(b\cdot b_t)(u\cdot\nabla n\cdot u)dS+ C\delta \eta(t)\big(\|\nabla\dot{u}\|_{L^2}^2+\|\nabla b_t\|_{L^2}^2\big)+C\eta(t)\|\nabla b\|_{L^2}^4\|b_t\|_{L^2}^2\nonumber\\
&\quad+C\eta(t)\|\nabla u\|_{L^4}^4+C\eta(t)\|\nabla u\|_{L^2}^4\|{\rm curl}^2b\|_{L^2}^2+C\eta(t)\big(\|\nabla u\|_{L^2}^4+\|\nabla b\|_{L^2}^4\big).
\end{align}
By Lemma \ref{l23}, H{\"o}lder's inequality, Sobolev's inequality, and \eqref{z2.5}, we arrive at
\begin{align*}
\left|\int_{\partial\Omega}(b\cdot b_t)(u\cdot\nabla n\cdot u)dS\right|
&\le C\int_{\partial\Omega}|u|^2|b||b_t|dS\le C\||u|^2|b||b_t|\|_{W^{1,1}} \\
&\le C\big(\|u\|_{L^4}^2\|b\|_{L^3}\|b_t\|_{L^6}
+\|u\|_{L^6}\|\nabla u\|_{L^2}\|b\|_{L^6}\|b_t\|_{L^6}\big) \\
&\quad+C\big(\|u\|_{L^6}^2\|\nabla b_t\|_{L^2}\|b\|_{L^6}+\|u\|_{L^6}^2\|\nabla b\|_{L^2}\|b_t\|_{L^6}\big) \\
&\le C\|u\|_{H^1}^2\|b\|_{H^1}\|b_t\|_{H^1}
\le C\|\nabla u\|_{L^2}^2\|\nabla b\|_{L^2}\|\nabla b_t\|_{L^2} \\
&\le \delta\|\nabla b_t\|_{L^2}^2+C\|\nabla u\|_{L^2}^4\|\nabla b\|_{L^2}^2.
\end{align*}
This together with \eqref{3.62} yields that
\begin{align*}
&\frac{d}{dt}\Big(\frac{\eta(t)}{2}\|b_t\|_{L^2}^2\Big)
+\eta(t)\|{\rm curl}\,b_t\|_{L^2}^2-\frac12\eta'(t)\|b_t\|_{L^2}^2\nonumber\\
&\le C\delta \eta(t)\big(\|\nabla\dot{u}\|_{L^2}^2+\|\nabla b_t\|_{L^2}^2\big)
+C\eta(t)\|\nabla b\|_{L^2}^4\|b_t\|_{L^2}^2+C\eta(t)\|\nabla u\|_{L^2}^4\|{\rm curl}^2b\|_{L^2}^2\nonumber\\
&\quad
+C\eta(t)\big(\|\nabla u\|_{L^2}^4+\|\nabla b\|_{L^2}^4\big)
+C\eta(t)\|\nabla u\|_{L^4}^4+C\eta(t)\|\nabla u\|_{L^2}^4\|\nabla b\|_{L^2}^2,
\end{align*}
which combined with \eqref{3.59} leads to \eqref{z3.18}.
\end{proof}

\begin{lemma}\label{l34}
Let the assumptions of Proposition \ref{p31} hold. Then there exists a positive constant $\varepsilon_2$ such that
\begin{align}\label{3.65}
A_2(\sigma(T))+\int_0^{\sigma(T)}\big(\|\sqrt{\rho}\dot{u}\|_{L^2}^2
+\|\sqrt{\rho}\dot{w}\|_{L^2}^2+\|b_t\|_{L^2}^2
+\|{\rm curl}^2b\|_{L^2}^2\big)dt\le 3K
\end{align}
provided that $E_0\le \varepsilon_2$.
\end{lemma}
\begin{proof}[Proof]
In view of Lemma \ref{l28}, we have
\begin{align}\label{3.66}
\|\nabla u\|_{L^3}^3&\le C\big(\|\sqrt{\rho}\dot{u}\|_{L^2}+\|\sqrt{\rho}\dot{w}\|_{L^2}+\|\nabla w\|_{L^2}+\|b\cdot\nabla b\|_{L^2}
+\|\rho\nabla(\Phi-\Phi_s)\|_{L^2}+\|(\rho-\rho_s)\nabla\Phi_s\|_{L^2}\big)^\frac{3}{2}\nonumber\\
&\quad\times\big(\|\nabla u\|_{L^2}+\|P-P_s\|_{L^2}+\|b\|_{L^4}^2\big)^\frac32+C\big(\|\nabla u\|_{L^2}^3+\|P-P_s\|_{L^3}^3+\|b\|_{L^6}^6\big)\nonumber\\
&\le  C\big(\|\sqrt{\rho}\dot{u}\|_{L^2}+\|\sqrt{\rho}\dot{w}\|_{L^2}+\|\nabla w\|_{L^2}+\|b\cdot\nabla b\|_{L^2}
+\|\rho\nabla(\Phi-\Phi_s)\|_{L^2}+\|(\rho-\rho_s)\nabla\Phi_s\|_{L^2}\big)^\frac{3}{2}\nonumber\\
&\quad\times\big(\|\nabla u\|_{L^2}+\|\rho-\rho_s\|_{L^2}+\|b\|_{L^4}^2\big)^\frac32+C\big(\|\nabla u\|_{L^2}^3+\|b\|_{L^6}^6
+\|\rho-\rho_s\|_{L^2}^2\big)\nonumber\\
&\le \delta\big(\|\sqrt{\rho}\dot{u}\|_{L^2}^2+\|\sqrt{\rho}\dot{w}\|_{L^2}^2\big)+C\big(\|\nabla w\|_{L^2}^2+\|b\cdot\nabla b\|_{L^2}^2
+\|\nabla u\|_{L^2}^6+\|b\|_{L^4}^{12}\big)\nonumber\\
&\quad+C\big(\|\nabla u\|_{L^2}^2+\|\nabla u\|_{L^2}^4+\|\nabla b\|_{L^2}^6\big)+CE_0\nonumber\\
&\le \delta\big(\|\sqrt{\rho}\dot{u}\|_{L^2}^2+\|\sqrt{\rho}\dot{w}\|_{L^2}^2\big)+C\|\nabla w\|_{L^2}^2+CE_0^\frac12\|{\rm curl}^2b\|_{L^2}^2
+C\|\nabla b\|_{L^2}^4\nonumber\\
&\quad+C\big(\|\nabla u\|_{L^2}^2+\|\nabla u\|_{L^2}^4+\|\nabla b\|_{L^2}^6\big)+CE_0,
\end{align}
due to
\begin{align*}
\|P-P_s\|_{L^3}\le C\|P-P_s\|_{L^\infty}^\frac13\Big(\int|P-P_s|^2dx\Big)^\frac13\le C(\hat{\rho})\|\rho-\rho_s\|_{L^2}^\frac23.
\end{align*}
Similarly, by \eqref{2.22}, we have
\begin{align}\label{3.67}
\|\nabla w\|_{L^3}^3&\le C\big(\|\sqrt{\rho}\dot{u}\|_{L^2}+\|\sqrt{\rho}\dot{w}\|_{L^2}
+\|\rho\nabla(\Phi-\Phi_s)\|_{L^2}+\|(\rho-\rho_s)\nabla\Phi_s\|_{L^2}\nonumber\\
&\quad+\|\nabla u\|_{L^2}+\|b\cdot\nabla b\|_{L^2}+\|w\|_{L^2}\big)^\frac32\|\nabla w\|_{L^2}^\frac32+C\|\nabla w\|_{L^2}^3\nonumber\\
&\le \delta\big(\|\sqrt{\rho}\dot{w}\|_{L^2}^2+\|\sqrt{\rho}\dot{u}\|_{L^2}^2\big)+C\|\nabla w\|_{L^2}^2+C\|\nabla u\|_{L^2}^2+
CE_0^\frac12\|{\rm curl}^2b\|_{L^2}^2\nonumber\\
&\quad+C\big(\|\nabla w\|_{L^2}^4+\|\nabla w\|_{L^2}^6\big)+CE_0^2.
\end{align}
By \eqref{3.66}, \eqref{3.67}, and Lemma \ref{l33}, one can check that
\begin{align}\label{3.68}
\int_0^{\sigma(T)}\big(\|\nabla u\|_{L^3}^3+\|\nabla w\|_{L^3}^3\big)dt
&\le \delta C(\hat{\rho})\int_0^{\sigma(T)}\big(\|\sqrt{\rho}\dot{w}\|_{L^2}^2
+\|\sqrt{\rho}\dot{u}\|_{L^2}^2\big)dt\nonumber\\
& \quad +C\int_0^{\sigma(T)}\big(\|\nabla w\|_{L^2}^4+\|\nabla u\|_{L^2}^4+\|\nabla b\|_{L^2}^4\big)dt\nonumber\\
&\quad+C\int_0^{\sigma(T)}\big(\|\nabla w\|_{L^2}^6+\|\nabla b\|_{L^2}^4\big)dt+CE_0^\frac12.
\end{align}
Taking $\eta(t)=1$ and integrating \eqref{3.15} over $[0, t]$ for $0<t\le \sigma(T)$, we deduce from \eqref{3.5},
\eqref{3.6}, and \eqref{3.68} that
\begin{align*}
&A_2(t)+\int_0^t\big(\|\sqrt{\rho}\dot{u}\|_{L^2}^2+\|\sqrt{\rho}\dot{w}\|_{L^2}^2
+\|{\rm curl}^2b\|_{L^2}^2+\|b_t\|_{L^2}^2\big)d\tau\nonumber\\
&\le \int(P-P_s){\rm div}\,udx\Big|_{t=0}^t+\int(\rho-\rho_s)u\cdot\nabla\Phi_sdx\Big|_{t=0}^t
-\int\Big(b\otimes b:\nabla u-\frac12|b|^2{\rm div}\,u\Big)dx\Big|_{t=0}^t\nonumber\\
&\quad+\big(CC_0+\delta\big)\int_0^t\big(\|{\rm curl}^2b\|_{L^2}^2+\|b_t\|_{L^2}^2\big)d\tau+C\int_0^t\big(\|\nabla u\|_{L^3}^3+\|\nabla w\|_{L^3}^3\big)d\tau\nonumber\\
&\quad+C\int_0^t\big(\|\nabla u\|_{L^2}^4+\|\nabla b\|_{L^2}^4+\|\nabla w\|_{L^2}^4+\|\nabla u\|_{L^2}^6+\|\nabla w\|_{L^2}^6
+\|\nabla b\|_{L^2}^6\big)d\tau+CE_0\nonumber\\
&\le K+\frac14A_2(t)+CE_0A_2(t)+CE_0A_2^2(t)+CE_0^\frac12,
\end{align*}
which leads to
\begin{align}\label{3.70}
&A_2(\sigma(T))+\int_0^{\sigma(T)}\big(\|\sqrt{\rho}\dot{u}\|_{L^2}^2+\|\sqrt{\rho}\dot{w}\|_{L^2}^2+\|{\rm curl}^2b\|_{L^2}^2+\|b_t\|_{L^2}^2\big)dt\nonumber\\
&\le 2K+CE_0A_2^2(t)+CE_0^\frac12\le \frac52K+CE_0KA_2(\sigma(T))
\end{align}
provided that $E_0$ is suitably small. We immediately obtain \eqref{3.65} from \eqref{3.70}.
\end{proof}

Next, motivated by \cite{YZ17}, we give the upper bounds of $A_1(T)$.
\begin{lemma}\label{l35}
Let the assumptions of Proposition \ref{p31} hold. For $\sigma_i\triangleq\sigma(t+1-i)$ with $i$ being an integer satisfying $1\le i\le [T]-1$, then there exists a positive constant $\varepsilon_3$ such that
\begin{align}\label{z3.44}
A_1(T)+\int_{i-1}^{i+1}\sigma_i\big(\|\sqrt{\rho}\dot{u}\|_{L^2}^2+\|\sqrt{\rho}\dot{w}\|_{L^2}^2+\|{\rm curl}^2b\|_{L^2}^2+\|b_t\|_{L^2}^2\big)dt\le E_0^\frac13
\end{align}
provided that $E_0\le \varepsilon_3$.
\end{lemma}
\begin{remark}
For simplicity, we only prove the case that $T>2$. Otherwise, the same thing can be done by choosing a suitably small step size.
\end{remark}
\begin{proof}[Proof]
For integer $i\ (1\le i\le [T]-1)$, taking $\eta(t)=\sigma_i$ and integrating \eqref{3.15} over $(i-1, i+1]$, we obtain after choosing $\delta$ suitably small that
\begin{align}\label{3.72}
&\sup_{i-1\le t\le i+1}\big(\sigma_i\|\nabla u\|_{L^2}^2
+\sigma_i\|\nabla w\|_{L^2}^2+\sigma_i\|\nabla b\|_{L^2}^2\big)\nonumber\\
&\quad+\int_{i-1}^{i+1}\sigma_i\big(\|\sqrt{\rho}\dot{u}\|_{L^2}^2+\|\sqrt{\rho}\dot{w}\|_{L^2}^2+\|{\rm curl}^2b\|_{L^2}^2+\|b_t\|_{L^2}^2\big)dt\nonumber\\
&\le \sigma_i\int(P-P_s){\rm div}\,udx+\sigma_i\int(\rho-\rho_s)u\cdot\nabla\Phi_sdx\nonumber\\
&\quad-\sigma_i\int\Big(b\otimes b:\nabla u-\frac12|b|^2{\rm div}\,u\Big)dx+C\int_{i-1}^{i+1}\big(\sigma_i+\sigma_i'\big)\|\nabla u\|_{L^2}^2dt\nonumber\\
&\quad+C\int_{i-1}^{i+1}\sigma_i'\big(\|w\|_{L^2}^2+\|\nabla w\|_{L^2}^2+\|\nabla b\|_{L^2}^2+\|\nabla u\|_{L^2}^4\big)dt\nonumber\\
&\quad+C\int_{i-1}^{i+1}\sigma_i\big(\|\nabla u\|_{L^3}^3+\|\nabla w\|_{L^3}^3\big)dt
+C\int_{i-1}^{i+1}\sigma_i\big(\|\nabla w\|_{L^2}^2+\|\nabla b\|_{L^2}^2\big)dt\nonumber\\
&\quad+C\int_{i-1}^{i+1}\sigma_i\big(\|\nabla u\|_{L^2}^4+\|\nabla b\|_{L^2}^4+\|\nabla w\|_{L^2}^4\big)dt\nonumber\\
&\quad+C\int_{i-1}^{i+1}\sigma_i\big(\|\nabla u\|_{L^2}^6+\|\nabla b\|_{L^2}^6+\|\nabla w\|_{L^2}^6\big)dt+CE_0\nonumber\\
&\quad+\delta\int_{i-1}^{i+1}\sigma_i\big(\|\sqrt{\rho}\dot{u}\|_{L^2}^2
+\|\sqrt{\rho}\dot{w}\|_{L^2}^2\big)dt
+(CE_0+\delta)\int_{i-1}^{i+1}\sigma_i\big(\|{\rm curl}^2b\|_{L^2}^2+\|b_t\|_{L^2}^2\big)dt\nonumber\\
&\le \frac12\sup_{i-1\le t\le i+1}\big(\sigma_i\|\nabla u\|_{L^2}^2\big)+C\sup_{i-1\le t\le i+1}\big(\sigma_i\|b\|_{L^4}^4\big)
+C\int_{i-1}^{i+1}\|\nabla u\|_{L^2}^2dt\nonumber\\
&\quad+C\sup_{i-1\le t\le i+1}\big(\|\nabla u\|_{L^2}^2
+\|\nabla w\|_{L^2}^2+\|\nabla b\|_{L^2}^2\big)\int_{i-1}^{i+1}\big(\|\nabla u\|_{L^2}^2+\|\nabla w\|_{L^2}^2+\|\nabla b\|_{L^2}^2\big)dt\nonumber\\
&\quad+C\sup_{i-1\le t\le i+1}\big(\|\nabla u\|_{L^2}^4
+\|\nabla w\|_{L^2}^4+\|\nabla b\|_{L^2}^4\big)\int_{i-1}^{i+1}\big(\|\nabla u\|_{L^2}^2+\|\nabla w\|_{L^2}^2+\|\nabla b\|_{L^2}^2\big)dt\nonumber\\
&\quad+C\sup_{i-1\le t\le i}\|\nabla u\|_{L^2}^2\int_{i-1}^i\|\nabla u\|_{L^2}^2dt+CE_0+\delta\int_{i-1}^{i+1}\sigma_i\big(\|\sqrt{\rho}\dot{u}\|_{L^2}^2
+\|\sqrt{\rho}\dot{w}\|_{L^2}^2\big)dt\nonumber\\
&\quad
+(CE_0+\delta)\int_{i-1}^{i+1}\sigma_i\big(\|{\rm curl}^2b\|_{L^2}^2+\|b_t\|_{L^2}^2\big)dt\nonumber\\
&\le  \frac12\sup_{i-1\le t\le i+1}\big(\sigma_i\|\nabla u\|_{L^2}^2\big)+CE_0^\frac12\sup_{i-1\le t\le i+1}\big(\sigma_i\|\nabla b\|_{L^2}^2\big)\nonumber\\
&\quad+CE_0\big(A_1(T)+A_2(\sigma(T))\big)+CE_0\big(A_1^2(T)+A_2^2(\sigma(T))\big)\nonumber\\
&\quad+(CE_0+\delta)\int_{i-1}^{i+1}\sigma_i\big(\|{\rm curl}^2b\|_{L^2}^2+\|b_t\|_{L^2}^2\big)dt
+\delta\int_{i-1}^{i+1}\sigma_i\big(\|\sqrt{\rho}\dot{u}\|_{L^2}^2
+\|\sqrt{\rho}\dot{w}\|_{L^2}^2\big)dt+CE_0\nonumber\\
&\le  \frac12\sup_{i-1\le t\le i+1}\big(\sigma_i\|\nabla u\|_{L^2}^2\big)+CE_0^\frac12\sup_{i-1\le t\le i+1}\big(\sigma_i\|\nabla b\|_{L^2}^2\big)\nonumber\\
&\quad+(CE_0+\delta)\int_{i-1}^{i+1}\sigma_i\big(\|{\rm curl}^2b\|_{L^2}^2+\|b_t\|_{L^2}^2\big)dt
+\delta\int_{i-1}^{i+1}\sigma_i\big(\|\sqrt{\rho}\dot{u}\|_{L^2}^2
+\|\sqrt{\rho}\dot{w}\|_{L^2}^2\big)dt+CE_0,
\end{align}
where we have used \eqref{3.6}, \eqref{3.8}, and \eqref{3.16}. Choosing $\delta$ and $E_0$ to be suitably small, we deduce from \eqref{3.72} that
\begin{align}\label{3.73}
\sup_{0\le t\le \sigma(T)}\big(\sigma\|\nabla u\|_{L^2}^2
+\sigma\|\nabla w\|_{L^2}^2+\sigma\|\nabla b\|_{L^2}^2\big)\le CE_0\le E_0^\frac13,
\end{align}
due to $\sigma_1(t)=\sigma(t)$ and
\begin{align}\label{3.74}
&\sup_{i\le t\le i+1}\big(\sigma_i\|\nabla u\|_{L^2}^2
+\sigma_i\|\nabla w\|_{L^2}^2+\sigma_i\|\nabla b\|_{L^2}^2\big)\nonumber\\
&\quad+\int_{i-1}^{i+1}\sigma_i\big(\|\sqrt{\rho}\dot{u}\|_{L^2}^2
+\|\sqrt{\rho}\dot{w}\|_{L^2}^2+\|{\rm curl}^2b\|_{L^2}^2+\|b_t\|_{L^2}^2\big)dt\le C_2E_0\le E_0^\frac13
\end{align}
provided that $E_0$ is properly small. Note that the constant $C_2$ is independent of $i$.
Thus, the desired \eqref{z3.44} follows from \eqref{3.73} and \eqref{3.74}.
\end{proof}

\begin{lemma}\label{l36}
Let the assumptions of Proposition \ref{p31} hold. Then there exists a positive constant $\varepsilon_4$ such that, for $0\le t_1<t_2\le T$,
\begin{align}\label{3.75}
\sup_{0\le t\le T}\big[\sigma^2\big(\|\sqrt{\rho}\dot{u}\|_{L^2}^2+\|\sqrt{\rho}\dot{w}\|_{L^2}^2
+\|b_t\|_{L^2}^2+\|{\rm curl}^2b\|_{L^2}^2\big)\big]\le C(\underline{\rho})E_0^\frac13,
\end{align}
and
\begin{align}\label{3.76}
\int_{t_1}^{t_2}\sigma^2\big(\|\nabla\dot{u}\|_{L^2}^2+\|\nabla\dot{w}\|_{L^2}^2+\|\dot{w}\|_{L^2}^2+\|\nabla b_t\|_{L^2}^2\big)dt
\le C(\underline{\rho})E_0^\frac13+C(\underline{\rho})E_0(t_2-t_1),
\end{align}
provided that $E_0\le \varepsilon_4$.
\end{lemma}
\begin{proof}[Proof]
1. We get from Lemma \ref{l28} that
\begin{align*}
\|\nabla u\|_{L^4}^4&\le C(\|\rho\dot{u}\|_{L^2}+\|\rho\dot{w}\|_{L^2}+\|\nabla w\|_{L^2}+\|b\cdot\nabla b\|_{L^2}
+\|\rho\nabla(\Phi-\Phi_s)\|_{L^2}+\|(\rho-\rho_s)\nabla\Phi_s\|_{L^2})^3\nonumber\\
&\quad\times(\|\nabla u\|_{L^2}+\|P-P_s\|_{L^2}+\|b\|_{L^4}^2)+C(\|\nabla u\|_{L^2}^4+\|P-P_s\|_{L^4}^4+\||b|^2\|_{L^4}^4),\\
\|\nabla w\|_{L^4}^4&\le C(\|\rho\dot{u}\|_{L^2}+\|\rho\dot{w}\|_{L^2}
+\|\rho\nabla(\Phi-\Phi_s)\|_{L^2}+\|(\rho-\rho_s)\nabla\Phi_s\|_{L^2}\nonumber\\
&\quad+\|\nabla u\|_{L^2}+\|b\cdot\nabla b\|_{L^2}+\|w\|_{L^2})^3\|\nabla w\|_{L^2}+C\|\nabla w\|_{L^2}^4,
\end{align*}
which implies that
\begin{align}
&\int_{i-1}^{i+1}\sigma_i^2\big(\|\nabla u\|_{L^4}^4+\|\nabla w\|_{L^4}^4\big)dt\nonumber\\
&\le C\int_{i-1}^{i+1}\sigma_i^2\Big(\|\nabla u\|_{L^2}+E_0^\frac12\Big)\big(\|\sqrt{\rho}\dot{u}\|_{L^2}^3
+\|\sqrt{\rho}\dot{w}\|_{L^2}^3+\|b\|_{L^3}^3\|\nabla b\|_{L^6}^3\big)dt\nonumber\\
&\quad+C\int_{i-1}^{i+1}\sigma_i^2\big(\|\nabla u\|_{L^2}^4+\|\nabla b\|_{L^2}^4+\|\nabla w\|_{L^2}^4+\|b\|_{L^8}^8\big)dt\nonumber\\
&\le C\sup_{i-1\le t\le i+1}\Big(\|\nabla u\|_{L^2}+E_0^\frac12\Big)\int_{i-1}^{i+1}\sigma_i^2\big(
\|\sqrt{\rho}\dot{u}\|_{L^2}^3
+\|\sqrt{\rho}\dot{w}\|_{L^2}^3+\|{\rm curl}^2b\|_{L^2}^3\big)dt\nonumber\\
&\quad+C\int_{i-1}^{i+1}\sigma_i^2\big(\|\nabla u\|_{L^2}^4+\|\nabla b\|_{L^2}^4+\|\nabla w\|_{L^2}^4\big)dt
+C\int_{i-1}^{i+1}\sigma_i^2\|b\|_{L^\infty}^4\|b\|_{L^4}^4dt\nonumber\\
&\le  C\sup_{i-1\le t\le i+1}\Big(\|\nabla u\|_{L^2}+E_0^\frac12\Big)\int_{i-1}^{i+1}\sigma_i^2\big(
\|\sqrt{\rho}\dot{u}\|_{L^2}^3
+\|\sqrt{\rho}\dot{w}\|_{L^2}^3+\|{\rm curl}^2b\|_{L^2}^3\big)dt\nonumber\\
&\quad+C\int_{i-1}^{i+1}\sigma_i^2\big(\|\nabla u\|_{L^2}^4+\|\nabla b\|_{L^2}^4+\|\nabla w\|_{L^2}^4\big)dt
+CE_0^\frac12\int_{i-1}^{i+1}\sigma_i^2\|{\rm curl}^2b\|_{L^2}^2dt.
\end{align}
For any integer $1\le i\le [T]-1$, integrating \eqref{z3.18} with $\eta(t)=\sigma_i^2$ over $(i-1, i+1]$,
we deduce from \eqref{3.6}, \eqref{3.9}, \eqref{3.65}, \eqref{z3.44}, and Young's inequality that
\begin{align}\label{3.78}
&\sup_{i-1\le t\le i+1}\big[\sigma_i^2\big(\|\sqrt{\rho}\dot{u}\|_{L^2}^2
+\|\sqrt{\rho}\dot{w}\|_{L^2}^2+\|b_t\|_{L^2}^2\big)\big]\nonumber\\
&\quad+\int_{i-1}^{i+1}\sigma_i^2\big(\|\nabla\dot{u}\|_{L^2}^2
+\|\nabla\dot{w}\|_{L^2}^2+\|\dot{w}\|_{L^2}^2+\|\nabla b_t\|_{L^2}^2\big)dt\nonumber\\
&\le -\int_{\partial\Omega}\sigma_i^2(u\cdot\nabla n\cdot u)F_1dS\Big|_{i-1}^{i+1}
-\int_{\partial\Omega}\sigma_i^2(w\cdot\nabla n\cdot w)F_2dS\Big|_{i-1}^{i+1}
+\delta C\sigma_i^2\int_{i-1}^{i+1}\big(\|\nabla\dot{u}\|_{L^2}^2
+\|\nabla\dot{w}\|_{L^2}^2\big)dt\nonumber\\
&\quad+C\int_{i-1}^{i+1}\sigma_i\sigma_i'\big(\|\sqrt{\rho}\dot{u}\|_{L^2}^2
+\|\sqrt{\rho}\dot{w}\|_{L^2}^2+\|\nabla u\|_{L^2}^2
+\|\nabla w\|_{L^2}^2+\|\nabla b\|_{L^2}^2+\|\nabla u\|_{L^2}^4\big)dt\nonumber\\
&\quad+C\int_{i-1}^{i+1}\sigma_i\sigma_i'\big(\|{\rm curl}^2b\|_{L^2}^2+\|b_t\|_{L^2}^2+\|\nabla u\|_{L^2}^4\|\nabla b\|_{L^2}^2+\|\nabla u\|_{L^2}^2\|\nabla b\|_{L^2}^2
+\|\nabla w\|_{L^2}^4\big)dt\nonumber\\
&\quad+C\int_{i-1}^{i+1}\sigma_i^2\big(\|\sqrt{\rho}\dot{u}\|_{L^2}^2\|\nabla u\|_{L^2}^2+\|\nabla u\|_{L^2}^6
+\|\nabla u\|_{L^2}^2+\|\nabla w\|_{L^2}^2
+\|\nabla b\|_{L^2}^2\big)dt\nonumber\\
&\quad+C\int_{i-1}^{i+1}\sigma_i^2\big(\|\nabla u\|_{L^2}^4+\|\nabla w\|_{L^2}^4+\|\nabla b\|_{L^2}^4
+\|{\rm curl}^2b\|_{L^2}^2\|\nabla u\|_{L^2}^2
+\|\nabla b\|_{L^2}^6\big)dt\nonumber\\
&\quad+C\int_{i-1}^{i+1}\sigma_i^2\big(\|\nabla u\|_{L^2}^4\|\rho\dot{u}\|_{L^2}^2+\|\nabla w\|_{L^2}^6
+\|\nabla u\|_{L^2}^4\|{\rm curl}^2b\|_{L^2}^2+\|b_t\|_{L^2}^2\|\nabla b\|_{L^2}^4\big)dt\nonumber\\
&\quad+\delta C\int_{i-1}^{i+1}\sigma_i^2\|\nabla b\|_{L^2}\|{\rm curl}^2b\|_{L^2}^3dt
+C\int_{i-1}^{i+1}\sigma_i^2\|\nabla b\|_{L^2}^2\|{\rm curl}^2b\|_{L^2}^2dt\nonumber\\
&\quad+C\int_{i-1}^{i+1}\sigma_i^2\big(\|\nabla b\|_{L^2}^4+\|\nabla u\|_{L^2}^4\big)\|{\rm curl}^2b\|_{L^2}^2dt
+CE_0^\frac12\int_{i-1}^{i+1}\sigma_i^2\|\nabla b_t\|_{L^2}^2dt\nonumber\\
&\quad+C\int_{i-1}^{i+1}\sigma_i^2\big(\|\nabla u\|_{L^4}^4+\|\nabla w\|_{L^4}^4\big)dt
+C\int_{i-1}^{i+1}\sigma_i^2\big(\|\dot{u}\|_{L^2}^2+\|\dot{w}\|_{L^2}^2\big)dt\nonumber\\
&\le \frac14\sup_{i-1\le t\le i+1}\sigma_i^2\big(\|\sqrt{\rho}\dot{u}\|_{L^2}^2
+\|\sqrt{\rho}\dot{w}\|_{L^2}^2\big)
+\delta\sup_{i-1\le t\le i+1}\big(\sigma_i^2\|{\rm curl}^2b\|_{L^2}^2\big)
\nonumber\\
&\quad+C\int_{i-1}^{i+1}\sigma_i^2\|{\rm curl}^2b\|_{L^2}^3dt+C\sup_{i-1\le t\le i+1}\big(\sigma_i^2\|\sqrt{\rho}\dot{u}\|_{L^2}^2\big)
\int_{i-1}^{i+1}\big(\|\nabla u\|_{L^2}^4
+\|\nabla u\|_{L^2}^2\big)dt\nonumber\\
&\quad+C\sup_{i-1\le t\le i+1}\big(\sigma_i^2\|{\rm curl}^2b\|_{L^2}^2\big)
\int_{i-1}^{i+1}\big(\|\nabla u\|_{L^2}^2+\|\nabla b\|_{L^2}^2+\|\nabla u\|_{L^2}^4+\|\nabla b\|_{L^2}^4\big)dt\nonumber\\
&\quad+ C\sup_{i-1\le t\le i+1}\Big(\|\nabla u\|_{L^2}+E_0^\frac12\Big)\int_{i-1}^{i+1}\sigma_i^2\big(
\|\sqrt{\rho}\dot{u}\|_{L^2}^3
+\|\sqrt{\rho}\dot{w}\|_{L^2}^3+\|{\rm curl}^2b\|_{L^2}^3\big)dt\nonumber\\
&\quad+C\int_{i-1}^{i+1}\sigma_i^2\big(\|\nabla u\|_{L^2}^4+\|\nabla b\|_{L^2}^4+\|\nabla w\|_{L^2}^4\big)dt
+CE_0^\frac12\sup_{i-1\le t\le i+1}\big(\sigma_i^2\|{\rm curl}^2b\|_{L^2}^2\big)\nonumber\\
&\quad+CE_0^\frac12\int_{i-1}^{i+1}\sigma_i^2\|\nabla b_t\|_{L^2}^2dt+C(\underline{\rho})E_0^\frac13\nonumber\\
&\le  \frac14\sup_{i-1\le t\le i+1}\sigma_i^2\big(\|\sqrt{\rho}\dot{u}\|_{L^2}^2
+\|\sqrt{\rho}\dot{w}\|_{L^2}^2\big)
+\Big(\delta+CE_0^\frac12\Big)\sup_{i-1\le t\le i+1}\big(\sigma_i^2\|b_t\|_{L^2}^2\big)
\nonumber\\
&\quad+C\sup_{i-1\le t\le i+1}\big(\sigma_i\|\sqrt{\rho}\dot{u}\|_{L^2}
+\sigma_i\|\sqrt{\rho}\dot{w}\|_{L^2}+\sigma_i\|b_t\|_{L^2}\big)
\int_{i-1}^{i+1}\sigma_i\big(\|\sqrt{\rho}\dot{u}\|_{L^2}^2
+\|\sqrt{\rho}\dot{w}\|_{L^2}^2+\|b_t\|_{L^2}^2\big)dt\nonumber\\
&\quad+C\sup_{i-1\le t\le i+1}\big(\sigma_i\|b_t\|_{L^2}\big)\int_{i-1}^{i+1}\sigma_i\|b_t\|_{L^2}^2dt
+CE_0^\frac12\int_{i-1}^{i+1}\sigma_i^2\|\nabla b_t\|_{L^2}^2dt+C(\underline{\rho})E_0^\frac13\nonumber\\
&\le  \frac14\sup_{i-1\le t\le i+1}\big[\sigma_i^2\big(\|\sqrt{\rho}\dot{u}\|_{L^2}^2
+\|\sqrt{\rho}\dot{w}\|_{L^2}^2\big)\big]
+\Big(\delta+CE_0^\frac12\Big)\sup_{i-1\le t\le i+1}
\big(\sigma_i^2\|b_t\|_{L^2}^2\big)
\nonumber\\
&\quad+CE_0^\frac13\sup_{i-1\le t\le i+1}\big(\sigma_i\|\sqrt{\rho}\dot{u}\|_{L^2}
+\sigma_i\|\sqrt{\rho}\dot{w}\|_{L^2}+\sigma_i\|b_t\|_{L^2}\big)\nonumber\\
&\quad+CE_0^\frac12\int_{i-1}^{i+1}\sigma_i^2\|\nabla b_t\|_{L^2}^2dt+C(\underline{\rho})E_0^\frac13\nonumber\\
&\le \frac12\sup_{i-1\le t\le i+1}\big(\sigma_i^2\|\sqrt{\rho}\dot{u}\|_{L^2}^2
+\sigma_i^2\|\sqrt{\rho}\dot{w}\|_{L^2}^2+\sigma_i^2\|b_t\|_{L^2}^2\big)
+CE_0^\frac12\int_{i-1}^{i+1}\sigma_i^2\|\nabla b_t\|_{L^2}^2dt+C(\underline{\rho})E_0^\frac13,
\end{align}
where we have used
\begin{align*}
&\int_{\partial\Omega}(u\cdot\nabla n\cdot u)F_1dS+\int_{\partial\Omega}(w\cdot\nabla n\cdot w)F_2dS\nonumber\\
&\le C\||u|^2|F_1|\|_{W^{1,1}}+\||w|^2|F_2|\|_{W^{1,1}}\nonumber\\
&\le C\|\nabla u\|_{L^2}^2\|F_1\|_{H^1}+\|\nabla w\|_{L^2}^2\|F_2\|_{H^1}\nonumber\\
&\le \frac{1}{4}\|\sqrt{\rho}\dot{u}\|_{L^2}^2+\frac14\|\sqrt{\rho}\dot{w}\|_{L^2}^2+\delta\|{\rm curl}^2b\|_{L^2}^2
+C\big(\|\nabla u\|_{L^2}^4+\|\nabla b\|_{L^2}^4\big)\nonumber\\
&\quad+C\big(\|\nabla u\|_{L^2}^2+\|\nabla w\|_{L^2}^2+\|\nabla b\|_{L^2}^2\big),
\end{align*}
and
\begin{align}\label{3.80}
\|{\rm curl}^2b\|_{L^2}&\le C\Big(\|b_t\|_{L^2}+\|{\rm curl}^2b\|_{L^2}^\frac12\|\nabla b\|_{L^2}^\frac12\|\nabla u\|_{L^2}
+\|\nabla b\|_{L^2}\|\nabla u\|_{L^2}\Big)\nonumber\\
& \le \frac12\|{\rm curl}^2b\|_{L^2}^2+C\big(\|b_t\|_{L^2}+\|\nabla b\|_{L^2}\|\nabla u\|_{L^2}^2
+\|\nabla b\|_{L^2}\|\nabla u\|_{L^2}\big).
\end{align}
According to \eqref{3.78}, we get that
\begin{align}\label{3.81}
\sup_{0\le t\le \sigma(T)}\big[\sigma_i^2\big(\|\sqrt{\rho}\dot{u}\|_{L^2}^2
+\|\sqrt{\rho}\dot{w}\|_{L^2}^2+\|b_t\|_{L^2}^2\big)\big]\le C(\underline{\rho})E_0^\frac13,
\end{align}
and
\begin{align}\label{3.82}
\sup_{i\le t\le i+1}\big(\|\sqrt{\rho}\dot{u}\|_{L^2}^2
+\|\sqrt{\rho}\dot{w}\|_{L^2}^2+\|b_t\|_{L^2}^2\big)\le C(\underline{\rho})E_0^\frac13.
\end{align}
Hence, we deduce \eqref{3.75} from \eqref{3.6}, \eqref{3.80}, \eqref{3.81}, and \eqref{3.82}.

2. We integrate \eqref{3.15} over $[t_1, t_2]\subseteq[0, T]$ and take $\eta(t)=\sigma$ to obtain, from \eqref{3.6}
and \eqref{3.9} that
\begin{align}\label{3.83}
&\int_{t_1}^{t_2}\sigma\big(\|\sqrt{\rho}\dot{u}\|_{L^2}^2+\|\sqrt{\rho}\dot{w}\|_{L^2}^2+\|{\rm curl}^2b\|_{L^2}^2+\|b_t\|_{L^2}^2\big)dt\nonumber\\
&\le C(E_0+A_1(T))+CE_0(t_2-t_1)+C\int_{t_1}^{t_2}\sigma\big(\|\nabla u\|_{L^2}^6+\|\nabla b\|_{L^2}^6+\|\nabla w\|_{L^2}^6\big)dt\nonumber\\
&\quad+C\int_{t_1}^{t_2}\sigma\big(\|\nabla u\|_{L^2}^4+\|\nabla b\|_{L^2}^4+\|\nabla w\|_{L^2}^4\big)dt\nonumber\\
&\le CE_0^\frac13+CE_0(t_2-t_1)+C\int_{t_1}^{t_2}\big(\|\nabla u\|_{L^2}^2
+\|\nabla b\|_{L^2}^2+\|\nabla w\|_{L^2}^2\big)dt\nonumber\\
&\le CE_0^\frac13+C(\underline{\rho})E_0(t_2-t_1).
\end{align}
Similarly to \eqref{3.78}, integrating \eqref{z3.18} over $[t_1, t_2]$ and taking $\eta=\sigma^2$, we find that
\begin{align}\label{3.84}
&\int_{t_1}^{t_2}\sigma^2\big(\|\nabla\dot{u}\|_{L^2}^2+\|\nabla\dot{w}\|_{L^2}^2+\|\dot{w}\|_{L^2}^2+\|\nabla b_t\|_{L^2}^2\big)dt\nonumber\\
&\le C\sup_{t_1\le t\le t_2}\big(\sigma\|\sqrt{\rho}\dot{u}\|_{L^2}
+\sigma\|\sqrt{\rho}\dot{w}\|_{L^2}+\sigma\|b_t\|_{L^2}\big)
\int_{t_1}^{t_2}\sigma\big(\|\sqrt{\rho}\dot{u}\|_{L^2}^2
+\|\sqrt{\rho}\dot{w}\|_{L^2}^2+\|b_t\|_{L^2}^2\big)dt
+C(\underline{\rho})E_0^\frac13\nonumber\\
&\le C(\underline{\rho})E_0^\frac13+CE_0^\frac16\int_{t_1}^{t_2}
\sigma\big(\|\sqrt{\rho}\dot{u}\|_{L^2}^2+\|\sqrt{\rho}\dot{w}\|_{L^2}^2
+\|b_t\|_{L^2}^2\big)dt\nonumber\\
&\le C(\underline{\rho})E_0^\frac13+CE_0(t_2-t_1),
\end{align}
owing to \eqref{3.6}, \eqref{3.9}, \eqref{3.75}, and \eqref{3.83}.
%Then, \eqref{3.84} follows from \eqref{3.76}.
\end{proof}

We still need the following result before showing the upper bounds of the density.
\begin{lemma}\label{l37}
Let the assumption of Proposition \ref{p31} hold. Then there exists a positive constant $\varepsilon_5$ such that
\begin{align}\label{3.85}
&\sup_{0\le t\le \sigma(T)}\big[\sigma\big(\|\sqrt{\rho}\dot{u}\|_{L^2}^2+
\|\sqrt{\rho}\dot{w}\|_{L^2}^2+\|b_t\|_{L^2}^2+\|{\rm curl}^2b\|_{L^2}^2\big)\big]\nonumber\\
&\quad+\int_0^{\sigma(T)}\sigma
\big(\|\nabla\dot{u}\|_{L^2}^2+\|\nabla\dot{w}\|_{L^2}^2+\|\dot{w}\|_{L^2}^2+\|\nabla b_t\|_{L^2}^2\big)dt\le C(\underline{\rho}),
\end{align}
provided that $E_0\le \varepsilon_5$.
\end{lemma}
\begin{proof}[Proof]
Taking $\eta(t)=\sigma$ and integrating \eqref{z3.18} over $[0, \sigma(T)]$, we get from \eqref{3.6}, \eqref{3.9}, \eqref{3.65}, and Young's inequality that
\begin{align*}
&\sup_{0\le t\le \sigma(T)}\big[\sigma\big(\|\sqrt{\rho}\dot{u}\|_{L^2}^2+
\|\sqrt{\rho}\dot{w}\|_{L^2}^2+\|b_t\|_{L^2}^2\big)\big]
+\int_0^{\sigma(T)}\sigma
\big(\|\nabla\dot{u}\|_{L^2}^2+\|\nabla\dot{w}\|_{L^2}^2+\|\dot{w}\|_{L^2}^2+\|\nabla b_t\|_{L^2}^2\big)dt\nonumber\\
&\le C\int_0^{\sigma(T)}\sigma'\big(\|\sqrt{\rho}\dot{u}\|_{L^2}^2+\|\sqrt{\rho}\dot{w}\|_{L^2}^2
+\|b_t\|_{L^2}^2+\|{\rm curl}^2b\|_{L^2}^2\big)dt+C(\underline{\rho})E_0^\frac13\nonumber\\
&\quad+C\int_0^{\sigma(T)}\sigma\big(\|\nabla u\|_{L^2}+\|\nabla b\|_{L^2}+E_0^\frac12\big)
\big(\|\sqrt{\rho}\dot{u}\|_{L^2}^3
+\|\sqrt{\rho}\dot{w}\|_{L^2}^3+\|b_t\|_{L^2}^3\big)dt\nonumber\\
&\le CE_0^\frac16\sup_{0\le t\le \sigma(T)}\big[\sigma^\frac12\big(\|\sqrt{\rho}\dot{u}\|_{L^2}
+\|\sqrt{\rho}\dot{w}\|_{L^2}+\|b_t\|_{L^2}\big)\big]\nonumber\\
&\quad\times\int_0^{\sigma(T)}
\big(\|\sqrt{\rho}\dot{u}\|_{L^2}^2
+\|\sqrt{\rho}\dot{w}\|_{L^2}^2+\|b_t\|_{L^2}^2\big)dt+C(K,\underline{\rho})\nonumber\\
&\le C(K)+C(K)\sup_{0\le t\le \sigma(T)}\big[\sigma^\frac12\big(\|\sqrt{\rho}\dot{u}\|_{L^2}
+\|\sqrt{\rho}\dot{w}\|_{L^2}+\|b_t\|_{L^2}\big)\big]\nonumber\\
&\le \frac12\sup_{0\le t\le \sigma(T)}\big(\sigma\|\sqrt{\rho}\dot{u}\|_{L^2}^2
+\sigma\|\sqrt{\rho}\dot{w}\|_{L^2}^2+\sigma\|b_t\|_{L^2}^2\big)+C(\underline{\rho}).
\end{align*}
This together with \eqref{3.80} and \eqref{3.6} leads to \eqref{3.85}.
\end{proof}

With Lemmas \ref{l36} and \ref{l37} at hand, we derive the uniform upper bounds of the density, which is the key to obtain
all the higher-order estimates and thus to extend the classical solution globally.
\begin{lemma}\label{l38}
There exists a positive constant $\varepsilon$ as in Theorem \ref{thm1} such that if $(\rho, u, w, b, \Phi)$ is a smooth solution of \eqref{a1}--\eqref{a6} in $\Omega\times(0, T]$ satisfying \eqref{3.6}, then
\begin{align}\label{3.87}
\sup_{0\le t\le T}\|\rho(t)\|_{L^\infty}\le \frac74\hat{\rho}
\end{align}
provided that $E_0\le\varepsilon\triangleq\min\{\varepsilon_1, \varepsilon_2, \varepsilon_3, \varepsilon_4, \varepsilon_5\}$.
\end{lemma}
\begin{proof}[Proof]
1. We rewrite $\eqref{a1}_1$ as
\begin{align}
D_t\rho=g(\rho)+h'(t),
\end{align}
where
\begin{align*}
D_t\rho=\rho_t+u\cdot\nabla\rho, \quad g(\rho)=-\frac{a\rho(P-P_s)}{2\mu+\lambda},
\quad h(t)=-\frac{1}{2\mu+\lambda}\int_0^t\rho \Big(\frac12|b|^2+F_1\Big)d\tau.
\end{align*}
For $t\in [0, \sigma(T)]$, we deduce from H\"older's inequality, Lemma \ref{l28}, and \eqref{3.65} that, for
$0\le t_1<t_2\le \sigma(T)$,
\begin{align}
&|h(t_2)-h(t_1)|\nonumber\\
&\le C(\hat{\rho})\int_0^{\sigma(T)}\big(\|F_1\|_{L^\infty}+\|b\|_{L^\infty}^2\big)dt\nonumber\\
&\le C\int_0^{\sigma(T)}\|F_1\|_{L^2}^\frac14\|\nabla F_1\|_{L^6}^\frac34dt+C\int_0^{\sigma(T)}\|F_1\|_{L^2}dt
+C\int_0^{\sigma(T)}\|\nabla b\|_{L^2}\big(\|{\rm curl}^2b\|_{L^2}+\|\nabla b\|_{L^2}\big)dt\nonumber\\
&\le C\int_0^{\sigma(T)}\Big(
\|\nabla\dot{u}\|_{L^2}+\|\nabla\dot{w}\|_{L^2}
+\|\nabla b\|_{L^2}^\frac12\|{\rm curl}^2b\|_{L^2}^\frac32
+\|\nabla b\|_{L^2}^\frac32\|{\rm curl}^2b\|_{L^2}^\frac12+\|\nabla b\|_{L^2}^2\nonumber\\
&\quad+\|\nabla u\|_{L^2}^2+\|\nabla w\|_{L^2}+E_0^\frac16\Big)^\frac34\Big(\|\nabla u\|_{L^2}^\frac14
+\|P-P_s\|_{L^2}^\frac14+\||b|^2\|_{L^2}^\frac14\Big)dt\nonumber\\
&\quad+C\int_0^{\sigma(T)}(\|\nabla u\|_{L^2}+\|P-P_s\|_{L^2}+\|b\|_{L^4}^2)dt
+C\int_0^{\sigma(T)}\|\nabla b\|_{L^2}\|{\rm curl}^2b\|_{L^2}dt\nonumber\\
&\quad+C\int_0^{\sigma(T)}\|\nabla b\|_{L^2}^2dt\nonumber\\
&\le CE_0^\frac18+CE_0^\frac14\Big(\int_0^{\sigma(T)}\|{\rm curl}^2b\|_{L^2}^2dt\Big)^\frac12
+CE_0^\frac16\int_0^{\sigma(T)}\|\nabla u\|_{L^2}^\frac32dt\nonumber\\
&\quad+C\int_0^{\sigma(T)}\Big(\sigma^{-\frac12}(\sigma^\frac12\|\nabla u\|_{L^2})^\frac14+E_0^\frac16\sigma^{-\frac38}\Big)
\big(\sigma\|\nabla\dot{u}\|_{L^2}^2+\sigma\|\nabla\dot{w}\|_{L^2}^2\big)^\frac38dt\nonumber\\
&\quad+C\int_0^{\sigma(T)}\Big(\sigma^{-\frac18}(\sigma\|\nabla u\|_{L^2}^2)^\frac18+E_0^\frac16\Big)\|{\rm curl}^2b\|_{L^2}^\frac98dt\nonumber\\
&\quad+C\int_0^{\sigma(T)}\Big(\sigma^{-\frac18}(\sigma\|\nabla u\|_{L^2}^2)^\frac18
+E_0^\frac16\Big)\big(\|{\rm curl}^2b\|_{L^2}^2\big)^\frac{3}{16}dt\nonumber\\
&\quad+C\int_0^{\sigma(T)}\sigma^{-\frac18}\big(\sigma\|\nabla u\|_{L^2}^2\big)^\frac18dt+C\int_0^{\sigma(T)}\big(\|\nabla u\|_{L^2}^2\big)^\frac74dt
\nonumber\\
&\le CE_0^\frac18+CE_0^\frac14\Big(\int_0^{\sigma(T)}\|{\rm curl}^2b\|_{L^2}^2dt\Big)^\frac12
+CE_0\Big(\int_0^{\sigma(T)}\|\nabla u\|_{L^2}^2dt\Big)^\frac34\nonumber\\
&\quad+C\left[E_0^\frac{1}{24}\Big(\int_0^{\sigma(T)}
\sigma^{-\frac45}dt\Big)^\frac58
+E_0^\frac16\Big(\int_0^{\sigma(T)}\sigma^{-\frac35}dt\Big)^\frac58\right]
\left[\int_0^{\sigma(T)}\big(\sigma\|\nabla\dot{u}\|_{L^2}^2
+\sigma\|\nabla\dot{w}\|_{L^2}^2\big)dt\right]^\frac38\nonumber\\
&\quad+C\left[E_0^\frac16+E_0^\frac{1}{24}
\Big(\int_0^{\sigma(T)}\sigma^{-\frac{2}{7}}dt\Big)^\frac{7}{16}\right]
\Big(\int_0^{\sigma(T)}\|{\rm curl}^2b\|_{L^2}^2\Big)^\frac{9}{16}\nonumber\\
&\quad+C\left[E_0^\frac16+E_0^\frac{1}{24}
\Big(\int_0^{\sigma(T)}\sigma^{-\frac{2}{13}}dt\Big)^\frac{13}{16}\right]
\Big(\int_0^{\sigma(T)}\|{\rm curl}^2b\|_{L^2}^2\Big)^\frac{3}{16}\nonumber\\
&\le C(\underline{\rho})E_0^\frac{1}{24},
\end{align}
provided that $E_0\le \varepsilon$. Thus, for $t\in [0, \sigma(T)]$, we can choose $N_0$ and $N_1$ in Lemma \ref{l210} as
\begin{align}
N_1=0, \quad N_0=CE_0^\frac{1}{24},
\end{align}
and
$\xi_0=0$ in Lemma \ref{l210}. Then,
\begin{align}
g(\xi)=-\frac{a\xi}{2\mu+\lambda}(\xi^\gamma-\rho_s^\gamma)\le -N_1=0 \quad {\rm for~all}~\xi\ge \xi_0=\hat{\rho},
\end{align}
we thus deduce from Lemma \ref{l210} that
\begin{align}\label{3.92}
\sup_{0\le t\le \sigma(T)}\|\rho\|_{L^\infty}\le \max\{\hat{\rho}, \xi_0\}+N_0\le \hat{\rho}+C(\underline{\rho})E_0^\frac{1}{24}\le \frac{3\hat{\rho}}{2},
\end{align}
provided that $E_0\le \varepsilon$.

2. For $t\in [\sigma(T), T]$, we derive from Lemma \ref{l28}, \eqref{3.6}, \eqref{3.9}, \eqref{3.16}, and \eqref{3.76} that, for $\sigma(T)\le t_1<t_2\le T$,
\begin{align}
|h(t_2)-h(t_1)|&\le C(\hat{\rho})\int_{t_1}^{t_2}\big(\|F_1\|_{L^\infty}+\|b\|_{L^\infty}^2\big)dt\nonumber\\
&\le C\int_{t_1}^{t_2}\|F_1\|_{L^\infty}^\frac83dt
+\frac{a}{4\mu+2\lambda}(t_2-t_1)+CE_0^\frac12\nonumber\\
&\le C\int_{t_1}^{t_2}\|F_1\|_{L^2}^\frac23\|\nabla F_1\|_{L^6}^2dt+\int_{t_1}^{t_2}\|F_1\|_{L^2}^\frac83dt
+\frac{a}{4\mu+2\lambda}(t_2-t_1)+CE_0^\frac12\nonumber\\
&\le C\int_{t_1}^{t_2}\Big(
\|\nabla\dot{u}\|_{L^2}^2+\|\nabla\dot{w}\|_{L^2}^2
+\|\nabla b\|_{L^2}\|{\rm curl}^2b\|_{L^2}^3
+\|\nabla b\|_{L^2}^3\|{\rm curl}^2b\|_{L^2}\nonumber\\
&\quad+\|\nabla b\|_{L^2}^4+\|\nabla u\|_{L^2}^4+\|\nabla w\|_{L^2}^2
+E_0^\frac13\Big)\Big(\|\nabla u\|_{L^2}^\frac23
+\|P-P_s\|_{L^2}^\frac23+\||b|^2\|_{L^2}^\frac23\Big)dt\nonumber\\
&\quad+C\int_{t_1}^{t_2}\Big(\|\nabla u\|_{L^2}^\frac83+\|P-P_s\|_{L^2}^\frac83+\|b\|_{L^4}^\frac{16}{3}\Big)dt
+\frac{a}{4\mu+2\lambda}(t_2-t_1)+CE_0^\frac12\nonumber\\
&\le \Big(\frac{a}{4\mu+2\lambda}+CE_0^\frac49\Big)(t_2-t_1)+CE_0^\frac49
+CE_0^\frac19\int_{t_1}^{t_2}\big(\|\nabla\dot{u}\|_{L^2}^2+\|\nabla\dot{w}\|_{L^2}^2\big)dt\nonumber\\
&\quad+CE_0^\frac49\Big(\int_{t_1}^{t_2}\|\nabla b\|_{L^2}^2\Big)^\frac12
\Big(\int_{t_1}^{t_2}\|{\rm curl}^2b\|_{L^2}^2dt\Big)^\frac12+CE_0^\frac49\int_{t_1}^{t_2}\big(\|\nabla b\|_{L^2}^2+\|\nabla u\|_{L^2}^2\big)dt\nonumber\\
&\le \Big(\frac{a}{4\mu+2\lambda}+C(\underline{\rho})E_0^\frac49+CE_0^\frac{10}{9}\Big)(t_2-t_1)
+CE_0^\frac49\nonumber\\
&\le \frac{a}{2\mu+\lambda}(t_2-t_1)+CE_0^\frac49,
\end{align}
provided that $E_0\le \varepsilon$. Thus, for $t\in [\sigma(T), T]$, we can choose $N_0$, $N_1$, and $\xi_0$ in Lemma \ref{l210} as follows:
\begin{align*}
N_0=CE_0^\frac49, \quad N_1=\frac{a}{2\mu+\lambda}, \quad \xi_0=\frac{3\hat{\rho}}{2}.
\end{align*}
Since for all $\xi\ge \xi_0=\frac{3\hat{\rho}}{2}>\rho_s+1$,
\begin{align}
g(\xi)=-\frac{a\xi}{2}(\xi^\gamma-\rho_s^\gamma)\le -N_1 \quad {\rm for~all}~\xi\ge \xi_0=\frac{3\hat{\rho}}{2}.
\end{align}
Thus, due to Lemma \ref{l210}, we arrive at
\begin{align}\label{3.95}
\sup_{\sigma(T)\le t\le T}\|\rho\|_{L^\infty}\le \frac{3\hat{\rho}}{2}+N_0\le \frac{3\hat{\rho}}{2}+CE_0^\frac49\le \frac{7\hat{\rho}}{4},
\end{align}
provided that $E_0\le \varepsilon$. The combination of \eqref{3.92} and \eqref{3.95}, we obtain \eqref{3.87}.
\end{proof}

Now, we are ready to prove Proposition \ref{p31}.
\begin{proof}[Proof of Proposition \ref{p31}]
Proposition \ref{p31} follows from Lemma \ref{l34}, Lemma \ref{l35}, and Lemma \ref{l38}.
\end{proof}

\subsection{Higher-order estimates}\label{sec3.2}

In this subsection, we establish the time-dependent higher-order estimates of solutions, which are necessary for the global existence of
classical solutions. In what follows, we denote by $C$ or $C_i\ (i=1, 2,\ldots)$ the various positive constants, which may depend on the initial data, $\mu$, $\lambda$, $u_r$, $c_0$, $c_a$, $c_d$,
$\gamma$, $a$, $\hat{\rho}$, $\Omega$, $M_1$, $M_2$, $M_3$, $\tilde{\rho}$, $\underline{\rho}$, and $T$ as well.
\begin{lemma}\label{zl39}
Under the conditions of Theorem \ref{thm1}, it holds that
\begin{align}\label{z3.69}
&\sup_{0\le t\le \sigma(T)}(\|\sqrt{\rho}\dot{u}\|_{L^2}^2+
\|\sqrt{\rho}\dot{w}\|_{L^2}^2+\|b_t\|_{L^2}^2+\|\curl^2b\|_{L^2}^2)\nonumber\\
&\quad+\int_0^{\sigma(T)}(\|\nabla\dot{u}\|_{L^2}^2+\|\nabla\dot{w}\|_{L^2}^2
+\|\dot{w}\|_{L^2}^2+\|\nabla b_t\|_{L^2}^2)dt\le C.
\end{align}
\end{lemma}
\begin{proof}[Proof]
Taking $\eta(t)=1$ in \eqref{z3.18} and integrating the resulting equation over $(0, \sigma(T)]$, we
deduce from \eqref{3.65} and Young's inequality that
\begin{align*}
&\int_0^{\sigma(T)}\big(\|\nabla\dot{u}\|_{L^2}^2+\|\nabla\dot{w}\|_{L^2}^2+\|\dot{w}\|_{L^2}^2+\|\nabla b_t\|_{L^2}^2\big)dt\nonumber\\
&\le C\sup_{0\le t\le \sigma(T)}\big(\|\sqrt{\rho}\dot{u}\|_{L^2}
+\|\sqrt{\rho}\dot{w}\|_{L^2}+\|b_t\|_{L^2}\big)
\int_{0}^{\sigma(T)}\big(\|\sqrt{\rho}\dot{u}\|_{L^2}^2
+\|\sqrt{\rho}\dot{w}\|_{L^2}^2+\|b_t\|_{L^2}^2\big)dt\nonumber\\
&\quad+\|\sqrt{\rho_0}g_1\|_{L^2}^2+\|\sqrt{\rho_0}g_2\|_{L^2}^2+CE_0^\frac13\nonumber\\
&\le C+C\sup_{0\le t\le \sigma(T)}\big(\|\sqrt{\rho}\dot{u}\|_{L^2}
+\|\sqrt{\rho}\dot{w}\|_{L^2}+\|b_t\|_{L^2}\big)\nonumber\\
&\le \frac12\sup_{0\le t\le \sigma(T)}\big(\|\sqrt{\rho}\dot{u}\|_{L^2}^2
+\|\sqrt{\rho}\dot{w}\|_{L^2}^2+\|b_t\|_{L^2}^2\big)+C.
\end{align*}
This along with \eqref{3.80} and \eqref{3.65} implies \eqref{z3.69}.
\end{proof}

\begin{lemma}\label{l30}
Under the conditions of Theorem \ref{thm1}, it holds that
\begin{align}
&\sup_{0\le t\le T}(\|\sqrt{\rho}\dot{u}\|_{L^2}^2+
\|\sqrt{\rho}\dot{w}\|_{L^2}^2+\|b_t\|_{L^2}^2+\|\curl^2b\|_{L^2}^2)\nonumber\\
&\quad+\int_0^T(\|\nabla\dot{u}\|_{L^2}^2+\|\nabla\dot{w}\|_{L^2}^2
+\|\dot{w}\|_{L^2}^2+\|\nabla b_t\|_{L^2}^2)dt\le C,\label{z3.71}\\
&\sup_{0\le t\le T}(\|\nabla\rho\|_{L^6}+\|u\|_{H^2}+\|w\|_{H^2})\nonumber\\
&\quad+\int_0^T(\|\nabla u\|_{L^\infty}+\|\nabla w\|_{L^\infty}+\|\nabla^2u\|_{L^6}^2
+\|\nabla^2w\|_{L^6}^2)dt\le C.\label{z3.72}
\end{align}
\end{lemma}
\begin{proof}[Proof]
The estimate \eqref{z3.71} follows directly from Lemmas \ref{l36} and \ref{zl39}. Furthermore, based on the Beale-Kato-Majda type
inequality (see Lemma \ref{l211}), we can obtain the estimates on the gradient of density and velocity. Since the arguments are similar to \cite{CJ21},
we omit the details for simplicity.
\end{proof}

\begin{lemma}\label{l31}
Under the conditions of Theorem \ref{thm1}, it holds that
\begin{align}
&\sup_{0\le t\le T}\big(\|\sqrt{\rho}u_t\|_{L^2}^2+\|\sqrt{\rho}w_t\|_{L^2}^2\big)
+\int_0^T(\|\nabla u_t\|_{L^2}^2
+\|\nabla w_t\|_{L^2}^2)dt\le C,\label{z3.73}\\
&\sup_{0\le t\le T}\big(\|\rho\|_{H^2}+\|P\|_{H^2}
+\|\rho_t\|_{H^1}+\|P_t\|_{H^1}\big)+\int_0^T(\|\rho_{tt}\|_{L^2}^2+\|P_{tt}\|_{L^2}^2)dt\le C,\label{z3.74}\\
&\sup_{0\le t\le T}\big[\sigma\big(\|\nabla u_t\|_{L^2}^2+\|\nabla w_t\|_{L^2}^2+\|\nabla b_t\|_{L^2}^2\big)\big]
+\int_0^T\sigma\big(\|\sqrt{\rho}u_{tt}\|_{L^2}^2
+\|\sqrt{\rho}w_{tt}\|_{L^2}^2+\|b_{tt}\|_{L^2}^2\big)dt\le C,\label{z3.75}\\[3pt]
&\sup_{0\le t\le T}\big(\|\nabla(\Phi-\Phi_s)\|_{H^3}+\|\nabla\Phi_t\|_{H^2}+\|\nabla \Phi_{tt}\|_{L^2}\big)\le C.\label{z3.76}
\end{align}
\end{lemma}
\begin{proof}[Proof]
Based on Lemmas \ref{l28} and \ref{l30}, \eqref{z3.73} and \eqref{z3.74} can be obtained by the same method as that of
in \cite{CJ21}.
From $\eqref{a1}_1$, we have
\begin{align*}
\Delta\Phi_t=\rho_t,\quad \Delta\Phi_{tt}=-\divv(\rho_tu+\rho u_t).
\end{align*}
This together with \eqref{z3.74}, \eqref{2.42}, and \eqref{2.43} leads to \eqref{z3.76}.

To prove \eqref{z3.75}, we introduce the function
\begin{align*}
\mathcal A(t)&\triangleq(2\mu+\lambda)\int(\divv u_t)^2dx+\mu\int|\curl u_t|^2dx
+4\mu_r\int\Big(w_t-\frac{\curl u_t}{2}\Big)^2dt\nonumber\\
&\quad+(2c_d+c_0)\int(\divv w_t)^2dt+(c_0+c_a)\int|\curl w_t|^2dx+\int|\curl b_t|^2dx.
\end{align*}
Since $u_t\cdot n=0$, $w_t\cdot n=0$, and $b_t\cdot n=0$ on $\partial\Omega$, we infer from Lemma \ref{l24} that
\begin{align}\label{z3.78}
\|\nabla u_t\|_{L^2}^2+\|\nabla w_t\|_{L^2}^2+\|\nabla b_t\|_{L^2}^2\le C(\Omega)\mathcal A(t).
\end{align}
Differentiating $\eqref{a1}_{2,3,4}$ with respect to $t$, and multiplying the resultant by $u_{tt}$, $w_{tt}$, and $b_{tt}$, respectively,
and integrating them by parts over $\Omega$, we obtain after summing up that
\begin{align}\label{z3.79}
&\frac{d}{dt}\mathcal A(t)+2\int\big(\rho|u_{tt}|^2+\rho|w_{tt}|^2+|b_{tt}|^2\big)dx\nonumber\\
&=\frac{d}{dt}\int\big(-\rho_t|u_t|^2-2\rho_tu\cdot\nabla u\cdot u_t+2P_t\divv u_t-\rho_t|w_t|^2-2\rho_tu\cdot\nabla w\cdot w_t\big)dx\nonumber\\
&\quad-\frac{d}{dt}\int\big(2(b\otimes b)_t:\nabla u_t-|b_t|^2\divv u_t\big)dx+\int\rho_{tt}|u_t|^2dx+2\int(\rho_tu\cdot\nabla u)_t\cdot u_tdx\nonumber\\
&\quad
-2\int\rho u_t\cdot\nabla u\cdot u_{tt}dx-2\int\rho u\cdot\nabla u_t\cdot u_{tt}dx-2\int P_{tt}\divv u_tdx+\int\rho_{tt}|w_t|^2dx\nonumber\\
&\quad+2\int(\rho_tu\cdot\nabla w)_t\cdot w_tdx
-2\int\rho u_t\cdot\nabla w\cdot w_{tt}dx-2\int\rho u\cdot\nabla w_t\cdot w_{tt}dx\nonumber\\
&\quad
+\int(2(b\otimes b)_{tt}:\nabla u_t-|b|_{tt}^2\divv u_t)dx
+2\int(b\cdot\nabla u-u\cdot\nabla b-b\divv u)_t\cdot b_{tt}dx\nonumber\\
&\quad+2\int(\rho_t\nabla(\Phi-\Phi_s)+\rho\nabla(\Phi-\Phi_s)_t+\rho_t\nabla \Phi_s)\cdot u_{tt}dx\nonumber\\
&\le \frac{d}{dt}\mathcal A_0+\frac12\|b_{tt}\|_{L^2}^2
+\frac12\|\sqrt{\rho} u_{tt}\|_{L^2}^2+\frac12\|\sqrt{\rho} w_{tt}\|_{L^2}^2
+C(\|\nabla u_t\|_{L^2}^4+\|\nabla w_t\|_{L^2}^4
+\|\nabla b_t\|_{L^2}^4)\nonumber\\
&\quad+C(1+\|\rho_{tt}\|_{L^2}^2+\|P_{tt}\|_{L^2}^2+\|\nabla u_t\|_{L^2}^2+\|\nabla w_t\|_{L^2}^2+\|\nabla b_t\|_{L^2}^2),
\end{align}
where
\begin{align*}
\mathcal A_0&\le \Big|\int\divv(\rho u)(|u_t|^2+|w_t|^2)dx\Big|+C\|\rho_t\|_{L^3}\|u\|_{L^\infty}\|\nabla u\|_{L^2}\|u_t\|_{L^6}
+C\|P_t\|_{L^2}\|\nabla u_t\|_{L^2}\nonumber\\
&\quad+C\|\rho_t\|_{L^3}\|u\|_{L^\infty}\|\nabla w\|_{L^2}\|w_t\|_{L^6}
+C\|b\|_{L^\infty}\|b_t\|_{L^2}\|\nabla u_t\|_{L^2}\nonumber\\
&\le C\int|u||\rho u_t||\nabla u_t|dx+C\int|u||\rho w_t||\nabla w_t|dx
+C\|\nabla u_t\|_{L^2}\nonumber\\
&\le C\|u\|_{L^6}\|\sqrt{\rho}u_t\|_{L^2}^\frac12\|u_t\|_{L^6}^\frac12\|\nabla u_t\|_{L^2}+C\|\nabla u_t\|_{L^2}
+C\|u\|_{L^6}\|\sqrt{\rho}w_t\|_{L^2}^\frac12\|w_t\|_{L^6}^\frac12\|\nabla w_t\|_{L^2}\nonumber\\
&\le \frac12\mathcal A(t)+C,
\end{align*}
due to $\eqref{a1}_1$, \eqref{z3.73}, \eqref{z3.74}, \eqref{z3.72}, \eqref{z3.78}, and  H\"older's, Sobolev's, Poincar\'e's inequalities.
We deduce after multiplying \eqref{z3.79} by $\sigma$, integrating it over $[0, T]$, and using \eqref{z3.71}, \eqref{z3.73}, \eqref{z3.74},
and Gronwall's inequality that
\begin{align*}
\sup_{0\le t\le T}(\sigma\mathcal A(t))+\int_0^T\sigma\big(\|\sqrt{\rho}u_{tt}\|_{L^2}^2+\|\sqrt{\rho}w_{tt}\|_{L^2}^2
+\|b_{tt}\|_{L^2}^2\big)dt\le C,
\end{align*}
which combined with \eqref{z3.78} leads to \eqref{z3.75}.
\end{proof}

\begin{lemma}\label{l312}
Under the conditions of Theorem \ref{thm1}, it holds that, for any $q\in (3, 6)$,
\begin{align}
&\sup_{0\le t\le T}\big[\sigma\big(\|\nabla u\|_{H^2}^2+\|\nabla b\|_{H^2}^2+\|\nabla w\|_{H^2}^2\big)\big]\nonumber\\
&\quad+\int_0^T\big(\|\nabla u\|_{H^2}^2+\|\nabla b\|_{H^2}^2+\|\nabla w\|_{H^2}^2+\sigma\|\nabla u_t\|_{H^1}^2
+\sigma\|\nabla w_t\|_{H^1}^2+\|\nabla^2u\|_{W^{1, q}}^{p_0}
+\|\nabla^2w\|_{W^{1, q}}^{p_0}\big)dt\nonumber\\
&\le C,\\
&\sup_{0\le t\le T}\big(\|\rho-\rho_s\|_{W^{2, q}}+\|P-P_s\|_{W^{2, q}}+\|\nabla(\Phi-\Phi_s)\|_{W^{3, q}}\big)\le C,\label{z3.87}
\end{align}
where $p_0=\frac{9q-6}{10q-12}\in (1, \frac76)$.
\end{lemma}
\begin{proof}[Proof]
1. By Lemma \ref{l30}, Poincar\'e's inequality, and Sobolev's inequality, one can check that
\begin{align}
\|\nabla(\rho\dot{u})\|_{L^2}&\le \||\nabla\rho||u_t|\|_{L^2}
+\|\rho|\nabla u_t|\|_{L^2}+\||\nabla\rho||u||\nabla u|\|_{L^2}
+\|\rho|\nabla u|^2\|_{L^2}+\|\rho|u||\nabla^2u|\|_{L^2}\nonumber\\
&\le C\|\nabla u_t\|_{L^2}+C,\label{z2.87}\\
\|\nabla(\rho\dot{w})\|_{L^2}&\le \||\nabla\rho||w_t|\|_{L^2}+\|\rho|\nabla w_t|\|_{L^2}
+\||\nabla\rho||u||\nabla w|\|_{L^2}+\|\rho|\nabla u||\nabla w|\|_{L^2}+\|\rho|u||\nabla^2w|\|_{L^2}\nonumber\\
&\le C\|\nabla w_t\|_{L^2}+C.\label{z2.88}
\end{align}
These together with \eqref{z2.42}, \eqref{z2.43}, \eqref{z2.87}, \eqref{z2.88}, Lemma \ref{l28}, and Lemma \ref{l30} yields that
\begin{align}\label{z2.89}
\|\nabla^2u\|_{H^1}+\|\nabla^2w\|_{H^1}
&\le C\big(\|\rho\dot{u}\|_{H^1}+\|\rho\dot{w}\|_{H^1}+\|b\cdot\nabla b\|_{H^1}+\|P-P_s\|_{H^1}+\||b|^2\|_{H^1}\big)\nonumber\\[3pt]
&\quad+C\big(\|\nabla u\|_{L^2}+\|\nabla w\|_{L^2}
+\|\rho\nabla(\Phi-\Phi_s)\|_{H^1}+\|P-P_s\|_{H^1}\big)\nonumber\\
&\le C\|\nabla u_t\|_{L^2}+\|\nabla w_t\|_{L^2}+C.
\end{align}
It follows from \eqref{z2.89}, \eqref{z3.73}, \eqref{z3.75}, and \eqref{z3.76} that
\begin{align}
\sup_{0\le t\le T}\big[\sigma\big(\|\nabla u\|_{H^2}^2+\|\nabla w\|_{H^2}^2\big)\big]
+\int_0^T\big(\|\nabla u\|_{H^2}^2+\|\nabla w\|_{H^2}^2\big)dt\le C.
\end{align}
By $\eqref{a1}_4$ and \eqref{z2.41}, we have
\begin{align}\label{z3.92}
\|\nabla^2b\|_{H^1}&\le C\big(\|b_t\|_{H^1}+\|u\cdot\nabla b\|_{H^1}+\|b\cdot\nabla u\|_{H^1}
+\|b\divv u\|_{H^1}+\|\nabla b\|_{L^2}\big)\nonumber\\[3pt]
&\le C\|\nabla b_t\|_{L^2}+C,
\end{align}
which together with \eqref{3.15}, \eqref{z3.72}, and \eqref{z2.87} implies that
\begin{align}
\sup_{0\le t\le T}\big(\sigma\|\nabla b\|_{H^2}\big)+\int_0^T\|\nabla b\|_{H^2}^2dt\le C.
\end{align}

2. We deduce from Lemmas \ref{l30} and \ref{l31} that
\begin{align}\label{z3.90}
&\|\nabla^2 u_t\|_{L^2}+\|\nabla^2 w_t\|_{L^2}\nonumber\\
&\le C(\|(\rho\dot{u})_t\|_{L^2}+\|(\rho\dot{w})_t\|_{L^2}+\|\nabla P_t\|_{L^2}+\|(\rho\Phi)_t\|_{L^2})
+\|((\nabla\times b)\times b)_t\|_{L^2}\nonumber\\[3pt]
&\quad+C(\|\curl u_t\|_{L^2}+\|w_t\|_{L^2}+\|\curl w_t\|_{L^2}+\|u_t\|_{L^2})\nonumber\\
&\le C\|\sqrt{\rho}u_{tt}\|_{L^2}+C\|\sqrt{\rho}w_{tt}\|_{L^2}+C\|\nabla u_t\|_{L^2}+C\|\nabla w_t\|_{L^2}+C\|\nabla b_t\|_{L^2}+C,
\end{align}
where we have used the $L^p$-estimate for the following elliptic systems
\begin{align*}
\begin{cases}
(\mu+\mu_r)\Delta u_t+(\mu+\lambda-\mu_r)\nabla{\rm div}\,u_t
=(\rho\dot{u})_t+\nabla P_t+(\rho\nabla\Phi)_t+((\nabla\times b)\times b)_t+2\mu_r\curl w_t &{\rm in}~\Omega,\\[3pt]
u_t\cdot n=0, \quad \curl u\times n=0 &{\rm on}~\partial\Omega,
\end{cases}
\end{align*}
and
\begin{align*}
\begin{cases}
(c_d+c_a)\Delta w_t+(c_d+c_0-c_a)\nabla{\rm div}\,w_t+4\mu_r
=(\rho\dot{w})_t+2\mu_r\curl u_t &{\rm in}~\Omega,\\[3pt]
w_t\cdot n=0, \quad \curl w_t\times n=0 &{\rm on}~\partial\Omega.
\end{cases}
\end{align*}
Combining \eqref{z3.75} and \eqref{z3.90}, one has that
\begin{align}\label{z3.94}
\int_0^T\sigma(\|\nabla u_t\|_{H^1}^2+\|\nabla w_t\|_{H^1}^2)dt\le C.
\end{align}
By Sobolev's inequality, \eqref{3.2},  \eqref{z3.72},
 and \eqref{z3.75}, we get that, for any $q\in (3, 6)$,
\begin{align}\label{z3.95}
\|\nabla(\rho\dot{u})\|_{L^q}&\le C\|\nabla\rho\|_{L^q}(\|\nabla\dot{u}\|_{L^q}+\|\nabla\dot{u}\|_{L^2}
+\|\nabla u\|_{L^2}^2)+C\|\nabla\dot{u}\|_{L^q}\nonumber\\
&\le C(\|\nabla u_t\|_{L^2}+1)+C\|\nabla u_t\|_{L^2}^\frac{6-q}{2q}\|\nabla u_t\|_{L^6}^\frac{3(q-2)}{2q}\nonumber\\
&\quad+C\|u\|_{L^\infty}\|\nabla^2u\|_{L^q}+C\|\nabla u\|_{L^\infty}\|\nabla u\|_{L^q}\nonumber\\
&\le C\sigma^{-\frac12}+C\|\nabla u\|_{H^2}+C\sigma^{-\frac12}(\sigma\|\nabla u_t\|_{H^1}^2)^\frac{3(q-2)}{4q}+C.
\end{align}
Similarly, by \eqref{z2.9}, \eqref{z3.72},
 and \eqref{z3.75}, one obtains that
\begin{align}\label{z3.96}
\|\nabla(\rho\dot{w})\|_{L^q}&\le C\|\nabla\rho\|_{L^q}(\|\nabla\dot{w}\|_{L^q}+\|\nabla\dot{w}\|_{L^2}
+\|\nabla u\|_{L^2}^2+\|\nabla w\|_{L^2}^2)+C\|\nabla\dot{w}\|_{L^q}\nonumber\\
&\le  C\sigma^{-\frac12}+C\|\nabla u\|_{H^2}
+C\|\nabla w\|_{H^2}+C\sigma^{-\frac12}(\sigma\|\nabla w_t\|_{H^1}^2)^\frac{3(q-2)}{4q}+C.
\end{align}
Integrating this inequality over $[0, T]$, by \eqref{z3.90},
\eqref{z3.94}, \eqref{z3.95}, and \eqref{z3.96}, we have
\begin{align}\label{z3.98}
\int_0^T\big(\|\nabla(\rho\dot{u})\|_{L^q}^{p_0}
+\|\nabla(\rho\dot{w})\|_{L^q}^{p_0}\big)dt\le C.
\end{align}

3. It follows from \eqref{z3.13} that
\begin{align}\label{z3.99}
\frac{d}{dt}\|\nabla^2P\|_{L^q}&\le C\|\nabla u\|_{L^\infty}\|\nabla^2P\|_{L^q}+C\|\nabla^2u\|_{W^{1, q}}\nonumber\\
&\le C\big(1+\|\nabla u_t\|_{L^2}+\|\nabla w_t\|_{L^2}
+\|\nabla(\rho\dot{u})\|_{L^q}+\|\nabla(\rho\dot{w})\|_{L^q}\big)\nonumber\\
&\quad+C(1+\|\nabla u\|_{L^\infty})\|\nabla^2P\|_{L^q},
\end{align}
where in the last inequality we have used
\begin{align}\label{z3.100}
&\|\nabla^2u\|_{W^{1, q}}+\|\nabla^2w\|_{W^{1, q}} \notag \\
& \le C(1+\|\nabla u_t\|_{L^2}+\|\nabla w_t\|_{L^2}+\|\nabla b_t\|_{L^2}+\|\nabla(\rho\dot{u})\|_{L^q}
+\|\nabla(\rho\dot{w})\|_{L^q}+\|\nabla^2 P\|_{L^q}),
\end{align}
due to \eqref{z2.42}, \eqref{z2.43}, \eqref{z3.71}, and \eqref{z3.74}.
Hence, applying Gronwall's inequality in \eqref{z3.99}, we deduce from \eqref{z3.72}, \eqref{z3.73},
and \eqref{z3.98} that
\begin{align}
\sup_{0\le t\le T}\|\nabla^2P\|_{L^q}\le C,
\end{align}
which along with \eqref{z3.74}, \eqref{z3.98}, and \eqref{z3.100} implies that
\begin{align}\label{z3.102}
\sup_{0\le t\le T}\|P-P_s\|_{W^{2, q}}+\int_0^T(\|\nabla^2u\|_{W^{1, q}}^{p_0}
+\|\nabla^2w\|_{W^{1, q}}^{p_0})dt\le C.
\end{align}
Similarly, we have
\begin{align*}
\sup_{0\le t\le T}\|\rho-\rho_s\|_{W^{2, q}}\le C,
\end{align*}
which together with \eqref{z3.102} and \eqref{2.42} gives \eqref{z3.87}.
\end{proof}

\begin{lemma}\label{l313}
Under the conditions of Theorem \ref{thm1}, it holds that, for any $q\in (3, 6)$,
\begin{align}\label{z3.104}
&\sup_{0\le t\le T}
\big[\sigma\big(\|\sqrt{\rho}u_{tt}\|_{L^2}+\|\sqrt{\rho}w_{tt}\|_{L^2}
+\|\nabla u_t\|_{H^1}+\|\nabla w_t\|_{H^1}\big)\big]\nonumber\\
&\quad+\sup_{0\le t\le T}\big[\sigma\big(\|b_{tt}\|_{L^2}+\|\nabla b_t\|_{H^1}+\|\nabla^2b\|_{H^2}+\|\nabla u\|_{W^{2, q}}+\|\nabla w\|_{W^{2, q}}\big)\big]\nonumber\\
&\quad+\int_0^T\sigma^2\big(\|\nabla u_{tt}\|_{L^2}^2+\|\nabla w_{tt}\|_{L^2}^2+\|\nabla b_{tt}\|_{L^2}^2\big)dt\le C.
\end{align}
\end{lemma}
\begin{proof}[Proof]
1. Differentiating $\eqref{a1}_{2,3,4}$ with respect to $t$ twice, and multiplying
the resultant by $u_{tt}$, $w_{tt}$, $b_{tt}$, respectively, then integrating the resulting equations over $\Omega$,
 we deduce after summing up that
\begin{align}\label{3.105}
&\frac12\frac{d}{dt}\int(\rho|u_{tt}|^2+\rho|w_{tt}|^2+|b_{tt}|^2)dx\nonumber\\
&\quad+(2\mu+\lambda)\int(\divv u_{tt})^2dx+\mu\int|\curl u_{tt}|^2dx+(2c_d+c_0)\int(\divv w_{tt})^2dx\nonumber\\
&\quad+(c_d+c_a)\int|\curl w_{tt}|^2dx+4\mu_r\int\Big(w_{tt}-\frac{\curl w_{tt}}{2}\Big)^2dx
+\int|\curl b_{tt}|^2dx\nonumber\\
&=-4\int\rho u\cdot\nabla u_{tt}\cdot u_{tt}dx-\int(\rho u)_t\cdot[\nabla (u_t\cdot u_{tt})+2\nabla u_t\cdot u_{tt}]dx\nonumber\\
&\quad-\int(\rho_{tt}u+2\rho_tu_t)\cdot\nabla u\cdot u_{tt}dx
-\int\rho u_{tt}\cdot\nabla u\cdot u_{tt}dx+\int P_{tt}\divv u_{tt}dx\nonumber\\
&\quad-4\int\rho u\cdot \nabla w_{tt}\cdot w_{tt}dx-\int (\rho u)_t\cdot[\nabla(w_t\cdot w_{tt})+2\nabla w_t\cdot w_{tt}]dx\nonumber\\
&\quad-\int(\rho_{tt}u+2\rho_tu_t)\cdot\nabla w\cdot w_{tt}dx-\int\rho u_{tt}\cdot\nabla w\cdot w_{tt}dx\nonumber\\
&\quad-\int(b\cdot\nabla b-\nabla|b|^2/2)_{tt}u_{tt}dx+\int(b\cdot\nabla u-u\cdot\nabla b-b\divv u)_{tt}b_{tt}dx\nonumber\\
&\quad-\int(\rho_{tt}\nabla\Phi+2\rho_t\nabla\Phi_t+\rho\Phi_{tt})u_{tt}dx
\triangleq\sum_{i=1}^{12}Q_i.
\end{align}
By virtue of Gagliardo-Nirenberg, Sobolev, and H\"older's inequalities, we obtain from \eqref{z3.72}, \eqref{z3.71}, and \eqref{z3.74} that
\begin{align*}
Q_1&\le C(\hat{\rho})\|u\|_{L^\infty}\|\sqrt{\rho}u_{tt}\|_{L^2}\|\nabla u_{tt}\|_{L^2}
\le \delta\|\nabla u_{tt}\|_{L^2}^2+\|\sqrt{\rho}u_{tt}\|_{L^2}^2,\\
Q_2&\le C\|(\rho u)_t\|_{L^3}(\|u_{tt}\|_{L^6}\|\nabla u_t\|_{L^2}+\|u_t\|_{L^6}\|\nabla u_{tt}\|_{L^2})\nonumber\\
&\le C(\|\sqrt{\rho}u_t\|_{L^2}^\frac12\|u_t\|_{L^6}^\frac12
+\|\rho_t\|_{L^6}\|u\|_{L^6})\|\nabla u_{tt}\|_{L^2}\|\nabla u_t\|_{L^2}\nonumber\\
&\le \delta\|\nabla u_{tt}\|_{L^2}^2+C\|\nabla u_t\|_{L^2}^3+C,\\
Q_3&\le C(\|\rho_{tt}\|_{L^2}\|u\|_{L^\infty}\|\nabla u\|_{L^3}+\|\rho_t\|_{L^6}\|u_t\|_{L^6}\|\nabla u\|_{L^2})\|u_{tt}\|_{L^6}\nonumber\\
&\le \delta\|\nabla u_{tt}\|_{L^2}^2+C\|\rho_{tt}\|_{L^2}^2+C,\\
Q_4+Q_5&\le C(\hat{\rho})\|\sqrt{\rho}u_{tt}\|_{L^2}\|u_{tt}\|_{L^6}\|\nabla u\|_{L^3}+C\|P_{tt}\|_{L^2}\|\nabla u_{tt}\|_{L^2}\nonumber\\
&\le \delta\|\nabla u_{tt}\|_{L^2}^2+C\|\sqrt{\rho}u_{tt}\|_{L^2}^2+C\|P_{tt}\|_{L^2}^2,\\
Q_6&\le C(\hat{\rho})\|u\|_{L^\infty}\|\sqrt{\rho}w_{tt}\|_{L^2}\|\nabla w_{tt}\|_{L^2}
\le \delta\|\nabla w_{tt}\|_{L^2}^2+\|\sqrt{\rho}w_{tt}\|_{L^2}^2,\\
Q_7&\le C\|(\rho u)_t\|_{L^3}(\|w_{tt}\|_{L^6}\|\nabla w_t\|_{L^2}+\|w_t\|_{L^6}\|\nabla w_{tt}\|_{L^2})\nonumber\\
&\le \delta\|\nabla u_{tt}\|_{L^2}^2+C\|\nabla u_t\|_{L^2}^3+C\|\nabla w_t\|_{L^2}^3+C,\\
Q_8&\le C(\|\rho_{tt}\|_{L^2}\|u\|_{L^\infty}\|\nabla w\|_{L^3}+\|\rho_t\|_{L^6}\|u_t\|_{L^6}\|\nabla w\|_{L^2})\|w_{tt}\|_{L^6}\nonumber\\
&\le \delta\|\nabla w_{tt}\|_{L^2}^2+C\|\rho_{tt}\|_{L^2}^2+C,\\
Q_9&\le C(\hat{\rho})\|\sqrt{\rho}u_{tt}\|_{L^2}\|w_{tt}\|_{L^6}\|\nabla w\|_{L^3}\le
\delta\|\nabla w_{tt}\|_{L^2}^2+\|\sqrt{\rho}u_{tt}\|_{L^2}^2,\\
Q_{10}&\le C\|b_{tt}\|_{L^2}\|\nabla b\|_{L^3}\|u_{tt}\|_{L^6}
+C\|b_t\|_{L^3}\|\nabla b_t\|_{L^2}\|u_{tt}\|_{L^6}+\|b\|_{L^\infty}\|\nabla b_{tt}\|_{L^2}\|u_{tt}\|_{L^2}\nonumber\\
&\le \delta(\|\nabla u_{tt}\|_{L^2}^2+\|\nabla b_{tt}\|_{L^2}^2)+C\|b_{tt}\|_{L^2}^2+C\|\nabla b_t\|_{L^2}^3,\\
Q_{11}&\le C(\|b_{tt}\|_{L^2}\|\nabla u\|_{L^3}+\|b_t\|_{L^3}\|\nabla u_t\|_{L^2}
+C\|b\|_{L^3}\|\nabla u_{tt}\|_{L^2})\|b_{tt}\|_{L^6}\nonumber\\
&\quad+C(\|u_{tt}\|_{L^6}\|\nabla b\|_{L^3}+C\|u_t\|_{L^6}\|\nabla b_t\|_{L^2}
+C\|u\|_{L^\infty}\|\nabla b_{tt}\|_{L^2})\|b_{tt}\|_{L^2}\nonumber\\
&\le \delta\|\nabla b_{tt}\|_{L^2}^2+C\|b_{tt}\|_{L^2}^2+C\|\nabla u_t\|_{L^2}^2\|\nabla b_t\|_{L^2}^2
+C\|\nabla u_t\|_{L^2}^2\|\nabla b_t\|_{L^2},\\
Q_{12}&\le C(\|\rho_{tt}\|_{L^2}\|\nabla\Phi\|_{L^3}
+\|\rho_t\|_{L^3}\|\nabla\Phi\|_{L^6})\|u_{tt}\|_{L^6}+C(\hat{\rho})\|\nabla\Phi_{tt}\|_{L^2}\|u_{tt}\|_{L^2}\nonumber\\
&\le \delta\|\nabla u_{tt}\|_{L^2}^2+C\|\rho_{tt}\|_{L^2}^2+C\|\nabla^2\Phi_t\|_{L^2}^2+C\|\nabla\Phi_{tt}\|_{L^2}^2.
\end{align*}
Putting above estimates into \eqref{3.105} and utilizing the facts
\begin{align*}
\|\nabla u_{tt}\|_{L^2}&\le C\big(\|\divv u_{tt}\|_{L^2}+\|\curl u_{tt}\|_{L^2}\big),\\
\|\nabla w_{tt}\|_{L^2}&\le C\big(\|\divv w_{tt}\|_{L^2}+\|\curl w_{tt}\|_{L^2}\big),\\
  \|\nabla b_{tt}\|_{L^2}&\le C\|\curl b_{tt}\|_{L^2},
\end{align*}
due to Lemma \ref{l24} and $(u_{tt}\cdot n,w_{tt}\cdot n,b_{tt}\cdot n)|_{\partial\Omega}=(0,0,0)$, we then get
after choosing $\delta$ suitably small that
\begin{align*}
&\frac{d}{dt}\big(\|\sqrt{\rho}u_{tt}\|_{L^2}^2
+\|\sqrt{\rho}w_{tt}\|_{L^2}^2+\|b_{tt}\|_{L^2}^2\big)
+\|\nabla u_{tt}\|_{L^2}^2+\|\nabla w_{tt}\|_{L^2}^2+\|\nabla b_{tt}\|_{L^2}^2\nonumber\\
&\le C\big(\|\sqrt{\rho}u_{tt}\|_{L^2}^2+\|b_{tt}\|_{L^2}^2+\|\rho_{tt}\|_{L^2}^2+\|P_{tt}\|_{L^2}^2
+\|\nabla u_t\|_{L^2}^3+\|\nabla w_t\|_{L^2}^3+\|\nabla b_t\|_{L^2}^3\big)\nonumber\\
&\quad+C\big(\|b_{tt}\|_{L^2}^2+\|\nabla u_t\|_{L^2}^2\|\nabla b_t\|_{L^2}^2
+\|\nabla u_t\|_{L^2}^2\|\nabla b_t\|_{L^2}+1\big),
\end{align*}
which together with \eqref{z3.74}, \eqref{z3.75}, and Gronwall's inequality yields that
\begin{align}\label{z3.107}
&\sup_{0\le t\le T}\big[\sigma^2\big(\|\sqrt{\rho}u_{tt}\|_{L^2}^2
+\|\sqrt{\rho}w_{tt}\|_{L^2}^2+\|b_{tt}\|_{L^2}^2\big)\big] \notag \\
& \quad +\int_0^T\sigma^2\big(\|\nabla u_{tt}\|_{L^2}^2+\|\nabla w_{tt}\|_{L^2}^2+\|\nabla b_{tt}\|_{L^2}^2\big)dt\le C.
\end{align}

2. It follows from  \eqref{z3.75}, \eqref{z3.90}, and \eqref{z3.107} that
\begin{align}\label{3.108}
&\sup_{0\le t\le T}\big[\sigma\big(\|\nabla^2 u_t\|_{L^2}+\|\nabla^2b_t\|_{L^2}+\|\nabla^2w_t\|_{L^2}\big)\big]\nonumber\\
&\le C\sigma\big(1+\|\sqrt{\rho}u_{tt}\|_{L^2}+\|\sqrt{\rho}w_{tt}\|_{L^2}+\|\nabla u_t\|_{L^2}+\|\nabla w_t\|_{L^2}+\|\nabla b_t\|_{L^2}\big)\le C.
\end{align}
Moreover, we deduce from \eqref{z3.36}, \eqref{z3.90}, \eqref{z3.92},
\eqref{z3.95}, \eqref{z3.96}, \eqref{z3.100}, and \eqref{z3.107} that
\begin{align}\label{3.109}
&\sigma\big(\|\nabla^2u\|_{W^{1, q}}+\|\nabla^2w\|_{W^{1, q}}\big)\nonumber\\
&\le C\sigma\big(1+\|\nabla u_t\|_{L^2}+\|\nabla w_t\|_{L^2}+\|\nabla b_t\|_{L^2}+\|\nabla(\rho\dot{u})\|_{L^q}
+\|\nabla(\rho\dot{w})\|_{L^q}+\|\nabla^2 P\|_{L^q}\big)\nonumber\\
&\le C\Big(1+\sigma\|\nabla u\|_{H^2}+\sigma\|\nabla w\|_{H^2}+\sigma^\frac12(\sigma\|\nabla u_t\|_{H^1}^2)^\frac{3(q-2)}{4q}
+\sigma^\frac12(\sigma\|\nabla w_t\|_{H^1}^2)^\frac{3(q-2)}{4q}\Big)\nonumber\\
&\le C+C\sigma^\frac12(\sigma^{-1})^\frac{3(q-2)}{4q}\le C,
\end{align}
and
\begin{align*}
\sigma\|\nabla^2b\|_{H^2}\le C\sigma\big(1+\|\nabla b_t\|_{H^1}+\|\nabla u\|_{H^2}\|\nabla b\|_{H^2}\big)\le C,
\end{align*}
which together with \eqref{z3.107} and \eqref{3.109} yields \eqref{z3.104}. \end{proof}

\section{Proof of Theorem \ref{thm1}}\label{sec4}
With all the \textit{a priori} estimates in Section 3 at hand, we are going to prove the main result of the paper.

\begin{proof}[Proof of Theorem \ref{thm1}]
By Lemma \ref{l22}, there exists a $T_*>0$ such that the system \eqref{a1}--\eqref{a6} has a unique classical solution
$(\rho, u, w, b, \Phi)$ in $\Omega\times(0, T_*]$.
By the definitions of \eqref{3.5} and \eqref{z3.2}, it is easy to check that
\begin{align*}
0<\underline{\rho}\le \rho_0\le \hat{\rho},\quad A_1(0)=0, \quad A_2(0)\le K.
\end{align*}
Therefore, there exists a $T_1\in (0, T_*]$ such that
\begin{align}\label{z4.1}
0<\underline{\rho}\le \rho_0\le 2\hat{\rho},\quad A_1(T)\le 2E_0^\frac13, \quad A_2(\sigma(T))\le 4K
\end{align}
holds for $T=T_1$.

Next, we set
\begin{align}
T^*=\sup\{T|\eqref{z4.1} ~\text{holds}\}.
\end{align}
Then $T^*\ge T_1>0$. Hence, for any $0<\tau<T\le T^*$ with $T$ finite, it follows from
Lemmas \ref{zl39}--\ref{l313} that
\begin{align}\label{z4.3}
\begin{cases}
(\rho-\rho_s, \nabla(\Phi-\Phi_s))\in C([0, T]; W^{2, q}),\\
(u, w, b)\in C([\tau, T]; H^2), \quad (\nabla u_t, \nabla w_t, \nabla b_t)\in C([\tau, T]; L^q),
\end{cases}
\end{align}
where one has taken advantage of the standard embedding
\begin{align*}
L^\infty(\tau, T; H^1)\cap H^1(\tau, T; H^{-1})\hookrightarrow C([\tau, T]; L^q), \quad {\rm for~any}~q\in (3, 6).
\end{align*}
Moreover, one has
\begin{align}
&\int_\tau^T\Big|\Big(\int\rho|u_t|^2dx\Big)_t+\Big(\int\rho|w_t|^2dx\Big)_t\Big|dt\nonumber\\
&\le \int_\tau^T\big(\|\rho_t|u_t|^2\|_{L^1}+\|\rho_t|w_t|^2\|_{L^1}+2\|\rho u_t\cdot u_{tt}\|_{L^1}
+2\|\rho w_t\cdot w_{tt}\|_{L^1}\big)dt\nonumber\\
&\le C\int_\tau^T\big(\|\sqrt{\rho}u_t\|_{L^2}\|\nabla u\|_{L^\infty}+\|\sqrt{\rho}w_t\|_{L^2}\|\nabla w\|_{L^\infty}
+\|u\|_{L^6}\|\nabla\rho\|_{L^2}\|u_t\|_{L^6}^2\big)dt\nonumber\\
&\quad+C\int_\tau^T\big(\|u\|_{L^6}\|\nabla\rho\|_{L^2}\|w_t\|_{L^6}^2
+\|\sqrt{\rho}u_{tt}\|_{L^2}^2+\|\sqrt{\rho}w_{tt}\|_{L^2}^2
+\|\sqrt{\rho}u_t\|_{L^2}^2+\|\sqrt{\rho}w_t\|_{L^2}^2\big)dt\nonumber\\
&\le C,
\end{align}
which together with \eqref{z4.3} leads to
\begin{align}\label{z4.5}
\sqrt{\rho}u_t, \sqrt{\rho}w_t, \sqrt{\rho}\dot{u}, \sqrt{\rho}\dot{w}\in C([\tau, T]; L^2).
\end{align}

Finally, we claim that
\begin{align}\label{z4.6}
T^*=\infty.
\end{align}
Otherwise, $T^*<\infty$. Then, by Proposition \ref{p31}, it holds that
\begin{align}
0<\rho\le \frac74\hat{\rho}, \quad A_1(T^*)\le E_0^\frac13, \quad A_2(\sigma(T))\le 3K.
\end{align}
It follows from Lemmas \ref{l312}--\ref{l313} and \eqref{z4.5} that $(\rho(x, T^*), u(x, T^*), w(x, T^*), b(x, T^*))$
satisfies the initial condition \eqref{a10}--\eqref{a12} with $g_1(x)\triangleq\sqrt{\rho}\dot{u}(x, T^*)$,
$g_2(x)\triangleq\sqrt{\rho}\dot{w}(x, T^*)$,
$x\in \Omega$. Thus, Lemma \ref{l22} implies that there exists some $T^{**}>T^*$ such that \eqref{z4.1} holds for $T=T^{**}$,
which contradicts the definition of $T^*$. As a result, \eqref{z4.6} follows. By Lemma \ref{l22} and Lemmas \ref{l312}--\ref{l313}, it indicates that $(\rho, u, w, b, \Phi)$ is in fact the unique classical solution defined in
$\Omega\times (0, T]$ for any $0<T<T^*=\infty$. The proof of Theorem \ref{thm1} is finished.
\end{proof}

\section*{Acknowledgments}
The authors would like to express their gratitude to the reviewers for careful reading and helpful suggestions which led to an improvement of the original manuscript.

\end{document}